\documentclass[11pt]{amsart}

\usepackage{ifxetex,ifluatex}
\usepackage{cite}
\if\ifxetex T\else\ifluatex T\else F\fi\fi T%
  \usepackage{fontspec}
\else
  \usepackage[T1]{fontenc}
  \usepackage[utf8]{inputenc}
  \usepackage{lmodern}
\fi

\usepackage[margin=3cm]{geometry}
\usepackage{enumitem}
\usepackage{amsaddr}
\usepackage{amsmath}
\usepackage{amssymb,stmaryrd}
\usepackage{amscd}
\usepackage{latexsym}
\usepackage{amsthm}
\usepackage{mathrsfs}
\usepackage{tikz}
\usetikzlibrary{decorations.pathreplacing,patterns, arrows, calc}
\usetikzlibrary{knots}
\usepackage[boxsize=1.2em,centerboxes]{ytableau}
\usepackage{xcolor}

\usepackage{hyperref}

\newtheorem{thm}{Theorem}[section]
\newtheorem*{thm*}{Theorem}
\newtheorem{cor}[thm]{Corollary}
\newtheorem*{cor*}{Corollary}
\newtheorem{lem}[thm]{Lemma}
\newtheorem*{lem*}{Lemma}
\newtheorem{prop}[thm]{Proposition}
\newtheorem*{prop*}{Proposition}

\theoremstyle{definition}
\newtheorem{definition}[thm]{Definition}
\newtheorem*{definition*}{Definition}

\newtheorem*{conjecture*}{Conjecture}
\newtheorem*{condition*}{Condition}
\newtheorem*{assumption*}{Assumption}

\theoremstyle{remark}
\newtheorem{remark}[thm]{Remark}
\newtheorem*{remark*}{Remark}
\newtheorem{example}[thm]{Example}

\newtheorem*{problem*}{Problem}

\newenvironment{skproof}{\textit{Sketch of Proof:}}{\hfill$\square$}

\numberwithin{equation}{section}

\newcommand{\BR}{\mathbb R}
\newcommand{\BC}{\mathbb C}

\newcommand{\BZ}{\mathbb Z}

\newcommand{\CA}{\mathcal A}

\newcommand{\CC}{\mathcal C}

\newcommand{\CF}{\mathcal F}

\newcommand{\CM}{\mathcal M}
\newcommand{\CT}{\mathcal T}
\newcommand{\CL}{\mathcal L}

\newcommand{\CW}{\mathcal W}

\newcommand{\id}{\mathrm{id}}

\newcommand{\DPT}{\textnormal{DPT}}
\DeclareMathOperator{\End}{End}
\DeclareMathOperator{\Hom}{Hom}
\DeclareMathOperator{\Id}{Id}
\DeclareMathOperator{\slice}{slice}
\DeclareMathOperator{\Irr}{Irr}
\DeclareMathOperator{\Sk}{Sk}

\makeatletter
\newcommand*{\@old@slash}{}\let\@old@slash\slash
\def\slash{\relax\ifmmode\delimiter"502F30E\mathopen{}\else\@old@slash\fi}
\makeatother

\title{Type $A$ DAHA and Doubly Periodic Tableaux}
\author{Léa Bittmann, Alex Chandler, Anton Mellit, Chiara Novarini}

\definecolor{forest}{rgb}{0.03, 0.47, 0.19}

\definecolor{org}{rgb}{1.0, 0.6, 0.0}

\begin{document}
\maketitle

\begin{abstract}
Analogously to the construction of Suzuki and Vazirani, we construct representations of the $GL_m$-type Double Affine Hecke Algebra at roots of unity. These representations are graded and the weight spaces for the $X$-variables are parametrized by the combinatorial objects we call doubly periodic tableaux. We show that our representations exhaust all graded $X$-semisimple representations, and the direct sum of all our representations is faithful. Analogously to the construction of Jordan and Vazirani of rectangular DAHA representations, we show that our representations can be interpreted in terms of ribbon fusion categories associated to $U_q(\mathfrak{gl}_N)$ at roots of unity. Combining the ribbon structure with faithfulness we deduce a conjecture of Morton and Samuelson about realization of DAHA as a skein algebra of the torus with base string modulo certain local relations.
\end{abstract}



















\section{Introduction}

A large part of the inspiration for the present paper comes from the work of Suzuki and Vazirani \cite{Suzuki2005Tableaux} on the classification of certain representations of the \emph{Double Affine Hecke Algebra} (DAHA) of type $A$, away from roots of unity. The representations considered there are \emph{$X$-semisimple} representations, which are semisimple for the action of the embedded Affine Hecke Algebra (AHA) (analogous to the \emph{calibrated} representations of the AHA studied for example in \cite{Ram2003Affine}). The $X$-semisimple representations of the DAHA were first classified by Cherednik in \cite{Cherednik2003Fourier} (see also \cite{CherednikBook}).
The theory of these $X$-semisimple representations is known to be related to that of quantum groups of type $A$ by the work of Jordan and Vazirani \cite{Jordan2018Schur}, where a certain space of intertwiners is used to construct the so-called rectangular representation of the DAHA (a particular example of $X$-semisimple representations).
The main goal of the present work is to extend these results to representations of the DAHA at roots of unity. 

The DAHA representations in \cite{Jordan2018Schur} have a basis indexed by periodic tableaux on an infinite horizontal strip, and the DAHA representations in \cite{Suzuki2005Tableaux} have a basis indexed by periodic tableaux on an infinite periodic skew shape. 
In the present paper, we require periodicity in two directions, thus we introduce the notion of doubly periodic tableaux. 
A large portion of this paper is dedicated to studying the combinatorics of doubly periodic tableaux, which we find to be interesting in their own right. 
We construct DAHA representations in the roots of unity case by defining an action of the DAHA on doubly periodic tableaux. 

The representations we consider are no longer $X$-semisimple in the usual sense, but they can be seen as \emph{graded} $X$-semisimple representations. Moreover, as in the non root of unity case, the representations constructed using tableaux provide a classification of all graded $X$-semisimple representations.
It is interesting to notice that the graded pieces of our modules turn out to be $X$-semisimple representations of a subalgebra of the DAHA called the \emph{small DAHA}, which is
in fact the double affine Hecke algebra associated to the (non-extended) affine Weyl group. We remark that this is indeed a key point in our classification result, as most of the work in \cite{Suzuki2005Tableaux} still applies to these graded pieces.

In the attempt of building a parallel theory for graded $X$-semsimple representations at roots of unity to that for $X$-semisimple representations away from roots of unity, we give a second construction of our modules using quantum groups. Interestingly, quantum groups at roots of unity have a very rich structure. We focus here on the \emph{fusion category}, defined as the quotient of the category of finite-dimensional representations by negligible modules (see for example \cite{Kirillov1996Inner} for a review), and on the associated \emph{fusion ring}, for which Andersen gave a combinatorial description in \cite{Andersen2014Fusion}. We show that the modules coming from doubly periodic tableaux are isomorphic to modules obtained from some spaces of intertwiners in the fusion category, obtaining a result resembling the one proved in \cite{Jordan2018Schur}.

One of the main ingredients we use is the fact that the fusion category has a \emph{ribbon category} structure (as seen in \cite{Kirillov1996Inner}). This structure allows us to define an action of the DAHA on ribbon graphs, which is isomorphic to the modules coming from doubly periodic tableaux. As an application, we also use this crucial ribbon structure to prove a conjecture of Morton-Samuelson giving an interpretation of the DAHA as a tangle algebra. The proof will rely on a faithfulness result for a module built out of the graded $X$-semisimple representations coming from doubly periodic tableaux.



In Section~\ref{sect_DPT} we introduce DPTs, and study their many interesting combinatorial properties. In particular, in Section~\ref{sect_count}, we obtain some formulas counting the number of DPTs.
In Section~\ref{sectDAHA} we build irreducible representations of the DAHA using those doubly periodic tableaux, and show that all \emph{irreducible graded} $X$-semisimple representations of the DAHA can be constructed this way.
In Section~\ref{dpt_intertwiners} we relate DPTs to intertwining spaces for the type $A$ quantum group at a root of unity, using quantum Shur-Weyl duality. We recover the previously defined action of the AHA on DPT using the ribbon category structure. 
In Section~\ref{sect_ribbon} we study some more precise ribbon calculus theory, in order to extend the AHA action on intertwining spaces to a DAHA action. In the end we prove a conjecture of Morton and Samuelson about realization of DAHA as a skein algebra of the torus with base string modulo certain local relations.

The authors are grateful to Peter Samuelson, David Jordan and Monica Vazirani for useful discussions. The work was supported by the Austrian Science Fund (FWF) Project P-31705

\section{Doubly Periodic Tableaux}\label{sect_DPT}

The main combinatorial objects of interest in this paper are functions $\sigma: S\to\BZ$ where $S\subseteq\BZ^2$. We refer to such functions as \textit{tableaux} (singularly \textit{tableau}).
For our coordinate system, we use \textit{reading directions}. That is, $x$ increases to the right, and $y$ increases downward.
We embed $\BZ^2\subseteq\BR^2$ and  depict a tableau $\sigma$ by displaying the value $\sigma(x,y)$ in the center of the square cell with top left corner $(x,y)$, that is, at $(x+\frac{1}{2},y+\frac{1}{2})$.
We use this identification of lattice coordinates with square cells inherently from now on. 
The tableau $\sigma$ will be called \textit{standard} if the values of $\sigma$ strictly increase in both the $x$ and $y$ direction. 
See Figure \ref{labeling_convention} for an example. 

\subsection{Doubly Periodic Tableaux}\label{sectDPTlattice}
Fix integers $K,N,a,b$, let $m=a N - b K$, and assume that $m>0$. 
\begin{definition}
\label{dptdef}
  A \textit{doubly periodic tableau} with respect to $(K,N,a,b)$ is a surjective function $\sigma:\BZ^2\to\BZ$ satisfying
  \begin{enumerate}
    \item $\sigma(x+K,y-N)=\sigma(x,y)$,
    \item $\sigma(x+a,y-b)=\sigma(x,y)+m$,
  \end{enumerate}
  for all $x,y\in\BZ$. 
  
  Let $\DPT(K,N,a,b)$ denote the set of all \textbf{standard} doubly period tableaux with respect to $(K,N,a,b)$.
\end{definition}


\begin{remark}\label{mod_nk}
  For any $r\in\BZ$, it is immediate from the definition that we have \[\DPT(K,N,a,b)=\DPT(K,N,a+r K,b+r N).\]
 One may consider to what extent $(a,b)$ is determined by $K,N,m$ modulo $(K,N)$. 
Let $g=\gcd(K,N)$ and let $a,b,a',b'\in\BZ$ with $aN-bK=m=a'N-b'K$. 
It follows that 
\begin{equation}\label{a_b_determined}
(a^\prime,b^\prime)=(a,b)+s\left(\frac{K}{g},\frac{N}{g}\right)
\end{equation}
for some $s\in\BZ$. 
In the case that $g=1$, one finds that $K,N,m$ determines $(a,b)$ modulo $(K,N)$, but in general there are $g=\gcd(K,N)$ distinct values for $(a,b)$ modulo $(K,N)$, those corresponding to $s=0,\dots,g-1$ in Equation \ref{a_b_determined}.
\end{remark}

\begin{example}
  Suppose that $K,N>0$. Consider the function $\sigma(x,y)=rx+sy+c$ for some integers $r,s,c$. The two conditions in Definition \ref{dptdef} require $r K - s N=0$ and $r a - s b = m$, so we must have $s = K, r = N$:
  \begin{equation}\label{eq:linear dpt}
    \sigma(x,y) = N x + K y + c.
  \end{equation}
  The function \eqref{eq:linear dpt} is surjective if and only if $\gcd(K,N)=1$. 

  More generally, for arbitrary $K,N$ with $g=\gcd(N,K)$ we perturb \eqref{eq:linear dpt} as follows:
  \begin{equation}
  \label{eq:linear dpt pert}
    \sigma(x,y) = N x + K y + c + \left\lfloor\frac{b x+a y \mod m}{m/g}\right\rfloor.
  \end{equation}
  The correction term is doubly periodic and takes values from $0,\ldots,g-1$. Surjectivity can be shown as follows. Let $i\in \BZ$. Begin by choosing $x, y$ so that $0 \leq i-(N x + K y + c) < g$. Consider the effect of iteratively replacing $(x,y)$ by $(x+K/g, y-N/g)$ in \eqref{eq:linear dpt pert}. After each replacement, the term $N x + Ky + c$ does not change, while the correction term cycles through all the possible values $0,1,\ldots,g-1$. Pick $0\leq r<g$ such that the correction term is equal to $i-(Nx+Ky+c)$. It follows that $\sigma(x+rK/g,y-rN/g)=i$.
  \label{linearex}
\end{example}

\begin{lem}\label{lemmaKNpositive}
  The set $\DPT(K,N,a,b)$ is nonempty if and only if $K,N>0$. 
\end{lem}

\begin{proof}
  Suppose that $\sigma\in\DPT(K,N,a,b)$. Then we have
  \begin{align}
    \sigma(x+m,y) &= \sigma(x+aN-bK, y-bN+bN) = \sigma(x,y) + N m,\label{eq_sigma_x+m_y}\\
    \sigma(x,y+m) &= \sigma(x-a K+K a, y+a N-b K) = \sigma(x,y) + K m.\label{eq_sigma_x_y+m}
  \end{align}
  Since $m>0$, standardness of $\sigma$ imply $Km,Nm>0$ and therefore $K,N>0$. 
Conversely, Formula \eqref{eq:linear dpt pert} in Example \ref{linearex} gives an element of $\DPT(K,N,a,b)$ whenever $K,N>0$. 
\end{proof}

\begin{example}\label{exampleLineTableau}
Let us describe all $\sigma\in\DPT(K,N,a,b)$ such that $\sigma(x,0)=x$ for $0\leq x<m$.
Assume $\sigma(0,1)=\alpha m + \beta>0$, where $\alpha\geq 0$ and $0\leq \beta<m$, $(\alpha,\beta)\neq(0,0)$. 
Since $\sigma(0,1)=\alpha m + \beta$ there exists $s \in \BZ$ satisfying $(0,1)=(\beta,0)+\alpha(a,-b)+s(K,-N)$. Therefore we have $\sigma(-\beta,1)=\alpha m$ and more generally $\sigma(-r\beta,r)=r\alpha m$ for all $r\in \BZ$. In particular, setting $r=N$ we have $\sigma(-N\beta,N)=N \alpha m$, which implies $\sigma(-N\beta-\alpha m,N)=0$ and therefore we have $K=N\beta + \alpha m$. Values divisible by $m$ can only appear at positions $\sigma(-r\beta+tm,r)=r \alpha m + t N m$.
So surjectivity implies that in the sequence $\{r\alpha\}_{r\in\BZ}$ all residues modulo $N$ must appear, so we need to have $\gcd(\alpha,N)=1$.

Choose $u,v \in \BZ$ such that $u N - v \alpha = 1$. Then we have
  \[
    \sigma(-v\beta, v) = v \alpha m = (u N-1) m,\qquad \sigma(-v\beta - u m, v) = -m.
  \]
Therefore $(a,b)=(v \beta + u m, v) \mod (K,N)$.

Conversely, suppose $m, N, \alpha, \beta\in\BZ$ are given satisfying $m,N>0$, $\alpha\geq 0$, $\gcd(\alpha,N)=1$, $0\leq\beta<m$, $(\alpha,\beta)\neq(0,0)$. Choose $u,v\in\BZ$ satisfying $u N - v \alpha = 1$. Set
\[
K=N\beta + \alpha m,\qquad a=v \beta + u m,\qquad b=v.
\]
In the matrix form, we have 
  \[
    \begin{pmatrix} a & K\\ b & N \end{pmatrix} = \begin{pmatrix} m & \beta \\ 0 & 1\end{pmatrix} \begin{pmatrix} u & \alpha \\ v & N \end{pmatrix}.
  \]
Note that a different choice of $u,v$ leads to equivalent vector $(a,b)$ modulo $(K,N)$.

Set $\sigma(-r\beta+i,r)=r \alpha m+i$ for all $r\in\BZ$, $0\leq i<m$ and extend $\sigma$ to the whole of $\BZ^2$ using \eqref{eq_sigma_x+m_y}. Resulting tableau satisfies
\begin{equation}\label{eq:periodicity}
\sigma(x-\beta,y+1)=\sigma(x,y) + \alpha m,\qquad \sigma(x+m,y)=\sigma(x,y) + mN.
\end{equation}
Using $(K,N)=N(\beta,1) + \alpha(m,0)$, $(a,b)=v(\beta,1)+u(m,0)$, we obtain
\[
\sigma(x+K,y-N)-\sigma(x,y)=-N\alpha m + \alpha mN = 0,\quad \sigma(x+a,y-b)-\sigma(x,y)=-v\alpha m + umN=m.
\]
Finally, to see that $\sigma$ is standard we use
\[
\sigma(i,1)=\alpha m + \beta+i>i=\sigma(i,0)\quad(0\leq i <m-\beta),
\]
\[
\sigma(i,1)=\alpha m + \beta+i-m+Nm>i=\sigma(i,0)\quad(m-\beta\leq i <m),
\]
and periodicity conditions \eqref{eq:periodicity}.

\end{example}

\begin{figure}
  \begin{tikzpicture}[yscale=-1,scale=.5]
      \draw[->,thick] (-6,-3.5)--(-4,-3.5);
      \node at (-5,-4){$x$};
      \draw[->,thick] (-6.5,-3)--(-6.5,-1);
      \node at (-7,-2){$y$};
      \draw[step=1cm,black,thin] (-6,-3) grid (7,7);
      \foreach \x in {-6,...,6} \foreach \y in {-3,...,6} {\pgfmathsetmacro\result{2*\x+3*\y} \node[scale=.7] at (\x+0.5,\y+0.5) {\pgfmathprintnumber{\result}}; }
      \draw[red, thick] (0,0) circle (1mm);
      \foreach \x/\y in {-6/6,-5/6,-5/5,-4/5,-4/4,-3/4,-2/4,-2/3,-1/3,-1/2,0/2,1/2,1/1,2/1,2/0,3/0,4/0,4/-1,5/-1,5/-2,6/-2} \draw[fill=red, fill opacity=0.2] (\x,\y) rectangle (\x+1,\y-1);
      \foreach \x/\y in {-7/6,-6/6,-5/6,-5/5,-4/5,-4/4,-3/4,-2/4,-2/3,-1/3,-1/2,0/2,1/2,1/1,2/1,2/0,3/0,4/0,4/-1,5/-1,5/-2} \draw[fill=blue, fill opacity=0.2] (\x+1,\y+1) rectangle (\x+2,\y);
      \foreach \x/\y in {-1/2,0/2,1/2,1/1,2/1} \draw[fill=red, fill opacity=0.4] (\x,\y) rectangle (\x+1,\y-1);
      \foreach \x/\y in {3/1,4/1,5/1,5/0,6/0} \draw[fill=blue, fill opacity=0.4] (\x,\y) rectangle (\x+1,\y-1);
      \draw[very thick,red](-6,5)--(-5,5)--(-5,4)--(-4,4)--(-4,3)--(-2,3)--(-2,2)--(-1,2)--(-1,1)--(1,1)--(1,0)--(2,0)--(2,-1)--(4,-1)--(4,-2)--(5,-2)--(5,-3)--(7,-3);
      \draw[very thick,blue](-6,7)--(-6,6)--(-4,6)--(-4,5)--(-3,5)--(-3,4)--(-2,4)--(-5+4,5-1)--(-5+4,4-1)--(-4+4,4-1)--(-4+4,3-1)--(-2+4,3-1)--(-2+4,2-1)--(-1+4,2-1)--(-1+4,1-1)--(1+4,1-1)--(1+4,0-1)--(2+4,0-1)--(2+4,-1-1)--(7,-2);
      \draw[very thick,forest](-3,7)--(-3,6)--(-2,6)--(-6+4,6-1)--(-4+4,6-1)--(-4+4,5-1)--(-3+4,5-1)--(-3+4,4-1)--(-2+4,4-1)--(-5+4+4,5-1-1)--(-5+4+4,4-1-1)--(-4+4+4,4-1-1)--(-4+4+4,3-1-1)--(6,1)--(6,0)--(7,0);
  \end{tikzpicture} 
  \caption{Shown above is the DPT $\sigma(x,y)=2x+3y$. We consider the case $(K,N,a,b)=(2,3,4,1)$ so $m=5$. The lattice paths $\CL(\sigma),\CL(\sigma)[4,1]$, and $\CL(\sigma)[8,2]$ are shown (in red, blue, and green respectively), and these bound the regions $\Delta^\prime$ (shown in red) and $\Delta^\prime[4,1]$ (shown in blue). The fundamental domain $\Delta$ is shown in darker red, and $\Delta[4,1]$ is shown in darker blue. The lattice point $(0,0)$ is circled in red. 
  }
  \label{labeling_convention}
\end{figure}
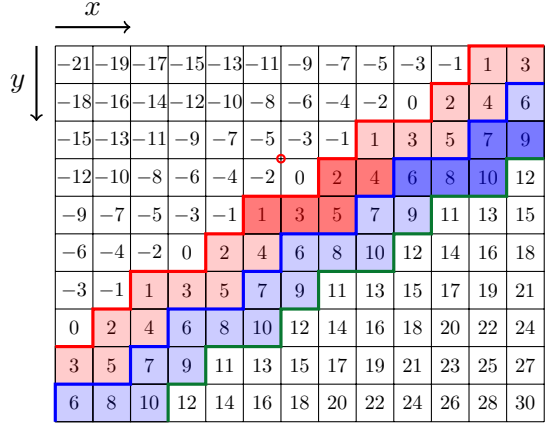

\begin{definition}\label{fundamental_domain_def}
  Let $\sigma\in\DPT(K,N,a,b)$. The \textit{extended fundamental domain} of $\sigma$ is
    \[\Delta^\prime(\sigma)=\{(x,y)\in\BZ^2 \mid  1\leq \sigma(x,y)\leq m\}.\]
  The \textit{fundamental domain} of $\sigma$ is 
  \[\Delta(\sigma)=\{(x,y)\in\Delta^\prime(\sigma) \mid 0\leq y<N\}.\]
\end{definition}

\begin{remark}
  If $\sigma$ is clear from context we will often write $\Delta=\Delta(\sigma)$ and $\Delta'=\Delta'(\sigma)$ for brevity. 
  Given $\sigma\in\DPT(K,N,a,b)$, the restriction $\sigma|_{\Delta'}$ is a standard cylindric tableau in the sense of Postnikov
\cite{Postnikov2005Affine} and Gessel-Krattenthaler \cite{Gessel1997Cylindric}.
\end{remark}

It is useful to consider the natural symmetries on the set of doubly periodic tableaux. There is an action of the abelian group $\CA$ with presentation
\begin{equation}
  \CA=\langle D,L,\pi \ | \ D^KL^{-N}=1,D^{-a}L^b=\pi^m\rangle
\end{equation}
on $\DPT(K,N,a,b)$. 
The generators act on $\sigma\in\DPT(K,N,a,b)$ as follows. For any $(x,y)\in\BZ^2$,
\begin{align}
D\sigma(x,y)&=\sigma(x-1,y)&D\sigma&\in\DPT(K,N,a,b)\label{D_def}\\
L\sigma(x,y)&=\sigma(x,y-1)&L\sigma&\in\DPT(K,N,a,b)\label{L_def}\\
\pi\sigma(x,y)&=\sigma(x,y)+1&\pi\sigma&\in\DPT(K,N,a,b)\label{del_def}
\end{align}
The elements $D^i,L^i,\pi^i$ act bijectively on each $\DPT(K,N,a,b)$ for each $i\in\BZ$ with inverses $D^{-i},L^{-i},\pi^{-i}$. 
Other symmetries exist as well. For example, the map $\sigma(x,y)\mapsto -\sigma(-x,-y)$ is an involution on $\DPT(K,N,a,b)$, and the map
$\sigma(x,y)\mapsto\sigma(y,x)$ gives a bijection
\begin{equation}\label{phi}
  \DPT(K,N,a,b)\xrightarrow{\sim} \DPT(N,K,-b,-a).
\end{equation}


\subsection{DPTs and Periodic Lattice Paths}

\begin{definition}
A \textit{lattice path} $\CL$ in $\BZ^2$ is a sequence $(\CL_i)_{i\in\BZ}$ where $\CL_i\in\BZ^2$, and for each $i\in\BZ$, $\CL_{i+1}=\CL_i+(0,-1)$ or $\CL_{i+1}=\CL_i+(1,0)$.
The \textit{image} of $\CL$ is $\{\CL_i\}_{i\in\BZ}$.
Given a lattice path $\CL$ and integers $a,b$, form the $(a,b)$-shifted path $\CL[a,b]$ where $\CL[a,b]_i=\CL_i+(a,-b)$.
  A lattice path $\CL$ is \textit{$(K,N)$-periodic} if $\CL$ and $\CL[K,N]$ have the same image.
  \end{definition}

  Given a $(K,N)$-periodic lattice path, for each $y\in\BZ$, we define $\CL(y)$, to be the unique value of $x\in\BZ$ such that both $(x,y)$ and $(x,y+1)$ are in the image of $\CL$ (i.e. the $x$ value of the vertical step from $y+1$ to $y$ on $\CL$).
  
There is a partial ordering $\leq$ on the set of $\BZ^2$ lattice paths defined by $\CL\leq\CM$ when $\CL(y)\leq\CM(y)$ for all $y\in\BZ$.

\begin{remark}\label{lattice_path_equivalent}
The collection of all values $\CL(y)$ for $y\in\BZ$ completely determine $\CL$. In fact, an equivalent way of defining $(K,N)$-periodic lattice paths is as weakly decreasing functions $\CL:\BZ\to\BZ$ such that \[\CL(y+N)=\CL(y)-K, \quad \forall y\in \BZ.\]

These are called \textit{cylindric partitions} in \cite{neyman2014cylindric}. 
Another equivalent description is to work with ``infinite diagrams'', i.e. non-empty lower order ideals of $\BZ^2$ with the order defined by $(x,y)\leq (x',y')$ if $x\leq x'$ and $y\leq y'$. The partial ordering $\CL\leq\CM$ defined above coincides with the reverse inclusion ordering of lower order ideals. We may use whichever description is most convenient for the matter at hand. 
\end{remark}

\begin{lem}\label{shift_lemma}
  Let $\CL$ be a lattice path in $\BZ^2$, $(a,b)\in\BZ^2$ and let $\CM=\CL[a,b]$. Then for all $y\in\BZ$,
  \[\CM(y)=\CL(y+b)+a.\]
\end{lem}
\begin{proof}
Given $y\in\BZ$, choose $i\in\BZ$ such that $\CL_i=(x,y+1)$ and $\CL_{i+1}=(x,y)$ for some $x\in\BZ$. Then $\CM_i=(x+a,y-b+1)$ and $\CM_{i+1}=(x+a,y-b)$ so $\CL(y)=x$ and $\CM(y-b)=x+a$.
\end{proof}

\begin{remark}\label{leq_equiv}
One immediate consequence of Lemma \ref{shift_lemma}, is that $\CL\leq\CL[a,b]$ is equivalent to $\CL(y+b)\geq\CL(y)-a$ for all $y\in\BZ$. Iterating this, we have that $\CL\leq\CL[a,b]$ if and only if the function
\begin{equation}\label{abr-pos-shifting}
  r\mapsto \CL(y+rb)+ ra,
\end{equation}
is increasing on $\BZ$.
\end{remark}


\begin{lem}\label{shift_suff}
  Let $\CL$ be a $(K,N)$-periodic lattice path and suppose that $a\geq K$ and $b\leq N$. Then $\CL<\CL[a,b]$.
\end{lem}
\begin{proof}
We have 
\begin{align*}
\CL(y+b)&\geq\CL(y+N)\\
&= \CL(y)-K\\
&\geq\CL(y)-a\qedhere
\end{align*}
\end{proof}

\begin{definition}\label{LP_fund}
  Let $\CL$ be a $(K,N)$-periodic lattice path such that $\CL<\CL[a,b]$. Let
  \begin{align}
    \Delta'(\CL)&:=\{(x,y)\in\BZ^2 \ | \ \CL(y)\leq x < \CL[a,b](y)\}\\
    \Delta(\CL)&:=\{(x,y)\in\Delta'(\CL) \ | \ 0\leq y < N\}.
  \end{align}
\end{definition}

The following gives a way to obtain periodic lattice paths from doubly periodic tableaux.

\begin{lem}\label{skew_lemma}
  Let $\sigma\in\DPT(K,N,a,b)$. Then 
\[y\mapsto \CL(y)=\min\{x\in\BZ \mid \sigma(x,y)\geq 1\}\]
defines a $(K,N)$-periodic lattice path $\CL=\CL(\sigma)$ such that $\CL<\CL[a,b]$ and $\Delta'(\sigma)=\Delta'(\CL)$.
\end{lem}
\begin{proof}
Define $\CM=\CL[a,b]$, and notice that 
  \[\CM(y)=\CL(y+b)+a=\min\{x \ | \ \sigma(x,y)\geq 1+m\},\]
so we have $\CL<\CM$.
  Now, by definition $\Delta'(\CL)=\{(x,y)\ | \ \CL(y)\leq x\leq\CM(y)\}$ consists of those $(x,y)$ such that $1\leq \sigma(x,y)\leq m$. That is, $\Delta'(\CL)=\Delta^\prime(\sigma)$. 
\end{proof}

\begin{example}
  Continuing from Example \ref{linearex}, consider $\sigma(x,y)=Nx+Ky$.
  Given $y\in\BZ$, we have that $\CL(y)$ is the minimum $x$ value such that $Nx+Ky\geq 1$, so $\CL(y)=\lceil{\frac{1-Ky}{N}}\rceil$. Define $\CM=\CL[a,b]$.
  By Lemma \ref{shift_lemma}, we have $\CM(y)=\CL(y+b)+a=\lceil{\frac{1-Ky-Kb}{N}}\rceil+a$.
  Consider $(K,N,a,b)=(3,2,4,1)$ so that $m=aN-bK=5$.
  We have $\CL(y)=\lceil{\frac{1-3y}{2}}\rceil$ and $\CM(y)=\lceil{\frac{1-3y-3}{2}}\rceil+4$. Notice that $\lceil{\frac{1-3y-3}{2}}\rceil+4>\lceil{\frac{1-3y-8}{2}}\rceil+4=\CL(y)$ so that $\CL<\CM$.
  
  As for the fundamental domain, we have $\CL(0)=1, \CM(0)=3$, giving cells $(1,0),(2,0)$ in $\Delta(\sigma)$ and $\CL(1)=-1, \CM(1)=2$ giving cells $(-1,1), (0,1), (1,1)$ in $\Delta(\sigma)$ (see Figure \ref{labeling_convention}). In this example, there are $8$ standard fillings of $\Delta$, $7$ of which extend periodically to standard fillings of $\Delta^\prime$.
\end{example}


\begin{lem}\label{lem_latt_Delta}
Let $\CL$ be a $(K,N)$-periodic lattice path such that $\CL<\CL[a,b]$. For any $(x,y)\in\BZ^2$ there exists a unique $r\in\BZ$ such that $(x-ra,y+rb) \in \Delta'(\CL)$.
\end{lem}

\begin{proof}
From Fix $(x,y)\in\BZ^2$. From Remark~\ref{leq_equiv}, the map
\[r \mapsto \CL(y+rb) + ra\]
is strictly increasing on $\BZ$. Thus there exists a unique $r\in\BZ$ such that
\begin{align*}
& \CL(y + rb) +ra \leq x < \CL(y + (r+1)b) + (r+1)a,\\
\Leftrightarrow \quad & (x-ra,y+rb) \in \Delta'(\CL).\qedhere
\end{align*}
\end{proof}

More precisely, we also have the following.

\begin{lem}\label{fundamental_lemma}
Let $\CL$ be a $(K,N)$-periodic lattice path such that $\CL<\CL[a,b]$. For any $(x,y)\in\BZ^2$ there exists a unique $(x^\prime,y^\prime,s,r)\in\BZ^4$ such that $(x^\prime,y^\prime)\in \Delta(\CL)$ and $(x,y)=(x^\prime,y^\prime)+s(K,-N)+r(a,-b)$.
\end{lem}
\begin{proof}
 Using Lemma~\ref{lem_latt_Delta}, given $(x,y)\in \BZ^2$ there is a unique $r\in\BZ$ such that $(x'',y''):=(x,y)-r(a,-b)\in\Delta'(\CL)$. Furthermore, there exists a unique $s\in\BZ$ such that $0\leq y''+sN<N$ and therefore $(x',y'):=(x'',y'')-s(K,-N)\in\Delta(\CL)$. Thus we have found unique $(x',y',s,r)$ such that $(x,y)=(x',y')+s(K,-N)+r(a,-b)$ and $(x',y')\in\Delta(\CL)$.
\end{proof}

\begin{lem}\label{fund_lem_iff}
  Let $\sigma\in\DPT(K,N,a,b)$, $r\in\BZ$, and $(x,y),(x',y')\in\BZ^2$. Then $\sigma(x,y)=\sigma(x',y')+rm$ if and only if $(x,y)=(x',y')+s(K,-N)+r(a,-b)$ for some $s\in\BZ$. 
\end{lem}
\begin{proof}
  The backwards direction follows immediately from the definition of $\DPT(K,N,a,b)$. For the forward direction, assume without loss of generality that $(x,y)\in\Delta(\sigma)$. Then $(x',y')+r(a,-b)\in\Delta'(\sigma)$ since $\sigma((x',y')+r(a,-b))=\sigma(x',y')+rm=\sigma(x,y)$. Now by Lemma \ref{fundamental_lemma} there exists $s\in\BZ$ such that $(x',y')+r(a,-b)+s(K,-N)\in\Delta(\sigma)$. Since $\sigma$ assigns distinct values to all cells in $\Delta(\sigma)$, the result follows. 
\end{proof}



\begin{lem}\label{lem_Delta_m_cells}
  Let $\CL$ be a $(K,N)$-periodic lattice path such that $\CL<\CL[a,b]$. Then $\Delta(\CL)$ contains exactly $m$ cells.
\end{lem}
\begin{proof}
  Since $\CL$ is $(K,N)$ periodic, the portion of the lattice path from $y=0$ to $y=N-1$ consists of $K$ steps in the $(1,0)$ direction and $N$ steps in the $(0,-1)$ direction. Shifting $\CL$ to the right by one unit contributes one cell for each $(0,-1)$ step, and shifting $\CL$ up by one unit subtracts one cell for each $(1,0)$ step. Therefore shifting right by $a$ units and up by $b$ units contributes $m=aN-bK$ cells to $\Delta$. 
\end{proof} 

Given a $(K,N)$-periodic lattice path $\CL$ it follows from Lemma \ref{fundamental_lemma}, that $\BZ^2$ is partitioned by copies of $\Delta:=\Delta(\CL)$ indexed by pairs $(r,s)\in\BZ^2$. More precisely, given $s,r\in\BZ$ we define
\begin{equation}
  \Delta_{s,r}:=\{(x,y)+s(K,-N)+r(a,-b)\ | \ (x,y)\in\Delta\}.
\end{equation} 
One has
\[\BZ^2 = \bigsqcup_{(s,r)\in\BZ^2}\Delta_{s,r}. \]

Given a $(K,N)$-periodic lattice path $\CL$, each filling $\sigma$ of $\Delta(\CL)$ by numbers $1,\dots,m$ can be extended to a filling $\sigma'$ of $\Delta'(\CL)$ by periodic extension: for each pair $(x,y)\in\Delta'(\CL)$, there exists a unique $s\in\BZ$ such that $(x+Ks,y-Ns)\in\Delta$, and in this case we define $\sigma'(x,y)=\sigma(x+Ks,y-Ns)$. Note that the extended filling is not necessarily standard. 

\begin{lem}\label{extend_suff}
  Suppose that $a\leq K$ and $b\geq N-1$. Let $\CL$ be a $(K,N)$-periodic lattice path. Then every standard filling of $\Delta(\CL)$ extends to a standard filling of $\Delta'(\CL)$.
\end{lem}
\begin{proof}
  The region $\Delta(\CL)$ is bounded inside a rectangle of width $\CL(b)+a-\CL(N-1)$ and height $N$. The desired result follows in the case that $\CL(b)+a-\CL(N-1)\leq K$, which indeed holds under the assumptions given. 
\end{proof}

Let $\Omega(K,N,a,b)$ denote the set of pairs $(\CL,\sigma)$ where 
\begin{enumerate}
\item $\CL$ is a $(K,N)$-periodic lattice path satisfying $\CL<\CL[a,b]$,
\item $\sigma$ is a standard filling of $\Delta(\CL)$ with numbers $1,\dots,m$ such that the periodic extension $\sigma^\prime$ to $\Delta^\prime$ is also standard.
\end{enumerate}


\begin{thm}
\label{dpt_bijection}
  The map 
  \[\sigma\mapsto (\CL(\sigma),\sigma|_{\Delta(\sigma)})\]
  as defined in Lemma~\ref{skew_lemma}, defines a bijection 
  \[\DPT(K,N,a,b) \simeq \Omega(K,N,a,b).\]
\end{thm}
\begin{proof}
By definition, for all $\sigma \in\DPT(K,N,a,b)$, $\sigma|_{\Delta(\sigma)}$ is a standard filling of $\Delta(\CL)=\Delta(\sigma)$ such that the periodic extension to $\Delta'(\CL)=\Delta'(\sigma)$ is also standard. Thus $(\CL(\sigma),\sigma|_{\Delta(\sigma)}) \in \Omega(K,N,a,b)$.

We must define a map from $\Omega(K,N,a,b)$ to $\DPT(K,N,a,b)$ and show it is inverse to $\sigma\mapsto (\CL(\sigma),\sigma|_{\Delta(\sigma)})$.
  Fix $(\CL,\sigma)\in\Omega(K,N,a,b)$. By defintion, the periodic extension $\sigma'$ of $\sigma$ to $\Delta'(\CL)$ is also standard. 
From Lemma \ref{lem_latt_Delta}, for all $(x,y)\in\BZ^2$, there exists unique $r\in\BZ$ such that $(x-ra,y+rb)\in \Delta'(\CL)$. Define, for all $(x,y)\in\BZ^2$,
 \[\tilde{\sigma}(x,y):=\sigma'(x-ra,y+rb)+rm.\]
Note that $rm < \tilde{\sigma}(x,y) \leq (r+1)m$.

Let check that $\tilde{\sigma} \in\DPT(K,N,a,b)$. By defintion of $\sigma'$, $\tilde{\sigma}$ is $(K,N)$-periodic. Moreover, one has $(x+a,y-b )\in\Delta'(\CL[(r+1)a,(r+1)b]) $, so 
\[\tilde{\sigma}(x+a,y-b) = \tilde{\sigma}(x,y) +m.\]
Moreover, let us prove that the filling of $\tilde{\sigma}$ is standard.

If $(x+1-ra,y+rb)\in \Delta'(\CL)$, then 
\[\tilde{\sigma}(x+1,y)>\tilde{\sigma}(x,y),\]
as the filling $\sigma'$ of $\Delta'$ is supposed standard. Otherwise, $(x+1,y)\in \Delta'(\CL[(r+1)a,(r+1)b])$, and $(r+1)m < \tilde{\sigma}(x+1,y) \leq (r+2)m$. In particular,
\[\tilde{\sigma}(x+1,y)> \tilde{\sigma}(x,y).\]
Similarly, if $(x-ra,y+1+rb)\in \Delta'(\CL)$, then $\tilde{\sigma}(x+1,y)>\tilde{\sigma}(x,y)$. Otherwise, as $z\mapsto\CL(z)$ is striclty decreasing, then $(x,y+1)\in \Delta'(\CL[(r+1)a,(r+1)b])$, and thus in any case
\[\tilde{\sigma}(x,y+1)> \tilde{\sigma}(x,y).\]

  
  We now show that the map $(\CL,\sigma)\mapsto \tilde{\sigma}$ is inverse to $\sigma\mapsto(\CL(\sigma),\sigma|_{\Delta(\sigma)})$.

  Given $\sigma\in\DPT(K,N,a,b)$, we send $\sigma\mapsto(\CL(\sigma),\sigma|_{\Delta(\sigma)})\mapsto\tilde{\sigma}$. 
  Given $(x,y)\in\BZ^2$, using Lemma~\ref{fundamental_lemma}, we write $(x,y)=(x',y')+s(K,-N)+r(a,-b)$ where $(x',y')\in\Delta(\CL)$. Then $\tilde{\sigma}(x,y)=\sigma|_{\Delta(\sigma)}(x',y')+rm=\sigma(x,y)$. 
  Going the other way, we send $(\CL,\sigma)\mapsto\tilde{\sigma}\mapsto (\CL(\tilde{\sigma}),\tilde{\sigma}|_{\Delta(\tilde{\sigma})})$.
  The equality $\sigma=\tilde{\sigma}|_{\Delta(\tilde{\sigma})}$ comes for free since $\tilde{\sigma}$ is defined by extending $\sigma$. 
  For $0\leq y<N$, by definition both $\CL$ and $\CL(\tilde{\sigma})$ coincide with the leftmost boundary of $\Delta(\tilde{\sigma})$. Since both lattice paths are $(K,N)$ periodic, we find that $\CL=\CL(\tilde{\sigma})$.
\end{proof}

In Section~\ref{sect_count}, we consider ways to get a finite count of DPTs by counting up to certain equivalences. One such result can be obtained as a corollary to Theorem \ref{dpt_bijection}.

If $\CL$ is a $(K,N)$-periodic lattice path such that $\CL<\CL[a,b]$, let $\DPT(\CL)$ denote the set
\[\DPT(\CL) :=\{\sigma\in\DPT(K,N,a,b) \mid \CL(\sigma)=\CL\}.\]

\begin{cor}\label{cor_DPT_CL}
  The map 
  \[\sigma\mapsto \sigma|_{\Delta(\sigma)}\]
  as defined in Lemma~\ref{skew_lemma}, defines a bijection 
  \[ \DPT(\CL) \simeq \left\lbrace \begin{array}{c}
\text{ standard fillings }  \sigma  \text{ of } \Delta(\CL)\\
\text{ such that } \sigma' \text{ is standard }\end{array} \right\rbrace.\]
\end{cor}

Consider the group $\langle D,L\rangle\leq \CA$ and its action on $\DPT(K,N,a,b)$ as defined in \ref{D_def} and \ref{L_def}.
Notice that $D$ and $L$ also act on $\Omega(K,N,a,b)$ by $D(\CL,\sigma)=(D\CL,D\sigma)$ where $D\CL(y)=\CL(y)+1$ and $D\sigma(x,y)=\sigma(x-1,y)$. Similarly $L(\CL,\sigma)=(L\CL,L\sigma)$ where $L\CL(y)=\CL(y-1)$, and $L\sigma(x,y)=\sigma(x,y-1)$. By construction, the actions of $D$ and $L$ commute with the bijection of Theorem \ref{dpt_bijection}, and therefore we have a bijection  $$\DPT(K,N,a,b)\slash\langle D,L\rangle\simeq\Omega(K,N,a,b)\slash\langle D,L\rangle$$ between the sets of orbits.

\subsection{DPTs and Partitions}\label{sect_DPT_part}

It will be convenient, particularly in relation with quantum groups in Section~\ref{quantum_fusion}, to describe $\Omega(K,N,a,b)$ using the language of partitions. An \textit{integer partition with $N$ parts} is a weakly decreasing sequence $\lambda=(\lambda_1,\dots,\lambda_N)$ of integers.

Define the \emph{fundamental alcove} (the terminology here comes from the representation theory of quantum groups at root of unity, see Section~\ref{quantum_fusion}):
\[\CA_{K,N}:=\left\lbrace \lambda =(\lambda_1,\ldots,\lambda_N)\in\BZ^N \mid \begin{array}{l}
\lambda_1 \geq \lambda_2 \geq \cdots \geq \lambda_N\\
\lambda_1 - \lambda_N \leq K
\end{array}\right\rbrace.\]

\begin{prop}\label{prop_latt_part}
The following defines a bijection:
\[
\begin{array}{ccc}
 \left\lbrace \begin{array}{c}
(K,N) \text{-periodic } \\
 \text{ lattice paths }\end{array} \right\rbrace & \to & \CA_{K,N},\\
 \CL & \mapsto & \lambda(\CL):=(\CL(0),\CL(1),\ldots, \CL(N-1)).
\end{array}
\]
\end{prop}

\begin{proof}
For all $(K,N)$-periodic lattice paths $\CL$, $\CL(0)\leq \CL(-1) = \CL(N-1) +K$, thus the tuple $(\CL(0),\CL(1),\ldots, \CL(N-1))$ is indeed in $\CA_{K,N}$.

Moreover, for all $\lambda \in \CA_{K,N}$, define a $(K,N)$-periodic lattice path $\CL_\lambda$ via, for all $0\leq i \leq N-1$, $r\in\BZ$,
\[\CL_\lambda(i + rN) = \lambda_{i+1} -r K.\]
Then $\lambda\mapsto \CL_\lambda$ is the inverse map of $\lambda \mapsto \lambda(\CL)$.
\end{proof}

\begin{example}
The lattice path $\CL(\sigma)$ represented in Figure~\ref{labeling_convention}, obtained from the linear doubly periodic tableau $\sigma$ of Example~\ref{linearex} for $(K,N)=(3,2)$ is also the periodic lattice path $\CL_\lambda$ constructed from the partition $\lambda = (1,-1)$.
\end{example}

Define actions of $D$ and $L$ on partitions $\lambda \in \CA_{K,N}$ by 
\begin{align}
  D\cdot\lambda&=(\lambda_1+1,\dots,\lambda_N+1)\label{D_lambda} \\
  L\cdot\lambda&=(\lambda_N+K,\lambda_1,\lambda_2,\dots,\lambda_{N-1})\label{L_lambda}.
\end{align} 
These actions are invertible and preserve the fundamental alcove $\CA_{K,N}$. Hence one can consider, for all $i\in\BZ$, $D^i\lambda, L^i\lambda \in \CA_{K,N}$.

Here again, these notations come from the context of the fusion ring, as these actions correspond to tensoring with the determinant and the line representation (see Section \ref{quantum_fusion}).

As in \eqref{D_def} and \eqref{L_def}, the actions of $D$ and $L$ correspond respectively to horizontal and vertical translations. Indeed, we have the following result.

\begin{lem}\label{lem_latt_DL}
For all periodic lattice paths $\CL$ and all $(a,b)\in\BZ^2$,
\[\lambda(\CL[a,b])= D^a L^{-b}\lambda(\CL).\]
\end{lem}

\begin{proof}
For all $(K,N)$-periodic lattice path,
\[\lambda(\CL[1,0]) = (\CL(0)+1, \CL(1)+1 ,\ldots , \ldots, \CL(N-1) +1 ) = D\lambda(\CL),\]
and
\[\lambda(\CL[0,-1]) =(\CL(-1) + K, \CL(0), \ldots, \CL(N-2)) = L\lambda(\CL).\qedhere\]
\end{proof}

Thus we consider the shifted partition \begin{equation}\label{shifted_partition}
\lambda[a,b]:=D^a L^{-b}\lambda.
\end{equation}


Given partitions $\lambda,\mu$ with $N$ parts, if $\lambda_i\leq\mu_i$ for all $1\leq i\leq N$, we define the skew shape 
\[
\mu\setminus\lambda:=\left\lbrace(x,y)\in\BZ^2 \mid \begin{array}{c}
0 \leq x\leq N-1\\
\lambda_{y+1}\leq x<\mu_{y+1}
\end{array} \right\rbrace.
\]

If $\lambda_i\leq\lambda[a,b]_i$ for $1\leq i\leq N$, it follows from Lemma \ref{lem_Delta_m_cells} that the skew shape $\lambda[a,b]\setminus\lambda$ has exactly $m$ cells. In this case we define \begin{equation}
\Delta(\lambda)=\lambda[a,b]\setminus\lambda,
\end{equation}
\begin{equation}
\Delta'(\lambda)=\{(x+rK,y-rN)\in\BZ^2 \mid (x,y)\in\Delta(\lambda),r\in\BZ\}.
\end{equation}
Notice that for $\sigma\in\DPT(K,N,a,b)$ we have 
\begin{align*}
& \Delta(\sigma)=\Delta(\CL(\sigma))=\Delta(\lambda(\sigma)),\\
\text{and } & \Delta'(\sigma)=\Delta'(\CL(\sigma))=\Delta'(\lambda(\sigma)).
\end{align*}

\subsection{Action of the extended affine Weyl group}
The extended affine Weyl group of type $GL_m$, i.e. the extended affine permutation group, acts on the set of double periodic tableaux with respect to $(K,N,a,b)$.

Let $\widehat{\mathfrak{S}}^e_m$ be the extended affine Weyl group:
\[
\widehat{\mathfrak{S}}_m^e = \left\langle \pi, s_i, i\in \BZ/m\BZ \mid \begin{array}{ll}
      s_i s_{i+1} s_i = s_{i+1} s_i s_{i+1}, & i\in \BZ/m\BZ\\
      s_i s_j = s_j s_i, & j \neq i\pm 1 \mod m\\
      s_i^2=1, & i\in \BZ/m\BZ\\
      \pi s_i=s_{i+1}\pi, &  i \in \BZ/m\BZ.
\end{array}\right\rangle.
\]

\begin{remark}
The extended affine Weyl group contains the Weyl group, i.e. the symmetric group $\mathfrak{S}_m$, generated by the reflections $s_i$ for $1\leq i\leq m$, as well as the affine Weyl group or affine permutation group $\widehat{\mathfrak{S}}_m$ as the subgroup generated by the reflections $s_i, i \in \BZ/m\BZ$.
\end{remark}

It is useful to recall that the elements of the extended Weyl group can be interpreted as affine permutations.

\begin{definition}
  An $m$-affine permutation is a bijection $f:\BZ\to \BZ$ satisfying $f(i+m)=f(i)+m$. Affine permutations are encoded by tuple $(f(0), f(1), \cdots f(m-1))$. An $m$-tuple of integers corresponds to an $m$-affine permutation if the residues of the entries modulo $m$ are all distinct. An affine permutation is positive if $f(i)\geq 0$ whenever $i\geq 0$.
\end{definition}

The elements of the extended affine Weyl group act on $\BZ$ as $m$-affine permutations as follows:

\[
\begin{array}{ll}
s_i(j)=j+1 & \text{for $j \equiv i \mod m$} \\
s_i(j)=j-1 & \text{for $j \equiv i+1 \mod m$}\\
s_i(j)=j & \text{for $j \not\equiv i,i+1 \mod m$}\\
\pi(j)=j+1 & \text{for all $j$}.
\end{array}
\]

\begin{remark}
The elements of the affine Weyl group $\widehat{\mathfrak{S}}_m$ can be interpreted as the affine permutations $f$ satisfying $\sum_{i=1}^mf(i)=\frac{m(m+1)}{2}$.
\end{remark}

\begin{remark}\label{remGrassmannianPermutations}
With this interpretation of the affine Weyl group we can give a convenient description of the minimal length coset of representatives $\widehat{\mathfrak{S}}_m/ \mathfrak{S}_m$ as the permutations satisfying $f(1)<f(2)<\dots <f(m)$  (see e.g. \cite{Humphreys}).
\end{remark}


The extended affine Weyl group acts on the set of (non necessarily standard) doubly periodic tableaux as follows.

For $f \in \widehat{\mathfrak{S}}_m^e$ and $\sigma$ a doubly periodic tableau, we set $f\sigma$ to be the filling of $\BZ^2$ given by
\[f\sigma(x,y):=f(\sigma(x,y)), \quad (x,y) \in \BZ^2.\]

This filling is doubly periodic but not necessarily standard, even if $\sigma\in \DPT(K,N,a,b)$.


\begin{definition}
Let $\sigma \in \DPT(K,N,a,b)$.
We say a permutation $f \in \widehat{\mathfrak{S}}_m^e$ is \emph{allowed to act on} $\sigma$ if $f\sigma \in \DPT(K,N,a,b)$.
\end{definition}

\begin{example}\label{exAllowedOnLineTableau}
We consider a tableau $\sigma$ as in Example \ref{exampleLineTableau}. Clearly no permutation in $\mathfrak{S}_m$ is allowed to act on $\sigma$. Let us consider now a permutation $f$ such that $f(1)<f(2)< \dots<f(m)$. We can see that for $N$ and $\alpha$ sufficiently large, more precisely for $N > \frac{f(m)-f(1)}{m}$ and $\alpha>\frac{f(m)}m$, 
the permutation $f$ is allowed to act on $\sigma$.
\end{example}

\begin{remark}\label{WeylGroupAction}
The action of the Weyl group $\mathfrak{S}_m$ preserves $\DPT(\CL)$, and thus preserves the extended fundamental domain and the fundamental domain associated to a tableau $\sigma$. 
\end{remark}

\begin{prop}\label{propTransitivity}
For $\sigma_1, \sigma_2\in \DPT(K,N, a, b)$ there is a unique $m$-affine permutation $f\in \widehat{\mathfrak{S}}_m^e$ such that $\sigma_2(x,y) = f(\sigma_1(x,y))$ for all $x,y\in\BZ^2$. We denote this affine permutation by $\sigma_2 \sigma_1^{-1}$.
\end{prop}
\begin{proof}
Given an affine permutation $f$ with $\sigma_2=f\sigma_1$, then for $0\leq i<m$, pick $(x,y)\in\BZ^2$ with $\sigma_1(x,y)=i$. Then \begin{equation}\label{eq_fi}
f(i)=\sigma_2(x,y).
\end{equation}
This proves uniqueness.

Moreover, \eqref{eq_fi} gives a well defined affine permutation. Indeed, if $\sigma_1(x,y)=\sigma_1(x',y')$, then $(x',y')=(x,y)+s(K,-N)$ for some $s\in\BZ$ by Lemma~\ref{fund_lem_iff}, and it follows that $\sigma_2(x,y)=\sigma_2(x',y')$.
  We must also show the values $f(0),\dots,f(m-1)$ are distinct modulo $m$. Suppose not, and let $0\leq i<j<m$ with $f(i)=f(j)+rm$ for some $m$. There exists $(x,y)\in\BZ^2$ and $(x',y')\in\BZ^2$ such that $\sigma_1(x,y)=i$, and $\sigma_1(x',y')=j$. Thus $\sigma_2(x,y)=f(i)=\sigma_2(x',y') +rm $. Then again by Lemma \ref{fund_lem_iff}, we have $(x,y)=(x',y')+r(a,-b)+s(K,-N)$ for some $s\in\BZ$. We conclude that $i=j+rm$ from which it follows that $r=0$ and $i=j$, a contradiction.
\end{proof}

\begin{remark}
  Fix $\sigma_0\in \DPT(K,N, a, b)$. Then any other $\sigma\in \DPT(K,N, a, b)$ can be encoded by the affine permutation $f_\sigma=\sigma_0 \sigma^{-1}$. 
We have the following \emph{sorting algorithm}. 
Suppose $f_\sigma(i)>f_\sigma(i+1)$ for some $i\in\BZ$. Then switching the positions of $i$ and $i+1$ in $\sigma$ again produces a standard DPT. 
Indeed, switching of $i$ and $i+1$ is impossible if and only if $i$ is directly to the north or to the west of $i+1$, but in that case we must have $f_\sigma(i)<f_\sigma(i+1)$, because otherwise $\sigma_0$ would not be a DPT, so we obtain a contradiction. 
Applying this operation finite number of times (equal to the number of inversions of $\sigma_0 \sigma^{-1}$) we obtain a DPT of the form $\pi^c\sigma_0$ for some $c\in\BZ$. 
Reversing this procedure, we can produce all DPTs (up to addition of a constant) by starting with some fixed DPT $\sigma_0$ and switching positions of $i, i+1$ whenever possible. 
\end{remark}

We now want to give a description of the permutations which are allowed to act on a given DPT. To do so, we need to introduce the notion of content of a DPT.

\begin{definition}\label{def_content}
The \emph{content} function of $\sigma \in \DPT(K,N,a,b)$ is the function $C_\sigma:\BZ \to \BZ/(N+K)\BZ$ that associates to $i \in \BZ$ the difference $x-y \mod (N+K)$, where $\sigma(x,y)=i$.
\end{definition}

Notice that, if two cells $(x,y),(x',y')\in \BZ^2$ are both labeled with $i$, then the difference $x-y-(x'-y')$ is a multiple of $N+K$, so the content function is well defined.

\begin{remark}
Recall that a box $(x,y)$ is located along the $j$-th diagonal if $x-y=j$.
Then we note that the content function of $i$ determines along which diagonals we can find the box labeled by $i$.
\end{remark}

\begin{remark}\label{remContentPartition}
We observe that a tableau $\sigma$ is completely determined if we know its content function and its associated partition. This is clear once we notice that using this information we can fill in the fundamental domain of $\sigma$.
\end{remark}

The following is easy to check.

\begin{prop}\label{propContent}
The generator $s_i$ of $\widehat{\mathfrak{S}}_m^e$ is allowed to act on $\sigma \in \DPT(K,N,a,b)$ if and only if $C_\sigma(i)-C_\sigma(i+1)\neq \pm 1 \mod (N+K)$. 
\end{prop}

Notice that the above proposition states in terms of the content function that we can swap $i$ and $i+1$ if and only if $i+1$ is not located right below $i$ or immediately to its right. 

\begin{prop}\label{remAllowedPermutations}
If an $m$-affine permutation $f$ is allowed to act on $\sigma$, then, if $f=\pi^r s_{j_1}\dots s_{j_s}$ is a minimal length expression of $f$, for each $\ell=1 \dots s$, the permutation $s_{j_\ell}$ is allowed to act on $s_{j_{\ell+1}}\dots s_{j_s}\sigma$.
\end{prop}
\begin{proof}
Assume that there is a value of $\ell$ for which $s_{j_\ell}$ is not allowed to act on $s_{j_{\ell+1}}\dots s_{j_s}\sigma$. Let $h=\pi^r s_{j_1}\dots s_{j_{\ell-1}}$ and $g=s_{j_{\ell+1}} \dots s_{j_s}$. As $s_{j_\ell}(g\sigma)$ is not standard, we have $s_{j_\ell}(g\sigma)(x+1,y)=i$ and $s_{j_\ell}(g\sigma)(x,y)=i+1$ or $s_{j_\ell}(g\sigma)(x,y+1)=i$ and $s_{j_\ell}(g\sigma)(x,y)=i+1$, for some choice of $(x,y) \in \BZ^2$. By the minimality of the length of the expression for $f$, we have that $l(hs_{j_\ell})>l(h)$, which implies (see e.g. \cite{Humphreys}) that $h(i)<h(i+1)$.  This means, in the first case, that $hs_{j_\ell}(g\sigma)(x+1,y)<hs_{j_\ell}(g\sigma)(x,y)$. Similarly, in the second case, we have $hs_{j_\ell}(g\sigma)(x,y+1)<hs_{j_\ell}(g\sigma)(x,y)$. In both cases this implies that $f\sigma$ is not standard, contradicting our hypothesis.
\end{proof}

We put together the last two results to describe the affine permutations which are allowed to act on a given tableau.

\begin{prop}
An $m$ affine permutation $f=\pi^rs_{j_1}\dots s_{j_s}$ is allowed to act on $\sigma$ if and only if, for all $\ell=1, \dots s$, we have
\[
C_\sigma(s_{j_s}\dots s_{j_{\ell+1}}(j_\ell))-C_\sigma(s_{j_s} \dots s_{j_{\ell+1}}(j_\ell+1)) \neq \pm 1 \mod (N+K),
\]
where, for $\ell=s$, we set $s_{j_{\ell+1}}=\id$.
\end{prop}

The following is the analog in our situation of  Proposition~3.20 of \cite{Suzuki2005Tableaux}.

\begin{prop}\label{propContentFunction}
Let $A \in \BZ_{>0}$, $B\in \BZ/A\BZ$, and $C:\BZ \rightarrow \BZ / A \BZ$, satisfying
\begin{enumerate}
\item $C(i+m)=C(i)+B \mod A$; 
\item For $i,j \in C^{-1}(r), r \in \BZ/ A \BZ$, such that there is no integer between $i$ and $j$ in $C^{-1}(r)$, there exist unique $i<p_\pm<j$ such that $p_\pm \in C^{-1}(r\pm 1)$.
\end{enumerate}
Then there exists integers $K,N,a,b$ such that $N+K=A, a+b=B, aN-Kb=m$ and $\sigma \in \DPT(K,N,a,b)$ such that $C=C_\sigma$. 
\end{prop}

\begin{proof}
First of all, one can see that conditions $(1)$ and $(2)$ imply that $C$ is surjective, and that for all $r\in \BZ/A\BZ$, $C^{-1}(r)$ is unbounded.

For all $r\in \BZ/A\BZ$, let us write $C^{-1}(r)= \left\lbrace i_r^{(j)} \mid j \in \BZ \right\rbrace$, such that
\[\forall j \in \BZ, i_r^{(j)} < i_r^{(j+1)}, \quad i_r^{(1)} = \min\{C^{-1}(r)\cap \BZ_{>0}\}.\]
By condition $(2)$ and an induction on $j$, we have the following,
\begin{enumerate}[label=(\roman*)]
  \item if $i_{r}^{(1)} < i_{r+1}^{(1)}$ then  $\left\lbrace\begin{array}{l}
i_{r}^{(j)} < i_{r+1}^{(j)} \\
i_{r}^{(j)} > i_{r+1}^{(j-1)}
\end{array}\right.$, for all $j\in\BZ$,
  \item if $i_{r}^{(1)} > i_{r+1}^{(1)}$ then  $\left\lbrace\begin{array}{l}
i_{r}^{(j)} > i_{r+1}^{(j)} \\
i_{r}^{(j-1)} < i_{r+1}^{(j)}
\end{array}\right.$, for all $j\in\BZ$,
\end{enumerate}

We will construct the DPT $\sigma$, and its associated lattice path $\CL$ recursively. First let $\sigma(0,0)=i_0^{(1)}$, and $\CL_0=(0,0)$.

For all $0\leq r \leq A-1$, suppose $\CL_0,\ldots,\CL_r$ are constructed and $\sigma(\CL_0), \ldots , \sigma(\CL_r)$ are defined. We distinguish two cases:
\begin{itemize}
  \item if $i_{r+1}^{(1)} > i_{r}^{(1)}$, then $\CL_{r+1} = \CL_r + (1,0)$ and $\sigma(\CL_{r+1}) = i_{r}^{(1)}$ \quad $\ytableausetup{boxsize=0.8cm}
\begin{ytableau}
i_r^{(1)} & i_{r+1}^{(1)}
\end{ytableau}$ ,
  \item if $i_{r+1}^{(1)} < i_{r+1}^{(1)}$, then $\CL_{r+1} = \CL_r + (0,-1)$ and $\sigma(\CL_{r+1}) = i_{r+1}^{(1)}$ \quad $\ytableausetup{boxsize=0.8cm}
\begin{ytableau}
 i_{r+1}^{(1)} \\
 i_r^{(1)}
\end{ytableau}$ .
\end{itemize}
Once $\CL_A$ is constructed, let $\CL_A=(K,-N)$. Necessarily, $N+K=A$.

Then for $r>A$, we continue constructing  $\CL_r$ and $\sigma(\CL_r)$ recursively (the values of $i_r^{(1)}$ being $A$-periodic). Similarly, we build $\CL_r$ and $\sigma(\CL_r)$ for $r\geq 0$ by downward recursion (the new box is placed to the left or below the previous box).
By construction, $\CL$ is a $(K,N)$-periodic lattice path. Note that we have filled exactly one box in each diagonal $x-y=\text{cst}$. 

Next, for all $(r,j)\in \BZ^2$, let 
\[\sigma(\CL_r + (j-1,j-1)) = i_{r}^{(j)}.\]
This defines a $(K,N)$-periodic filling of the whole plane $\BZ^2$. Moreover, for all $(x,y)\in \BZ^2$, $\sigma(x,y)=i_r^{(j)}$ with $r=x-y$. 

Moreover, this filling is standard. Indeed, for all $(x,y)\in \BZ^2$, let $\sigma(x,y) = i_{x-y}^{(j)}$. Consider $\sigma(x+1,y) = i_{x-y+1}^{(j')}$. The configuration is the following
\[\ytableausetup{boxsize=1.3cm}
\begin{ytableau}
 i_{x-y}^{(j)} & i_{x-y+1}^{(j')}\\
\none &  i_{x-y}^{(j+1)}
\end{ytableau} .\]
By definition of $\sigma$, by moving diagonally in the plane, we encounter these configurations
\[\ytableausetup{boxsize=1.4cm}
\begin{ytableau}
 i_{x-y}^{(1)} & i_{x-y+1}^{(j'-j+1)}\\
\none &  i_{x-y}^{(2)}
\end{ytableau} \quad \text{ and } \quad
\begin{ytableau}
 i_{x-y}^{(0)} & i_{x-y+1}^{(j'-j)}\\
\none &  i_{x-y}^{(1)}
\end{ytableau}.\]
Now, if $i_{x-y+1}^{(1)}> i_{x-y}^{(1)}$, then from the first configuration, $j'=j$ and by statement~$(i)$ from above, $\sigma(x+1,y)=i_{x-y+1}^{(j)} > \sigma(x,y)$.
Otherwise, if $i_{x-y+1}^{(1)} < i_{x-y}^{(1)}$, then from the second configuration, $j'=j+1$ and by statement~$(ii)$ from above, $\sigma(x+1,y)=i_{x-y+1}^{(j+1)} > \sigma(x,y)$. We can prove similarly that $\sigma(x,y+1)>\sigma(x,y)$.

Using condition $(1)$, we see that $C(i_0^{(1)} +m) =B$, so the value $i_0^{(1)}+m$ appears in the $B$-th diagonal of $\sigma$. Let $(a,-b)\in\BZ^2$ be the position of the box labeled $i_0^{(1)}+m$ in that diagonal. We necessarily have $a+b =B$. Thus $\sigma(a,-b) = \sigma(0,0) +m$. From condition $(1)$ again, all the values appearing the $B$-th diagonal are $m$ plus the values appearing in the zeroth diagonal, and both diagonals are ordered increasingly. Thus the translation sending $i$ to $i+m$, between those two diagonals, is always $(a,-b)$, and we have $\sigma(x+a,x-b)=\sigma(x,x) +m$, for all $x\in \BZ$. 
We now consider the diagonals 1 and $B+1$, whose values are $m$ plus values of the former. From condition $(2)$, the value $i_1^{(j')}$ is the unique value between $i_{0}^{(j)}$ and $i_{0}^{(j+1)}$ in diagonal 1 ($j'=j$ or $j+1$). Necessarily, the value $i_1^{(j')} +m$ is the unique value between $i_{0}^{(j)}+m$ and $i_{0}^{(j+1)}+m$ in diagonal $B+1$. So the value $i_1^{(j')} +m$ is also obtained from $i_1^{(j')}$ by translation by $(a,-b)$.

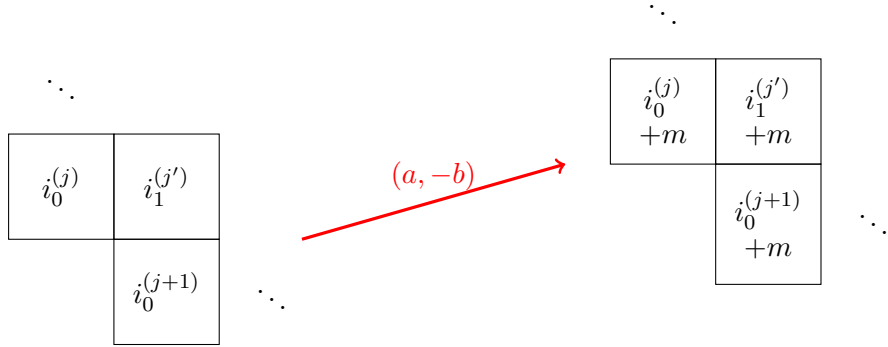
\begin{center}
\begin{tikzpicture}
\draw (-1.4,1.4) rectangle  node{$i_{0}^{(j)} $} (0,0);
\draw (0,1.4) rectangle  node{$i_{1}^{(j')} $} (1.4,0);
\draw (0,0) rectangle  node{$i_{0}^{(j+1)} $} (1.4,-1.4);
\node at (2.1,-0.7) {$\ddots$};
\node at (-0.7,2.1) {$\ddots$};
\draw[align=center] (6.6,2.4) rectangle  node{$i_{0}^{(j)}$\\ $+m $} (8,1);
\draw[align=center] (8,1) rectangle  node{$i_{1}^{(j')}$\\ $+m $} (9.4,2.4);
\draw[align=center] (8,1) rectangle  node{$i_{0}^{(j+1)}$\\ $+m $} (9.4,-0.6);
\node at (10.1,0.3) {$\ddots$};
\node at (7.3,3.1) {$\ddots$};
\draw[->, very thick, red] (2.5,0) to node[midway, above] {$(a,-b)$} (6,1) ;
\end{tikzpicture}
\end{center}
From one diagonal to the next, we show iteratively that, for all $(x,y)\in\BZ^2$, $\sigma(x+a,y-b) = \sigma(x,y) + m$. Thus $\sigma$ is indeed a doubly periodic tableau with respect to $(K,N,a,b)$. From this reasoning, we also deduce that $\CL< \CL[a,b]$ and that the boxes between these lattice paths only have fillings between 1 and $m$, thus the relation $m=aN-bK$ is satisfied (using Lemma~\ref{lem_Delta_m_cells} for example).

Finally, one clearly has $C=C_\sigma$.

\end{proof}

\subsection{Some Counting Formulas and Bounds on the Number of DPT}\label{sect_count}

In order to obtain finite counts of DPT, we consider two restrictions: first counting DPT with 0 in $(0,0)$, then counting DPT with fixed lattice path. We then relate the two ways using Dyck paths.

Let $\Lambda(K,N,a,b)$ denote the set of pairs $(\lambda,\sigma)$ where 
\begin{enumerate}
\item $\lambda=(\lambda_1,\dots,\lambda_N)\in\CA_{K,N}$  with $\lambda_i\leq\lambda[a,b]_i$ for all $1\leq i\leq N$, 
\item $\sigma$ is a standard filling of $\Delta(\lambda)$ with numbers $1,\dots,m$, 
\item  for each box $(x,y)$ in the $N$th row of $\lambda[a,b]\setminus\lambda$, we have $\sigma(x,y)<\sigma(x+K,y-N+1)$ whenever $(x+K,y-N+1)\in \Delta(\lambda)$.
\end{enumerate}

As a consequence of Proposition~\ref{prop_latt_part}, the following map is a bijection:
\begin{align*}
\Omega(K,N,a,b) & \xrightarrow{\sim} \Lambda(K,N,a,b),\\
(\CL,\sigma)&\mapsto(\lambda(\CL),\sigma)
\end{align*}
where the inverse map sends $(\lambda,\sigma)$ to $(\CL_\lambda,\sigma)$.
Composing with the bijection in Theorem \ref{dpt_bijection} gives a bijection 
\[\DPT(K,N,a,b)\xrightarrow{\sim}\Lambda(K,N,a,b).\]

We extend the actions of $D$ and $L$ to $\Lambda(K,N,a,b)$ to commute with the bijection with $\DPT(K,N,a,b)$. The action of $D$ shifts the fundamental domain to the right by one, so we set $D(\lambda,\sigma)=(D\lambda,D\sigma)$ where $D\sigma(x,y)=\sigma(x-1,y)$. The action of $L$ cuts the bottom row off of the fundamental domain, shifts it by $(K,-N)$, and appends it as a new top row. Thus we have $L(\lambda,\sigma)=(L\lambda,L\sigma)$ where $L\sigma(x,y)=\sigma(x',y'-1)$ where $(x',y'-1)\in\Delta(\lambda)$ and $(x,y-1)=(x',y'-1)\mod (K,-N)$. 
\begin{example}
  For $(K,N,a,b)=(3,2,4,1)$, let $\lambda=(2,0)$. Then $D\lambda=(3,1)$ and $L\lambda=(3,2)$ and $\lambda[a,b]=(4,3)$. The operators $D,L$ act on the filling $\sigma$ of $(4,3)\setminus(2,0)$ with reading word $13524$ as follows:
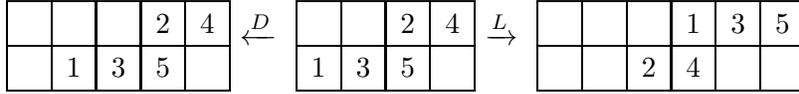
\begin{center}
\ytableausetup{boxsize=1.5em}
\ytableaushort{\none \none \none 24, \none 135,}
* {5,5}
\ $\xleftarrow{D}$\ 
\ytableausetup{boxsize=1.5em}
\ytableaushort{\none \none 24, 135,}
* {4,4}
\ $\xrightarrow{L}$\ 
\ytableausetup{boxsize=1.5em}
\ytableaushort{\none \none\none 135, \none\none 24,}
* {6,6}
\end{center}
where the leftmost box in each case has $x$ coordinate $0$ and all boxes have $y$ coordinate $0$ or $1$.
\end{example}

\begin{thm}\label{dpt_mod_dl}
  There is a bijection between $\DPT(K,N,a,b)\slash\langle D,L\rangle$ and the set of all pairs $(\lambda,\sigma)$, where 
\begin{enumerate}
\item $\lambda=(\lambda_1,\dots,\lambda_{N-1},0)\in \CA_{K,N}$ is an partition with $\lambda_i\leq\lambda[a,b]_i$ for all $1\leq i\leq N$,
\item $\sigma$ is a standard filling of $\Delta(\lambda)$ such that $1$ is in the first row,
\item for each box $(x,y)$ in the $N$th row of $\Delta(\lambda)$, we have 
\[\sigma(x,y)<\sigma(x+K,y-N+1)\]
whenever $(x+K,y-N+1)\in \Delta(\lambda)$.
\end{enumerate}
\end{thm}
\begin{proof}
We give a bijection between pairs $(\lambda,\sigma)$ satisfying the conditions above and the set $\Omega(K,N,a,b)\slash\langle D,L\rangle$. The result follows by composing this bijection with the bijection \\  $\DPT(K,N,a,b)\slash\langle D,L\rangle\to\Omega(K,N,a,b)\slash\langle D,L\rangle$ which follows from Theorem \ref{dpt_bijection}.
Given $[(\CL,\sigma)]\in\Omega(K,N,a,b)\slash\langle D,L\rangle$, there is a unique representative $(\CL_0,\sigma_0)$ in $\Omega(K,N,a,b)$ such that $\sigma_0$ assigns a $1$ to the first box in the top row of $\Delta(\CL_0)$ and the first box in the last row of $\Delta(\CL_0)$ has $x$ coordinate $0$. Define the partition $\lambda_0$ by $(\lambda_0)_i=\CL_0(i-1)$ for $1\leq i\leq N$. Since $\CL_0$ is $(K,N)$ periodic, and $(\lambda_0)_N=0$ by assumption, it follows that $(\lambda_0)_1\leq K$. By definition we have $\lambda_0[a,b]_i=\CL_0[a,b](i-1)$ for $1\leq i\leq N$. Therefore $\Delta(\CL_0)=\lambda_0[a,b]\setminus\lambda_0$ and we define $[(\CL,\sigma)]\mapsto (\lambda_0,\sigma_0)$.

  Let $(\lambda,\sigma)$ satisfy the given conditions. Then the partition $\lambda$ extends uniquely to a $(K,N)$-periodic lattice path $\CL$ and $\sigma$ gives a standard filling of $\Delta(\CL)$ which extends to a standard filling of $\Delta'(\CL)$ by assumption. The map $(\lambda,\sigma)\mapsto [(\CL,\sigma)]$ is inverse to the map $[(\CL,\sigma)]\mapsto (\lambda_0,\sigma_0)$. 
\end{proof}


\begin{example}\label{mod_dl_ex}
  Let us return to the case $(K,N,a,b)=(3,2,4,1)$. The partitions $\lambda$ with $2$ parts, $\lambda_1\leq 3$, and $\lambda_2=0$ are $(3,0),(2,0),(1,0)$ and $(0,0)$. We have $(3,0)[4,1]=(4,4)$, $(2,0)[4,1]=(4,3)$, $(1,0)[4,1]=(4,2)$, and $(0,0)[4,1]=(4,1)$. The associated skew shapes $\lambda[a,b]\setminus\lambda$ are:
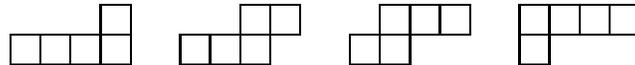
\begin{center}
  \ytableausetup{boxsize=1em}
  \ydiagram{3+1,0+4} \ \ \ \ \ 
  \ydiagram{2+2,0+3} \ \ \ \ \ 
  \ydiagram{1+3,0+2} \ \ \ \ \ 
  \ydiagram{0+4,0+1}
\end{center}
There are $11$ fillings of these satisfying the conditions of Theorem \ref{dpt_mod_dl}:
  
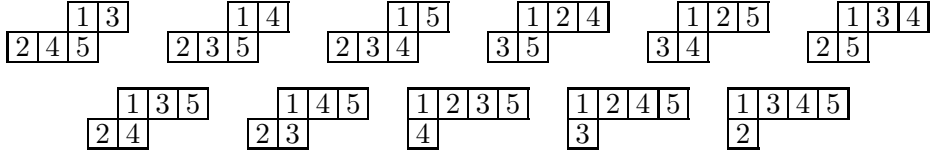
\begin{center}
\ytableausetup{boxsize=1em}
    \begin{ytableau}
    \none & \none & 1& 3 \\
    2 & 4 & 5
    \end{ytableau}
\ \ 
\ytableausetup{boxsize=1em}
    \begin{ytableau}
    \none & \none & 1& 4 \\
    2 & 3 & 5
    \end{ytableau}
\ \ 
\ytableausetup{boxsize=1em}
    \begin{ytableau}
    \none & \none & 1& 5 \\
    2 & 3 & 4
    \end{ytableau}
\ \ 
\ytableausetup{boxsize=1em}
    \begin{ytableau}
    \none & 1 & 2& 4 \\
    3 & 5
    \end{ytableau}
\ \ 
\ytableausetup{boxsize=1em}
    \begin{ytableau}
    \none & 1 & 2& 5 \\
    3 & 4
    \end{ytableau}
\ \ 
\ytableausetup{boxsize=1em}
    \begin{ytableau}
    \none & 1 & 3& 4 \\
    2 & 5
    \end{ytableau}
    
\vspace{3mm}

\ytableausetup{boxsize=1em}
    \begin{ytableau}
    \none & 1 & 3& 5 \\
    2 & 4
    \end{ytableau}
\ \ 
\ytableausetup{boxsize=1em}
    \begin{ytableau}
    \none & 1 & 4& 5 \\
    2 & 3
    \end{ytableau}
 \ \ 
\ytableausetup{boxsize=1em}
    \begin{ytableau}
    1& 2 & 3& 5 \\
    4
    \end{ytableau}
  \ \ 
\ytableausetup{boxsize=1em}
    \begin{ytableau}
    1& 2 & 4& 5 \\
    3
    \end{ytableau}   
 \ \ 
\ytableausetup{boxsize=1em}
    \begin{ytableau}
    1& 3 & 4& 5 \\
    2
    \end{ytableau}
\end{center}
It follows from Theorem \ref{dpt_mod_dl} that $|\DPT(3,2,4,1)\slash\langle D,L\rangle|=11$. Compare with Example \ref{mod_partial_ex}.
\end{example}

\begin{prop}
 For any $(\lambda,\sigma)\in\Lambda(K,N,a,b)$, we have
 \begin{equation}
  D^KL^{-N}(\lambda,\sigma)=(\lambda,\sigma).
  \end{equation}
\end{prop}
\begin{proof}
  This follows from the fact that we have $D^K L^{-N}(\sigma)=\sigma$ for any $\sigma\in\DPT(K,N,a,b)$ together with the fact that the actions of $D,L$ commute with the natural bijection $\DPT(K,N,a,b)\to\Lambda(K,N,a,b)$. 
\end{proof}



We now define an action of $\pi$ on $(\lambda,\sigma)\in\Lambda(K,N,a,b)$. 
Let $(x_m,y_m)$ be the cell in $\lambda[a,b]\setminus\lambda$ such that $\sigma(x_m,y_m)=m$. Pick $s\in\BZ$ such that $(x_0,y_0):=(x_m,y_m)-(a,-b)+s(K,-N)$ has $0\leq y_0<N$. Then $x_0=\lambda_{y_m+b-sN+1}-1$ and $\sigma(x_0,y_0)=0$. 
Define 
\[
  (\pi\lambda)_i =
  \begin{cases}
     \lambda_i-1 & \textnormal{if}\ i= y_m+b-sN+1 \\
     \lambda_i & \textnormal{else}
  \end{cases}
\]
so that $(x_0,y_0)\in\pi\lambda[a,b]$ and define $\pi\sigma$ by setting $\pi\sigma(x_0,y_0)=1$ and $\pi\sigma(x,y)=\sigma(x,y)+1$ for all cells which $\pi\lambda[a,b]\setminus\pi\lambda$ has in common with $\lambda[a,b]\setminus\lambda$ (that is, all cells besides $(x_0,y_0)$ and $(x_m,y_m)$.
Let
$\pi(\lambda,\sigma)=(\pi\lambda,\pi\sigma)$. 
Intuitively, letting $\Delta$ denote the domain of $\sigma$, we are finding the box containing $m$, removing it from $\Delta$, locating the corresponding box containing $0$ just to the left of $\Delta$, adding that box to $\Delta$, and then adding $1$ everywhere. 
For example, for $(K,N,a,b)=(3,2,4,1)$ we have
\begin{center}
    \begin{ytableau}
    \none & \none & 2 & 4 \\
    1 & 3 & 5
    \end{ytableau}
\ $\xrightarrow{\pi}$\ 
    \begin{ytableau}
    \none & 1 & 3 & 5 \\
    2 & 4 
    \end{ytableau}
\end{center}
where in this case $s=1$.
One can uniquely extend this definition to get pairs $\pi^i(\lambda,\sigma)=(\pi^i\lambda,\pi^i\sigma)$ for all $i\in\BZ$.  By construction, given $\sigma\in\DPT(K,N,a,b)$ we have 
\[\lambda(\CL(\pi^c\sigma))=\pi^c\lambda(\CL(\sigma))\]
\[\Delta(\pi^c\sigma)=\pi^c\lambda(\sigma)[a,b]\setminus\pi^c\lambda(\sigma)\]
for any $c\in\BZ$. 
Translating Theorem \ref{dpt_bijection} into the language of partitions, we see there is a bijection between $\DPT(K,N,a,b)$ and $\Lambda(K,N,a,b)$ sending $\sigma\mapsto(\lambda,\sigma|_{\Delta(\lambda)})$ where $\lambda=\lambda(\CL(\sigma))$. This bijection commutes with the operators $D,L,\pi$. Therefore we have the following.
\begin{thm}\label{dpt_mod_partial}
  There is a bijection between $\DPT(K,N,a,b)\slash\langle \pi\rangle$ and the set of pairs $(\lambda,\sigma)$ satisfying 
  \begin{enumerate}
  \item $\lambda=(\lambda_1,\dots,\lambda_{N-1},0)\in \CA_{K,N}$ is an partition with $\lambda_i\leq\lambda[a,b]_i$ for all $1\leq i\leq N$,
\item $\sigma$ is a standard filling of $\Delta(\lambda)$ with 1 in the $N$th row
\item for each box $(x,y)$ in the $N$th row of $\lambda[a,b]\setminus\lambda$, we have \[\sigma(x,y)<\sigma(x+K,y-N+1)\]
whenever $(x+K,y-N+1)\in \Delta(\lambda)$
\end{enumerate}
\end{thm}
\begin{proof}
For any pair $(\lambda,\sigma)$, apply $\pi$ (or $\pi^{-1}$) repeatedly until the first box in the bottom row of $\Delta(\lambda)$ has $x$ coordinate zero and is filled by $1$. 
\end{proof}

\begin{example}\label{mod_partial_ex}
  Let us return to the case $(K,N,a,b)=(3,2,4,1)$. The partitions $\lambda$ with $2$ parts, $\lambda_1\leq 3$, and $\lambda_2=0$ are $(3,0),(2,0),(1,0)$ and $(0,0)$. We have $(3,0)[4,1]=(4,4)$, $(2,0)[4,1]=(4,3)$, $(1,0)[4,1]=(4,2)$, and $(0,0)[4,1]=(4,1)$. The associated skew shapes are
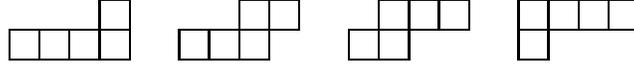
\begin{center}
  \ytableausetup{boxsize=1em}
  \ydiagram{3+1,0+4} \ \ \ \ \ 
  \ydiagram{2+2,0+3} \ \ \ \ \ 
  \ydiagram{1+3,0+2} \ \ \ \ \ 
  \ydiagram{0+4,0+1}
\end{center}
There are $11$ fillings of these satisfying the conditions of Theorem \ref{dpt_mod_partial}:
  
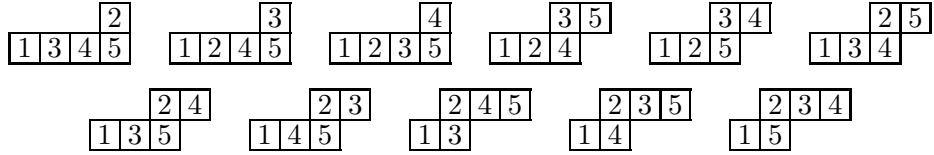
\begin{center}
\ytableausetup{boxsize=1em}
    \begin{ytableau}
    \none & \none & \none& 2 \\
    1 & 3 & 4 & 5
    \end{ytableau}
\ \ 
\ytableausetup{boxsize=1em}
    \begin{ytableau}
    \none & \none & \none& 3 \\
    1 & 2 & 4 & 5
    \end{ytableau}
\ \ 
\ytableausetup{boxsize=1em}
    \begin{ytableau}
    \none & \none & \none& 4 \\
    1 & 2 & 3 & 5
    \end{ytableau}
\ \ 
\ytableausetup{boxsize=1em}
    \begin{ytableau}
    \none & \none & 3& 5 \\
    1 & 2 & 4
    \end{ytableau}
\ \ 
\ytableausetup{boxsize=1em}
    \begin{ytableau}
    \none & \none & 3& 4 \\
    1 & 2 & 5
    \end{ytableau}
\ \ 
\ytableausetup{boxsize=1em}
    \begin{ytableau}
    \none & \none & 2& 5 \\
    1 & 3 & 4
    \end{ytableau}
    
\vspace{3mm}

\ytableausetup{boxsize=1em}
    \begin{ytableau}
    \none & \none & 2& 4 \\
    1 & 3 & 5
    \end{ytableau}
\ \ 
\ytableausetup{boxsize=1em}
    \begin{ytableau}
    \none & \none & 2& 3 \\
    1 & 4 & 5
    \end{ytableau}
 \ \ 
\ytableausetup{boxsize=1em}
    \begin{ytableau}
    \none & 2 & 4& 5 \\
    1& 3
    \end{ytableau}
  \ \ 
\ytableausetup{boxsize=1em}
    \begin{ytableau}
    \none & 2 & 3& 5 \\
    1& 4
    \end{ytableau}   
 \ \ 
\ytableausetup{boxsize=1em}
    \begin{ytableau}
    \none & 2 & 3& 4 \\
    1& 5
    \end{ytableau}
\end{center}
It follows from Theorem \ref{dpt_mod_partial} that $|\DPT(3,2,4,1)\slash\langle \pi\rangle|=11$. 
By comparison with Example \ref{mod_dl_ex} we have see that $\DPT(3,2,4,1)\slash\langle D,L\rangle$ and $\DPT(3,2,4,1)\slash\langle \pi\rangle$ have the same number of elements. 
\end{example}

The set $\DPT(K,N,a,b)/\langle D,L\rangle$ can be thought of as the set of $\sigma\in\DPT(K,N,a,b)$ with $0$ in the $(0,0)$ cell. The same can be said about $\DPT(K,N,a,b)/\langle \pi\rangle$. Thus the cardinalities of these sets necessarily coincide. In fact, there is a nice bijection. 

\begin{thm}\label{mod_dl_mod_pi_bij}
 There is a bijection
 \[\DPT(K,N,a,b)\slash\langle D,L\rangle \xrightarrow{\sim}\DPT(K,N,a,b)\slash\langle \pi\rangle.\]
\end{thm}


\begin{proof}
  It suffices to give a bijection between the pairs $(\lambda,\sigma)$ appearing in Theorem \ref{dpt_mod_dl} and those appearing in Theorem \ref{dpt_mod_partial}. Given a pair $(\lambda,\sigma)\in \Lambda(K,N,a,b)$ with $\lambda_N=0$ and 1 in the last now, the pair $(L\lambda,L\sigma)$ has $L\lambda_1=K$ and the filling has 1 in the first row. Now apply $D^{-\lambda_{N-1}}$ to get a pair $(D^{-\lambda_{N-1}}L\lambda,D^{-\lambda_{N-1}}L\sigma)$ satisfying $(D^{-\lambda_{N-1}}L\lambda)_1\leq K$, $(D^{-\lambda_{N-1}}L\lambda)_N=0$, and has 1 in the first row. This operation is clearly reversible. 
\end{proof}
\begin{prop}
  For any $(\lambda,\sigma)\in\Lambda(K,N,a,b)$, we have
  \begin{equation}
    \pi^m(\lambda,\sigma)=D^{-a}L^b(\lambda,\sigma)
  \end{equation}
\end{prop}
\begin{proof}
  This follows from the fact that we have $\pi^m\sigma=D^{-a}L^b\sigma$ for all $\sigma\in\DPT(K,N,a,b)$ and the bijection of $\DPT(K,N,a,b)$ with $\Lambda(K,N,a,b)$ commutes with the actions of $D,L,\pi$. 
\end{proof}


A $(K,N)$-Dyck path is a portion of a lattice path from $(x,y)$ to $(x+K,y-N)$ for some $x,y$. The collection of $(K,N)$-Dyck paths are in correspondence with $(K,N)$-periodic lattice paths in the following way. Any $(K,N)$-Dyck path extends periodically to a $(K,N)$-periodic lattice path. Conversely, given a $(K,N)$-periodic lattice path, take an infinite line with slope $-\frac{K}{N}$ which is disjoint with the path and below it, and move the line towards the path until it touches. If the line touches the path at $(x,y)$, then it also touches at $(x+K,y-N)$. If $N$ and $K$ are relatively prime, it is impossible to have any intersection points in between $(x,y)$ and $(x+K,y-N)$. In the relatively prime case, the $(K,N)$-Dyck paths modulo translation by the vector $(K,-N)$ are in bijective correspondence with $(K,N)$-periodic lattice paths and therefore lattice paths modulo $\langle D,L\rangle$ are in bijection with the $(K,N)$-Dyck paths ending at $(0,0)$. 
For the non-relatively prime case, see \cite{Gorsky2017Dyck}, where they give a bijection between $(dK,dN)$-Dyck paths and $d$ tuples of $(K,N)$-Dyck paths subject to certain gluing data. 
\begin{example}
  A periodic lattice path corresponding to a linear DPT (as in Example \ref{linearex}) gives rise to the Dyck path which minimally stays above the diagonal.
\end{example}
One can associate to a $(K,N)$-Dyck path $d$ a partition $\lambda(d)=(\lambda_1,\dots,\lambda_N)$ where the $\lambda_i$ are the $x$ values of the vertical steps. Using this, we can define a skew shape $\Delta(d)=\lambda(d)[a,b]\setminus \lambda(d)$. 
With the obvious actions of $D$ and $L$ in mind, we conclude the following. 

\begin{thm}\label{dpt_dyck}
Suppose that $\gcd(K,N)=1$. Then there is a bijection between the set $\DPT(K,N,a,b)\slash\langle D,L\rangle$ and the set of pairs $(d,\sigma)$ such that
\begin{enumerate}
\item $d$ is a $(K,N)$-Dyck path ending at $(0,0)$ such that $\lambda(d)[a,b]\geq\lambda(d)$
\item $\sigma$ is a filling of $\Delta(d)$ with numbers $1,\dots,m$ such that for each box $(x,y)$ in the $N$th row of $\Delta(d)$, we have $\sigma(x,y)<\sigma(x+K,y-N+1)$ whenever $(x+K,y-N+1)\in \Delta(d)$
\end{enumerate}
\end{thm}

\begin{example}\label{mod_partial_ex_2}
  Let us return again to the case $(K,N,a,b)=(3,2,4,1)$. There are 2 Dyck paths from $(-3,-2)$ to $(0,0)$ staying above the diagonal. The associated skew shapes are
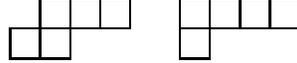
\begin{center}
  \ytableausetup{boxsize=1em}
  \ydiagram{1+3,0+2} \ \ \ \ \ 
  \ydiagram{0+4,0+1}
\end{center}
There are $11$ fillings of these satisfying the conditions of Theorem \ref{dpt_dyck}:
  
\begin{center}
\ytableausetup{boxsize=1em}
    \begin{ytableau}
    \none & 1 & 2& 4 \\
    3& 5
    \end{ytableau}
  \ \ 
    \begin{ytableau}
    \none & 1 & 2& 5 \\
    3& 4
    \end{ytableau}   
 \ \ 
    \begin{ytableau}
    \none & 1 & 4& 4 \\
    2& 5
    \end{ytableau}
 \ \ 
    \begin{ytableau}
    \none & 1 & 3& 5 \\
    2& 4
    \end{ytableau}
 \ \ 
    \begin{ytableau}
    \none & 1 & 4& 5 \\
    2& 3
    \end{ytableau}
 \ \ 
    \begin{ytableau}
    \none & 2 & 3& 4 \\
    1& 5
    \end{ytableau}

\vspace{3mm}

    \begin{ytableau}
    \none & 2 & 3& 5 \\
    1& 4
    \end{ytableau}
 \ \ 
    \begin{ytableau}
    \none & 2 & 4& 5 \\
    1& 3
    \end{ytableau}
 \ \ 
    \begin{ytableau}
    1 & 3 & 4& 5 \\
    2
    \end{ytableau}
 \ \ 
    \begin{ytableau}
    1 & 2 & 4& 5 \\
    3
    \end{ytableau}
 \ \ 
    \begin{ytableau}
    1 & 2 & 3& 5 \\
    4
    \end{ytableau}
\end{center}
\end{example}

With Lemmas \ref{shift_suff} and \ref{extend_suff} in mind, we see that for every $(K,N)$-Dyck path $d$, we have $\lambda(d)[K,N-1]>\lambda(d)$, and every standard filling of $\Delta(\lambda):=\lambda[K,N-1]\setminus\lambda$ extends to a standard filling of $\Delta'(\lambda)$. Thus the case $(K,N,a,b)=(K,N,K,N-1)$ should be the simplest case for which we can attempt writing down an exact formula to count $\DPT(K,N,a,b)\slash\langle D,L\rangle$. 
Note that $|\DPT(K,N,K,N-1)\slash\langle D,L\rangle|=|\DPT(N,K,1,0)\slash\langle D,L\rangle|$
by Remark \ref{mod_nk} and Equation \ref{phi} 
so we find that (unsurprisingly) the simplest case is where $(a,b)=(1,0)$. It turns out the result is related to the rational shuffle conjecture. 
By Theorem \ref{dpt_dyck} and the discussion above, $\DPT(K,N,1,0)\slash\langle D,L\rangle$ is counted by pairs $(d,\sigma)$ where $d$ is a $(K,N)$-Dyck path and $\sigma$ is a standard filling of $\Delta(d)$.
  In the case that $\gcd(K,N)=1$, these are exactly the rational parking functions appearing in \cite{armstrong2016rational} which are shown to be counted by $K^{N-1}$ in their Corollary 4 
  We give an independent proof of this fact which works not only in the relatively prime case, but for general pairs $(K,N)$.

\begin{prop}\label{k_n-1}
  For any $K,N>0$, we have
  \begin{equation}
    |\DPT(K,N,1,0)\slash\langle D,L\rangle|=K^{N-1}.
  \end{equation}
\end{prop}
\begin{proof}
  Define $\Sigma=\{(0,y) \ | \ 0\leq y<N\}$ and assume $\sigma(0,0)=0$. For any $(x,y)\in\BZ^2$ there exists unique $r,s\in\BZ$ such that $(x,y)+r(K,-N)+s(1,0)\in\Sigma$ so the values $\sigma(x,y)$ for $(x,y)\in\Sigma$ completely determine $\sigma$. We have $\sigma(0,N)=\sigma(K,0)=KN$ so for $0<i<N$ we have $0<\sigma(0,i)<KN$.
  Suppose that $\sigma(0,i)=\sigma(0,j)+rN$ for some $r\in\BZ$ and $0<i<j<N$. 
  It follows from Lemma \ref{fund_lem_iff} that $(0,j)=(0,i)+s(K,-N)+r(1,0)$. But $0<i<j<N$ implies $s=0$ so we conclude that $r=0$ and consequently $i=j$. This is a contradiction, and we conclude that all values $\sigma(0,i)$ are distinct mod $N$ for $0\leq i<N$. Thus $\sigma$ is determined by a choice of $N-1$ numbers which are all distinct and nonzero mod $N$ (that is, an $N$-affine permutation) and are strictly between $0$ and $KN$, so $K^{N-1}$ is an upper bound for the number of DPT. It remains to show that every such filling determines a DPT.
 
  Let $\sigma_0$ be the DPT determined by $\sigma_0(0,y)=y$ for $0\leq y<N$. We have that $\sigma_0(x,y)=y+(r+s)N$ where $(x,y)=(0,y')+s(K,-N)+r(1,0)$. 
Let $f:\BZ\to\BZ$ be an $N$-affine permutation such that $f(0)=0$ and the numbers $f(1),\dots,f(N-1)$ are contained in the integer interval $(0,KN)$. 
  Then $f$ determines a filling $\sigma_f=f\sigma_0$. 
  We have $\sigma_f(x,y)=f\sigma_0(x,y)=f(y+(r+s)N)$ where $(x,y)=(0,y')+s(K,-N)+r(1,0)$. 
  We now check that $f_\sigma$ is standard. We have $f_\sigma(x+1,y)=f_\sigma(x,y)+N$ by definition. 
  Moving from $(x,y)$ to $(x,y+1)$ either leaves both $s$ and $r$ fixed, in which case $\sigma_f(x,y)<\sigma_f(x,y+1)$ by assumption, or we decrease $s$ by $1$ and increase $r$ by $K$ in which case $\sigma_f(x,y+1)=y+1+(r+K+s-1)N>y+(r+s)N=\sigma_f(x,y)$.
\end{proof}

One may consider what happens in the more general case $\gcd(K,a)=1$. It turns out the argument above partially generalizes but provides only an upper bound. 

\begin{prop}\label{upper_bound}
  Suppose that $\gcd(K,a)=1$. Then 
  \begin{equation}
    |\DPT(K,N,a,b)\slash\langle D,L\rangle|\leq K^{m-1}
  \end{equation}
\end{prop}
\begin{proof}
  Let $\Sigma'=\{(0,y) \ | \ y\in\BZ\}$ and let $\Sigma=\{(0,y) \ | \ 0\leq y<m\}$. 
  First we show that for any $(x,y)\in\BZ^2$ there exist unique $r,s\in\BZ$ such that $(x,y)+r(K,-N)+s(a,-b)\in\Sigma$. 
  Let $p,q\in\BZ$ such that $pK+qa=1$. 
  Observe that $(0,y'):=(x,y)+(-xp)(K,-N)+(-xq)(a,-b)\in\Sigma'$. Now, given we have $(0,y')+(-ra)(K,-N)+rK(a,-b)=(0,y'+rm)$ for any $r\in\BZ$. 
  Given $\sigma\in\DPT(K,N,a,b)$, with $\sigma(0,0)=0$, we have $\sigma(0,m)=\sigma((0,0)-a(K,-N)+K(a,-b)=Km$. 
  It follows from Lemma \ref{fund_lem_iff} that the values $\sigma(0,j)$ are distinct mod $m$ for $0\leq j<m$. 
  Thus $\sigma$ is determined by a choice of $m-1$ numbers strictly between $0$ and $Km$ which are all distinct and nonzero mod $m$. 
\end{proof}

We now move on to the case of DPT corresponding to a fixed partition $\lambda\in\CA_{K,N}$. Suppose that $\lambda[a,b]>\lambda$, and let 
$\omega(\lambda)=(\lambda_N,\dots,\lambda_N)$.
Let $\Delta_0(\lambda)=\lambda[a,b]\setminus\omega(\lambda)$.
Let $D\subseteq\Delta_0(\lambda)$. Say that a cell $u=(x,y)\in D$ is $\textit{active}$ with respect to $D$ if the cells $(x+1,y), (x,y+1),$ and $(x+1,y+1)$ are in $\Delta_0(\lambda)\setminus D$. 
Given an active cell $u\in D$, define $\alpha_u(D)$ to be the subset of $\Delta_0(\lambda)$ gotten from $D$ by removing $(x,y)$ and adding $(x+1,y+1)$. Such an operation is called an \textit{excited move} on $D$. Let $\mathcal{E}(\lambda)$ denote the set of diagrams obtainable by performing a sequence of excited moves starting from $\lambda\setminus\omega(\lambda)$.

\begin{remark}
Corollary \ref{cor_DPT_CL} and Lemma \ref{extend_suff} imply that when $a\leq K$ and $b\geq N-1$, counting elements of $\DPT(\lambda)$ is the same as counting standard fillings of $\Delta(\lambda)$. 
  Suppose that $a\leq K$ and $b\geq N-1$ and fix a partition $\lambda\in\CA_{K,N}$.  Then by \cite[Theorem 1.2]{pak2018hook} (first announced by Naruse in 2014) we have 
  \begin{equation}
|\DPT(\lambda)|=m!\sum_{D\in\mathcal{E}(\lambda)}\prod_{u\in\Delta_0(\lambda)\setminus D}\frac{1}{h(u)}
  \end{equation}
  where $h(u)$ denotes the hook length of the cell $u$ in $\Delta_0(\lambda)$.
\end{remark}


\section{DAHA Representations from Doubly Periodic Tableaux}\label{sectDAHA}


\subsection{The double affine Hecke algebra}\label{SectDAHA}
We recall the definition of the double affine Hecke algebra of type $GL_m$ over $\BC[q^{\pm 1}, t^{\pm 1}]$.

\begin{definition}\label{def_DAHA}
Let $m \in \BZ_{>2}$. The double affine Hecke algebra $\ddot{H}_{q,t}(m)$ is the algebra over $\BC[q^{\pm 1},t^{\pm 1}]$ generated by 
\[
T_0,T_1, \dots, T_{m-1}, \pi^{\pm 1}, X_1^{\pm 1},X_2^{\pm 1}, \dots, X_m^{\pm 1},
\]
subject to the relations
\begin{enumerate}
  \item \label{rel_Hecke} $(T_i-q)(T_i+1)=0$, for $i=0, \dots, m-1$;
  \item \label{rel_TTT} $T_iT_{j}T_i=T_{j}T_iT_{j}$, for $j=i \pm 1 \mod m$;
  \item \label{rel_TiTj} $T_iT_j=T_jT_i$ if $j \neq i \pm 1 \mod m$;
  \item \label{rel_TiXi} $T_iX_iT_i=qX_{i+1}$, for $i=1, \dots m-1$, $T_0X_mT_0=t^{-1}q X_1$;
  \item \label{rel_TiXj} $T_iX_j=X_jT_i$, for $j\neq i,i+1$;
  \item $\pi X_i \pi^{-1}=X_{i+1}$, for $i=1, \dots, m-1$, $\pi X_m\pi^{-1}=t^{-1}X_1$;
  \item \label{rel_piTi}$\pi T_i \pi^{-1}= T_{i+1}$, for $i=0, \dots, m-2$, $\pi T_{m-1}\pi^{-1}=T_0$;
  \item\label{rel_XiXj} $X_iX_j=X_jX_i$, for $i,j \in \{1, \dots, m \}.$
\end{enumerate}
For $m=2$, the double affine Hecke algebra $\ddot{H}_{q,t}(2)$ is the algebra defined by the same generators and relations (1),(4)-(8).

The \emph{small DAHA} $\ddot{H}_{q,t}(m)^s$ is the subalgebra of $\ddot{H}_{q,t}(m)$ generated by $T_0,T_1, \dots, T_{m-1}$ and $X_1^{\pm 1},X_2^{\pm 1}, \dots, X_m^{\pm 1}$.
\end{definition}

\begin{remark}\label{rem_AHA}
The DAHA $\ddot{H}_{q,t}(m)$ contains as subalgebras two copies of the affine Hecke algebra, one generated by $T_0, \dots, T_{m-1}, \pi^{\pm 1}$ and the other by $T_1, \dots, T_{m-1},X_1^{\pm 1}, \dots, X_m^{\pm 1}$.

Let us denote by $\dot{H}_q(m)$ the copy of the AHA generated by $T_1, \dots, T_{m-1}$ and $X_1^{\pm 1}, \dots, X_m^{\pm 1}$. In fact, using \eqref{rel_TiXi}, one only needs $T_1, \dots, T_{m-1}$ and $X_1^{\pm 1}$. 
\end{remark}

It is convenient to extend the notation $T_i$, $X_i$ to all $i\in \BZ$ in such a way that $T_{i+m} = T_i$, $X_{i+m} = t^{-1} X_i$. 
Moreover, for $f=\pi^rs_{j_1}\dots s_{j_s}$, we set $T_f=\pi^rT_{j_1}\dots T_{j_s}.$

We will be considering the grading on $\ddot{H}_{q,t}(m)$ given by $\deg(\pi)=1$ and $\deg(X_i)=\deg(T_i)=0$. 

\subsection{Semisimple representations}

Recall that a \emph{weight} is an $m$-tuple $\underline{w}=(w_i)$, where $w_i \in q^{\BZ}$. For a $\ddot{H}_{q,t}$-module $M$ we define its \emph{weight space} of weight $\underline{w}$ as the space
\[
M_{\underline{w}}=\{v \in M \mid (X_i-w_i)v=0, i=1, \dots, m\}.
\]
If $M_{\underline{w}}\neq 0$ we say that $\underline{w}$ is a weight of $M$.

\begin{definition}\label{defn_semisimple}
A representation $M$ of $\ddot{H}_{q,t}$ is called $X$-semisimple if it is finitely generated and admits a decomposition $M=\oplus_{\underline{w}} M_{\underline{w}}$ with $\dim M_{\underline{w}}<\infty$ for all weights $\underline{w}$. 
\end{definition}

\begin{remark}
We observe that Definition~\ref{defn_semisimple} makes sense for representations of the small DAHA $\ddot{H}_{q,t}(m)^s$ and of the AHA $\dot{H}_{q,t}(m)$, therefore we will use the same terminology for $\ddot{H}_{q,t}(m)^s$-modules and $\dot{H}_{q,t}(m)$-modules.
\end{remark}

The following is the analogue of a well-known result on semisimple representations of the affine Hecke algebra (see for example \cite{Ram2003Affine}). This result was proven in Suzuki-Vazirani in \cite[Proposition 4.14]{Suzuki2005Tableaux}, for the full DAHA. As the generators $\pi^{\pm 1}$ do not play any role in the proof, it extends to representations of the small DAHA.


\begin{prop}
Let $L$ be an irreducible $X$-semisimple $\ddot{H}_{q,t}(m)$-module or $\ddot{H}_{q,t}(m)^s$-module. Then $\dim L_{\underline{w}} \leq 1$, for all weights $\underline{w}$.
\end{prop}


\subsection{The polynomial representation}
The DAHA $\ddot{H}_{q,t}(m)$ admits a polynomial representation $\mathcal{P}=\BC(q,t)[Y_1^{\pm 1}, \dots, Y_m^{\pm 1}]$. We recall the main facts about this representation which will be useful for us in the following. For more details we refer to \cite{CherednikBook}.

  The representation $\mathcal{P}$ has a basis of weight vectors labeled by the minimal length right coset representatives in $\widehat{\mathfrak{S}}_m / \mathfrak{S}_m$. In particular, the identity $1$ is part of this basis and has weight $(1, q, \dots, q^m)$. It will be useful for us to note that $T_f(1)$ gives a basis, as $f$ varies in the set $S=\{ f \in \widehat{\mathfrak{S}}_m \mid f(1)<f(2)< \dots <f(m) \}$, which is the set of minimal length coset representatives of $\widehat{\mathfrak{S}}_m / \mathfrak{S}_m$, as we saw in Remark \ref{remGrassmannianPermutations}.
  
  The polynomial representation is known to be faithful and remains faithful as $q$ and $t$ are specialized to values which are not roots of unity.

\subsection{DAHA representations from DPTs}\label{sect_DAHA_DPT}
We will define here a graded representation of the DAHA which splits as a direct sum of weight spaces with respect to the action of the $X_i$'s. 

\begin{remark}
Our representation will not satisfy the usual definition of $X$-semisimple DAHA representation as the weight spaces will not be finite dimensional.
\end{remark}

Let $q$ be a primitive root of unity of order $K+N$. For a cell $(x,y)$ in $\BZ^2$ its \emph{weight} is $q^{x-y}$. If $\sigma$ is a DPT then for any $i\in\BZ$ we set
\[
  w_\sigma(i) = q^{x-y}, \qquad(\sigma(x,y)=i).
\]
Note that, as $q$ is a root of unity of order $N+K$, the weight is well defined by the same argument we used to define the content function. In fact, the weight $w_\sigma(i)$ is $q^{C_\sigma(i)}$, where $C_\sigma(i)$ is, by abuse of notation, any representative of the content $C_\sigma(i)$ in $\BZ$.

Consider the vector space over $\BC$, denoted by $W_{(K,N,a,b)}$, with basis $\nu_\sigma$ where $\sigma$ runs over the set $\DPT(K,N,a,b)$. Recall that $q$ is now an $N+K$-th root of unity. We set $t=q^{-a-b}$.

Note that we have
\[
  m = aN-bK = (a+b)N = -(a+b)K \pmod{K+N}.
\]
This implies
\[
  q^m = t^{K} = t^{-N}.
\]

\begin{prop}\label{prop_DAHA_act}
The DAHA $\ddot{H}_{q,t}(m)$ acts on $W_{(K,N,a,b)}$ by 
\begin{align*}
 X_i \nu_\sigma & = w_\sigma(i) \nu_\sigma,\qquad \pi \nu_\sigma = \nu_{\pi \sigma},\\
  T_i \nu_\sigma & = \begin{cases}
    q \nu_\sigma                                                                                                                                          & \text{if $i$ is to the left of $i+1$,} \\
    -\nu_\sigma                                                                                                                                           & \text{if $i$ is on top of $i+1$,}      \\
    - \frac{1-q}{1-w_\sigma(i)w^{-1}_\sigma(i+1)} \nu_\sigma + \frac{1-qw_\sigma(i)w^{-1}_\sigma(i+1)}{1-w_\sigma(i)w^{-1}_\sigma(i+1)} \nu_{s_i(\sigma)} & \text{otherwise.}
  \end{cases}
\end{align*}
\end{prop}

\begin{proof}
Some of the defining relations of the DAHA can be immediately verified:

\[
  \pi X_i \nu_\sigma = w_\sigma(i) \nu_{\pi \sigma},\qquad X_i \pi \nu_\sigma = X_i\nu_{\pi \sigma} = w_\sigma(i-1)\nu_{\pi \sigma}\;\Rightarrow\; \pi X_i = X_{i+1} \pi
\]
\[
  \pi X_m = t^{-1} X_1 \pi,\quad\text{where}\quad t=q^{-a-b}.
\]

The remaining relations can be verified by direct computation.
\end{proof}

\begin{remark}\label{remTAction}
Note that $T_i\nu_\sigma$ is a multiple of $\nu_\sigma$ when $s_i$ is not allowed to act on $\sigma$ and, when $s_i$ is allowed to act on $\sigma$, is a sum of a multiple of $\nu_\sigma$ and a multiple of $\nu_{s_i\sigma}$. In particular, it will be useful to observe that this multiple of $\nu_{s_i \sigma}$ is non-zero by Proposition \ref{propContent}.
\end{remark}

\begin{remark}\label{remMapToPoly}
We observe that for the values of $K,N,a,b$ in Example \ref{exampleLineTableau}, we can define a map of $\ddot{H}_{q,t}$-modules from the polynomial representation $\mathcal{P}$ to $W_{(K,N,a,b)}$ sending the identity to the tableau with fundamental domain a line with labels from 1 to $m$, i.e. the tableau $\pi \sigma$, where $\sigma$ is the tableau we defined in the example. 
\end{remark}

For $\CL$ a $(K,N)$-periodic lattice path such that $\CL<\CL[a,b]$, let $W_\CL$ be the subspace of $W_{(K,N,a,b)}$ with basis $\nu_\sigma$ for $\sigma \in\DPT(\CL)$. 
The next result follows immediately from the expressions given in Proposition~\ref{prop_DAHA_act}, and Remark~\ref{WeylGroupAction}.

\begin{prop}\label{prop_AHA_act}
$W_\CL$ is a finite-dimensional $\dot{H}_q(m)$-submodule of $W_{(K,N,a,b)}$.
\end{prop}

\begin{remark}
By the computations in the proof of Lemma \ref{lemmaKNpositive}, we observe that $\pi^{m(K+N)}$ acts as the $m$-th power of the diagonal shift.
\end{remark}

We prove the following lemma to show that the action is well defined. 

\begin{lem}
  The product $w_\sigma(i)w^{-1}_\sigma(i+1) \neq 1$, for $0 \leq i \leq m-1$ and a DPT $\sigma$.
\end{lem}

\begin{proof}
  Let $(x,y),(x',y')\in \BZ^2$ such that $\sigma(x,y)=i$ and $\sigma(x',y')=i+1$. Our statement is equivalent to the fact that $x-y-(x'-y')$ cannot be a multiple of $N+K$. 
  Notice that it is easy to see that this quantity cannot be $0$, i.e. that a box labeled by $i$ and one labeled by $i+1$ cannot be located along the same diagonal, as $\sigma$ is a standard DPT. We assume now $x-y-(x'-y')=r(K+N)$, for some $r \in \BZ$. Then $x-y-((x'+rK)-(y'-rN))=0$. Observe now that the box $(x'+rK,y'-rN)$ is labeled by $i+1$ and is located on the same diagonal as the box $(x,y)$, which is labeled by $i$, so we obtained a contradiction.
\end{proof}

\subsection{Graded representations}

Recall that the DAHA is a graded algebra, with grading defined by
\begin{align*}
\deg(\pi^{\pm 1}) & =\pm 1,\\
\deg(X_i^{\pm 1}) = \deg(T_i) & =0, \quad (i \in \BZ).
\end{align*}

Next we define a grading on DPTs which is compatible with the grading on the DAHA and turns the representation $W_{(K,N,a,b)}$ into a graded representation.

\begin{definition}
A representation $M$ of a graded algebra $\mathcal{A}=\bigoplus_{p \in \BZ}A_p$ is a \emph{graded representation} if $M$ has a decomposition $M=\bigoplus_{i \in \BZ}M_i$ such that, for all $(p,\ell)\in \BZ^2$,
\[ A_p\cdot M_i \subset M_{p+i}.
\]
\end{definition}

\begin{remark}
A graded representation of the DAHA has a decomposition $M=\bigoplus_{i \in \BZ}M_i$ such that for all $(p,i)\in \BZ^2$, $\pi^p(M_i)\subset M_{p+i}$.
\end{remark}

The following result is clear from the definition of the grading.

\begin{lem}\label{lemIrreducibleGraded}
A graded representation $M=\bigoplus_{i\in\BZ}M_i$ of $\ddot{H}_{q,t}(m)$ is irreducible as a graded representation if and only if each graded piece $M_i$ is an irreducible representation of $\ddot{H}_{q,t}(m)^s$.
\end{lem}
In that case, we say that $M$ is a \emph{graded irreducible} representation.

  In order to define a grading on $W_{(K,N,,b)}$, we consider for a moment the action of the extended affine Weyl group on the partitions associated to DPTs. Let $\lambda$ be the partition associated to a given tableau $\sigma$, then the action of $\pi$ on $\sigma$ corresponds to the action on $\lambda$ described in Section \ref{sect_count} and, by Remark \ref{WeylGroupAction}, the action of the Weyl group does not affect $\lambda$, meaning that the partition associated to $f\sigma$ is again $\lambda$. 

Let us examine now the action of $s_0$, assuming $s_0$ is allowed to act on $\sigma$. The partition associated to $s_0\sigma$ is then given by 
\[
s_0\lambda_i=
\begin{cases}
     \lambda_i-1 & \text{if $i= y_m+b-sN+1$} \\
     \lambda_{i}+1 & \text{if $i=y_1+1$}\\
     \lambda_i & \text{else,}
\end{cases}
\]
where we choose $y_m$ and $s$ as in the description of the $\pi$ action (see Section \ref{sect_count}) and $y_1$ is such that the box $(x_1,y_1)$ is labeled by $1$ and $0 \leq y < N$.

We can now observe that the action of the affine Weyl group preserves the sum $\sum_{i=1}^N \lambda_i$, whereas the action of $\pi$ decreases the sum by $1$, so we can now define a grading as follows.

For $\sigma\in\DPT(K,N,a,b)$, let $\lambda$ be the associated partition. We set \[\deg(\sigma)=-\sum_{i=1}^N \lambda_i.\]

For a basis vector $\nu_\sigma$, we set its degree to be \[\deg(\nu_\sigma)=\deg(\sigma).\]

\begin{prop}
The module $W_{(K,N,a,b)}$ is a graded representation of $\ddot{H}_{q,t}(m)$.
\end{prop}

The following lemma establishes that a tableau $\sigma$ is determined by its content and its degree.

\begin{lem}\label{LemmaGradingWeightSp}
Let $\sigma,\tau \in \DPT(K,N,a,b)$ be tableaux of same degree, and such that their content functions are equal
\[ \deg(\sigma)=\deg(\tau), \quad C_\sigma=C_\tau.\]
Then, $\sigma=\tau$.
\end{lem}

\begin{proof}
Suppose $\deg(\sigma)=\deg(\tau)$. Let $\lambda$ and $\mu$ be the partitions associated to $\sigma$ and $\tau$, respectively.
First we note that $\lambda$ and $\mu$ cannot coincide as $\sigma$ and $\tau$ are distinct tableaux (recall Remark \ref{remContentPartition}).
For all $1\leq i \leq N$, write $c_i=\mu_i -\lambda_i \in \BZ$. Note that we have $\sum_{i=1}^N c_i=0$. Let us assume $c_i \neq 0$, for some $i$. 
Without loss of generality, we can assume that $c_i>0$ and $c_{i+1} \leq 0$.
By the definition of the partition $\lambda$, the box $(\lambda_i,i-1)$ of $\sigma$ is labeled with 
\[p:=\min\left( C_\sigma^{-1}(\lambda_i-i+1 \mod (N+K)) \cap \BZ_{>0}\right).\]
Let $r$ be the label of the box $(\lambda_i,i)$. As $\sigma$ is standard, we have $r>p$.
We now look at the filling of $\tau$. As $\mu_i=\lambda_i+c_i>\lambda_i$, the box $(\lambda_i,i-1)$ is to the left of the lattice path of $\tau$, and thus its label is non-positive. As $c_{i+1}\leq 0$ though, the box $(\lambda_i+1,i)$ of $\tau$ is labeled by $p$ (first box after the lattice path in that diagonal). The box $(\lambda_i,i)$ will again be labeled by $r$. As $\tau$ is standard, we obtain $p>r$, a contradiction. 
\end{proof}

\subsection{Graded semisimple representations}

\begin{definition}
A $\ddot{H}_{q,t}(m)$ module $M$ is called \emph{graded $X$-semisimple} if it is a graded as a $\ddot{H}_{q,t}(m)$-module and each graded piece is an $X$-semisimple representation of $\ddot{H}_{q,t}(m)^s$.
\end{definition}

\begin{remark}
Note that graded $X$-semisimple modules are not $X$-semisimple in the usual sense. Indeed, weight spaces are no longer necessarily finite dimensional.
\end{remark}

\begin{thm}\label{thmIrreducibility}
The representation $W_{(K,N,a,b)}$ is an irreducible graded $X$-semisimple representation of $\ddot{H}_{q,t}(m)$. 
\end{thm}

\begin{proof}
The fact that $W_{(K,N,a,b)}$ is a graded $X$-semisimple representation follows from the definition of the action in Proposition \ref{prop_DAHA_act} and the grading we defined on it. We now prove irreducibility.
Take a non-zero graded $X$-semisimple submodule $M$. By Lemma \ref{LemmaGradingWeightSp}, as $M$ is graded, it contains a non-zero weight vector $\nu_\sigma$, for some $\sigma \in \DPT(K,N,a,b)$. By Proposition \ref{propTransitivity}, we know that, for any tableau $\tau \in \DPT(K,N,a,b)$, there is an affine permutation $f$ such that $f\sigma=\tau$. 
By Proposition \ref{remAllowedPermutations} and Remark \ref{remTAction}, the action of $T_f+\sum_{g \in \widehat{\mathfrak{S}}_m,g \prec f} h_gT_g$, for some coefficients $h_g \in \BC[q^{\pm 1}, t^{\pm 1}][X_1^{\pm 1}, \dots, X_m^{\pm 1}]$, where $\prec$ denotes the Bruhat order, sends $\nu_\sigma$ to a non-zero multiple of $\nu_\tau$.
We have then proven that every weight vector has to be contained in $M$, so $M=W_{(K,N,a,b)}$. 
\end{proof}

\begin{remark}\label{rmkIrreducibiltySmall}
Note that a consequence of the previous theorem is that each graded piece $(W_{(K,N,a,b)})_i$ is an irreducible $X$-semisimple representation of the small DAHA $\ddot{H}_{q,t}(m)^s$ (recall \ref{lemIrreducibleGraded}).
\end{remark}

\subsection{Classification of graded semisimple representations}


For a $\ddot{H}_{q,t}(m)$-module $M$ such that $q$ acts as a primitive $A$th root of unity, and $t$ as $q^{-B}$, we extend the notion of weight periodically. For a weight $\underline{w}$, and $i\in\BZ$, write $i = \underline{i} + km$, with $\underline{i}\in [1,m]$, then 
\[w_i =q^{kB}w_{\underline{i}}.\]
The map $\BZ \to q^\BZ, i \mapsto w_i$ is still denoted $\underline{w}$.

The following result is proved by Suzuki-Vazirani in \cite[Lemma 4.19]{Suzuki2005Tableaux}, by induction on $j-i$.

\begin{lem}\label{lemConditionContent}
Let $L$ be an irreducible $X$-semisimple $\ddot{H}_{q,t}(m)^s$-module, where $q$ acts as a primitive $A$-th root of unity with $A\neq 2$. For any weight $\underline{w}$ and $i,j\in \BZ$ such that $i <j $ and $w_i=w_j$, there exists $p_\pm\in[i+1,j-1]$ respectively satisfying:
\[w_i = q^{\mp 1 }w_{p_\pm} .\]
\end{lem}

We are now ready to state the classification theorem for irreducible $X$-semisimple small DAHA representations.

\begin{thm}\label{thmSmallDahaClassification}
Let $M$ be an irreducible $X$-semisimple representation of $\ddot{H}_{q,t}(m)^s$ where $q$ acts as a primitive $A$-th root of unity and $t$ as $q^{-B}$. Then $M$ is isomorphic to a graded piece $(W_{(N,K,a,b)})_i$ of $W_{(K,N,a,b)}$, for some choice of $K,N,a,b$ such that $N+K=A$ and $a+b=B$.
\end{thm}


The proof of this theorem mostly follows the proof of \cite[Theorem 4.20]{Suzuki2005Tableaux}. We point out that the main difference lies in the combinatorial results concerning content functions which are necessary for the proof (compare our Proposition \ref{propContentFunction} to \cite[Proposition 3.20]{Suzuki2005Tableaux}). We outline here the proof and refer to \cite{Suzuki2005Tableaux} for more details.

\begin{skproof}
We start by picking a weight $\underline{w}$ of $M$. For $i \in \BZ$, let $C(i)$ be such that $w_i=q^{C(i)}$. As $q$ is an $A$-th root of unity, we can define a function $C:\BZ \rightarrow \BZ/A \BZ$ sending $i$ to $C(i)$. We note that $C$ satisfies the hypothesis of Proposition \ref{propContentFunction}. Indeed, the first condition is clearly fulfilled and the second follows by Lemma \ref{lemConditionContent}. Therefore, applying Proposition \ref{propContentFunction}, we obtain a $\sigma \in \DPT(K,N,a,b)$ with $N+K=A, a+b=B$ and $C=C_\sigma$. We now want to show that $M$ is isomorphic to $(W_{(K,N,a,b)})_i$, where $i$ is the degree of $v_\sigma$ in $W_{(K,N,a,b)}$.

  We define $\psi:M \rightarrow (W_{(K,N,a,b)})_i$ as the map of representations sending $v_{\underline{w}}$ to $v_\sigma$, i.e. we set $\psi((\sum_f \alpha_f T_f) v_{\underline{w}})=\sum_f \alpha_f T_f v_\sigma$. By the irreducibility of $(W_{(K,N,a,b)})_i$ (recall Remark \ref{rmkIrreducibiltySmall}) we can conclude that $\psi$ is an isomorphism. 

\end{skproof}

\begin{remark}
The isomorphism $\psi$ can be defined more explicitly using so-called DAHA intertwiners. Namely, using these elements of $\ddot{H}_{q,t}(m)$, one can give an explicit formula for the image of each weight vector in $M$ (see the proof of \cite[Theorem 4.20]{Suzuki2005Tableaux})
\end{remark}

As a consequence of the previous result we obtain the following classification theorem for irreducible graded $X$-semisimple representations.

\begin{cor}
Let $M$ be an irreducible graded $X$-semisimple representation of $\ddot{H}_{q,t}(m)$ where $q$ acts as a primitive $A$-th root of unity and $t$ as $q^{-B}$. Then $M$ is isomorphic to $W_{(K,N,a,b)}$, for some $K,N,a,b$ such that $N+K=A$ and $a+b=B$.
\end{cor}

\begin{proof}
We take the graded piece $M_0$ of $M$ in degree $0$. By Theorem \ref{thmSmallDahaClassification}, we have an isomorphism $\psi:M_0 \rightarrow W_{(K,N,a,b)}$ of $X$-semisimple $\ddot{H}_{q,t}(m)^s$-modules. We extend $\psi$ to a $\ddot{H}_{q,t}(m)$-homomorphism. By the irreducibility of $W_{(K,N,a,b)}$ we obtain that $\psi$ is an isomorphism. 
\end{proof}


\subsection{A faithful DAHA representation}
Here we consider the $\ddot{H}_{q,t}$ representation $\CW=\oplus_{aN-bK=m}W_{(K,N,a,b)}$, where in the direct sum, for fixed $K,N$, we take only one pair of equivalent $a,b$, both positive.

\begin{prop}\label{prop:faithful}
The representation $\CW$ is a faithful $\ddot{H}_{q,t}(m)$ representation.
\end{prop}

\begin{proof}
We need to prove that any nonzero element $\xi$ of $\ddot{H}_{q,t}(m)$ does not act as zero on $W$.
We know that the polynomial representation $\mathcal{P}$ of $\ddot{H}_{q,t}(m)$ is faithful, so $\xi(p(Y)) \neq 0$, for some $p(Y) \in \mathcal{P}$. Moreover, the element $\xi(p(Y))$, possibly up to multiplication by a denominator, can be written as a sum $\sum_{f \in S} \alpha_f T_f(1)$, where $S$ is a finite subset of $\{f \in \widehat{\mathfrak{S}}_m \mid f(1)< \dots <f(m)\}$ and each $\alpha_f$ is a nonzero polynomial in $q$ and $t$. 

  We now recall that for values of $K,N,a,b$ as in Example \ref{exampleLineTableau}, by Remark \ref{remMapToPoly}, there exists a map from $\mathcal{P}$ to $W_{(K,N,a,b)}$ sending $1$ to the tableau $\sigma_{(K,N,a,b)}$ whose fundamental domain is a line with labels $1, \dots, m$. We want to show that for some choice of $K,N,a,b$ we have $\sum_{f \in S} (\alpha_f T_f) (\sigma_{(K,N,a,b)}) \neq 0$, which would then conclude our proof. Assume this is not the case and $\sum_{f \in S} (\alpha_f T_f)(\sigma_{(K,N,a,b)})$ is zero for all values of $N,K,a,b$, so, in particular, for those values of $N,K,a,b$ for which each $f$ is allowed to act on $\sigma_{(K,N,a,b)}$. Recall that for this to be the case, by Example \ref{exAllowedOnLineTableau}, we just need to take sufficiently large values of $N$, $\alpha$.
From Remark \ref{remTAction} we deduce that $\alpha_f(q,q^{-a-b})=0$, for $q$ a primitive $N+K$-th root of unity, for all these choices of $N,K,a,b$. We now want to prove that the set of pairs $(q,q^{-a-b})$ as above is Zariski dense in $\BC^2$. This would conclude the proof as it would imply that each $\alpha_f$ is the zero polynomial, contradicting $\sum_{f \in S} \alpha_f T_f(1) \neq 0$.

The construction in Example \ref{exampleLineTableau} depends on numbers $N,\alpha,u,v,\beta$. We set $\beta=m-1$. Remaining numbers need to satisfy $u N - v \alpha = 1$. Computing $K, a, b$ we obtain
\[
K+N = m(N + \alpha),\qquad a+b = m(u+v).
\]
For fixed values of $u,v>0$ satisfying $\gcd(u,v)=1$, the values $a$ and $b$ are fixed and the set $\{(q,q^{-a-b}) \}$, as $N$ and $\alpha$ vary, is dense in the curve $\{(z,z^{-m(u+v)}), z \in \BC \}$. Indeed, adding $(u,v)$ to $(\alpha,N)$ makes $\alpha$, $N$, $K+N$ arbitrarily large, so $q$ takes infinitely many values.

Next we show that a polynomial vanishing on all these curves has to be identically zero on $\BC^2$.
We take a polynomial $p(x,y)$ and write 
\[
p(x,y)=x^np_n(y)+ \dots + p_0(y).
\]
Fix a complex number $\eta \neq 0$ which is not a root of unity. We know that for all the solutions to the equations $z^{-m(u+v)}=\eta$, as $u$ and $v$ vary over pairs of positive relatively prime integers, the polynomial 
\[
p_\eta(x)=x^np_n(\eta)+ \dots+ p_0(\eta)
\]
takes the value $0$. Therefore $p_\eta$ is a polynomial in one variable with infinitely many zeroes, so it is the zero polynomial, hence $p_i(\eta)=0, i=0, \dots,n$. As this has to be true for any choice of $\eta$, we conclude that $p_i(y)$ is the zero polynomial for $i=0, \dots,n$, hence $p(x,y)$ is the zero polynomial.
\end{proof}

\section{From DPT to quantum groups intertwiners}\label{dpt_intertwiners}

\subsection{Fusion ring for quantum groups}\label{quantum_fusion}

Let us fix $v$ a square root of $q$. Note that if $K+N$ is odd, then $v$ is also a primitive $K+N$ root of unity, and if $K+N$ is even, then $v$ is a $2(K+N)$ root of unity.

There are several ways to consider representations of the quantum group at root of unity $U_v(\mathfrak{gl}_N)$. We will focus here of the \emph{Fusion category} $\CC_\CF$ (at \emph{level} $K$), which is the quotient of the category of tilting modules by negligible modules (see for example \cite{Kirillov1996Inner} for a detailed review of the subject). In particular, the fusion category $\CC_\CF$ is a semi-simple ribbon category.

The Grothendieck ring of the fusion category is the \emph{Fusion ring} $\CF_v=\CF_v(\mathfrak{gl}_N,K)$ at level $K$. We will consider the objects in the fusion category by their classes in the fusion ring.
Let us recall the combinatorical definition of the fusion ring $\CF_v$, as in \cite{Andersen2014Fusion}. 

\begin{remark}
If $K+N$ is odd, we apply directly the construction of \cite{Andersen2014Fusion}. If it is even, one has to divide the order of the quantum parameter by 2 (see \cite[Sect. 3.2.1]{Andersen2014Fusion}), and thus also apply the construction to $K+N$.
\end{remark}

Let $X=\BZ^N$ be the weight lattice, $\{\varepsilon_1,\varepsilon_2\ldots ,\varepsilon_N\}$ its standard basis, and $\left\langle~,~\right\rangle$ the symmetric form on $\BR^N$.

A weight $\lambda\in X$ is \emph{dominant} if $\lambda_1 \geq \lambda_2\geq \cdots \geq \lambda_N$, where $\lambda=\sum_{i=1}^N \lambda_i\varepsilon_i$. Let $X^+$ be the set of dominant weights. The \emph{fundamental} weights are the $\omega_i=\varepsilon_1 + \varepsilon_2 + \cdots + \varepsilon_i$, for $1\leq i \leq N$. We introduce a particular weight
\[\rho = \frac{1}{2} \left( (N-1)\varepsilon_1 + (N-3)\varepsilon_2  + \cdots + (-N+1)\varepsilon_N\right).\]
Note that $\rho$ might not be in the weight lattice, but $2\rho \in X^+$.

Define the \emph{fundamental alcove}:
\[ \CA_{K,N}:=\left\lbrace \lambda \in X^+\mid \lambda_1 - \lambda_N \leq K\right\rbrace.
\]

Then the classes of the irreducible representations in the fusion ring $\CF_v$ are indexed by the dominant weights in the fundamental alcove. For $\lambda\in \CA_{K,N}$, let $[\lambda] \in \CF_v$ denote the class of the corresponding irreducible representation.

A weight $\lambda = (\lambda_1,\lambda_2,\ldots,\lambda_N) \in X$ is represented by $N$ rows of boxes, where for all $i$, the $i$th row is infinite to the left and stops at $\lambda_i$. Then $\lambda\in \CA_{K,N}$ if and only if the rows are non-increasing and the first row has at most $K$ more boxes than the last one.

\ytableausetup{boxsize=1.2em}
\begin{figure}[!ht]\label{figweight}
\begin{tikzpicture}[inner sep=0in,outer sep=0in]
\node[anchor=west] (O) at (0,0)  {\ydiagram{12,9,9,7}};
\draw[very thick, red] (9.9em, 3em)--(9.9em, -2.5em);
\node at (-1em,1.8em) {$\cdots$};
\node at (-1em,0.6em) {$\cdots$};
\node at (-1em,-0.6em) {$\cdots$};
\node at (-1em,-1.8em) {$\cdots$};
\node[red] at (9.9em,-3em) {0};
\end{tikzpicture}
\caption{For $N=4$, diagram corresponding to the weight $\lambda=(4,1,1,-1)$.}
\end{figure}
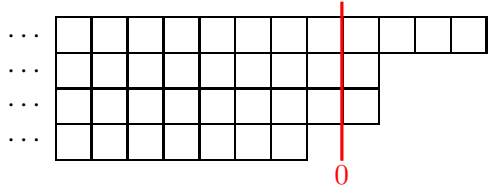

We will consider three representations in particular:
\begin{itemize}
\item let $V$ be the class of the \emph{standard} representation $[\omega_1]$,
\item let $L$ be the class of the \emph{line} representation $[K\omega_1]$,
\item let $D$ be the class of the \emph{determinant} representation $[\omega_N]$.
\end{itemize}

The diagrams corresponding to these representations are the following (where we omitted the infinite rows of boxes to the left):
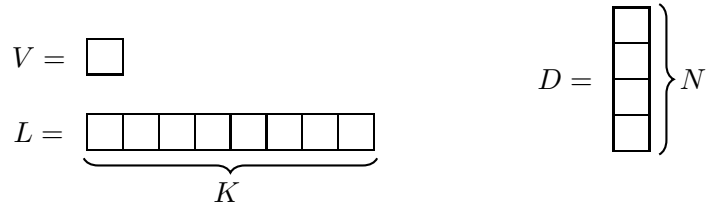
\begin{center}
\begin{tikzpicture}
\node[anchor=west] (V) at (0,0) {\ydiagram{1}};
\node at (-0.5,0) {$V=$};
\node[anchor=west] (L) at (0,-1) {\ydiagram{8}};
\node at (-0.5,-1) {$L=$};
\draw [decoration={brace,amplitude=0.5em, mirror},decorate, thick] (0.1,-1.35) -- (4,-1.35);
\node at (2,-1.8) {$K$};
\node[anchor=west] (D) at (7,-0.3) {\ydiagram{1,1,1,1}};
\node at (6.5,-0.3) {$D=$};
\draw [decoration={brace,amplitude=0.5em},decorate, thick]
 (7.75,0.7) -- (7.75,-1.3);
 \node at (8.2,-0.3) {$N$};
\end{tikzpicture}
\end{center}

Recall that the multiplication in the fusion ring is given by the \emph{Pieri rule}, for which we can give explicit formulas for the three representations above (see for example \cite{Goodmann1990Littlewood}, or \cite{Morse2012combinatorial}). 

\begin{prop}\label{propVLD}
For all $\lambda = (\lambda_1,\lambda_2,\ldots,\lambda_N)\in \CA_{K,N}$, 
\begin{align*}
V\otimes [\lambda] & = \sum_{\substack{1\leq j\leq N\\ \lambda+\varepsilon_j\in\CA_{K,N}}} [\lambda+\varepsilon_j],\\
L \otimes [\lambda] & = [(K+\lambda_N, \lambda_1,\lambda_2, \ldots, \lambda_{N-1})],\\
D \otimes [\lambda] & = [\lambda + \omega_N].
\end{align*}
\end{prop}

\begin{remark}\label{rem_shift}
Multiplying by $D$ shifts the whole diagram to the right, and multiplying by $L$ adds a line from below. 
Moreover, both $L$ and $D$ are invertible elements of the fusion ring, whose inverses are given by:
\begin{align*}
L^{-1} & = [(0,0,\ldots, 0 , -K)],\\
D^{-1} & = [-\omega_N].
\end{align*}
\end{remark}

\subsection{Ribbon diagrams and AHA action on intertwining spaces}\label{sect_ribbon_Uq}


As a ribbon category, the fusion category $\CC_\CF$ is equipped with a braiding, which is a functorial isomorphism:
\[\check{R}_{U,W} : U\otimes W \to W\otimes U, \qquad (U,W,\in\CF_v).\]

Note that the braiding here comes from the universal $R$-matrix of the quantum group, and so we also call it the \emph{$R$-matrix}.


T

Additionally, the ribbon category structure endows the fusion category with a \emph{ribbon element} $\theta$. For all $U\in \CF_v$, $\theta_U:U\to U$, is a functorial isomorphism satisfying,
\begin{equation}\label{eq_ThetaUW}
\theta_{U\otimes W} = \check{R}_{W,U}\check{R}_{U,W}\left( \theta_U\otimes \theta_W\right)\qquad (U,W\in \CF_v).
\end{equation}

The ribbon category has a nice pictorial representation, where morphisms can be drawn as directed tangles (see Figure~\ref{fig_R_theta}). In particular, the relation~\eqref{eq_ThetaUW} can observed in terms of tangles, by replacing $U$ by a ribbon, and observing that $\theta_U$ is a double twist of the ribbon, in the same direction as $\check{R}$.

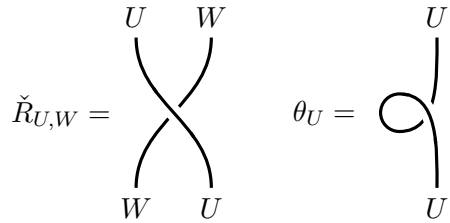
\begin{figure}[!h]
\begin{tikzpicture}
\begin{knot}[clip width =5]
\strand[very thick] (0,0) to[out=down, in =up] (1,-2);
\strand[very thick] (1,0) to[out=down, in =up] (0,-2);
\end{knot}
\node[above] at (0,0) {$U$};
\node[above] at (1,0) {$W$};
\node[below] at (1,-2) {$U$};
\node[below] at (0,-2) {$W$};
\node at (-1,-1) {$\check{R}_{U,W}=$};
\begin{knot}[clip width =5, consider self intersections, flip crossing=1]
\strand[very thick] (4,0) to[out=down, in =right] (3.5, -1.25) to[out = left, in =down] (3.25,-1) to[out=up, in=left] (3.5,-0.75) to[out=right, in=up] (4,-2);
\end{knot}
\node at (2.5,-1) {$\theta_U=$};
\node[above] at (4,0) {$U$};
\node[below] at (4,-2) {$U$};
\end{tikzpicture}
\caption{Tangle representations of $\check{R}_{U,W}$ and $\theta_U$}
\end{figure}\label{fig_R_theta}

\begin{lem}\cite{Kirillov1996Inner}
For all $\lambda\in\CA_{K,N}$,
\[\theta_{[\lambda]} = v^{\left\langle \lambda, \lambda +2\rho\right\rangle}\id_{[\lambda]}.\]
\end{lem}\label{lem_theta_lambda}

For any $U\in\CF_v$, write the decomposition of $U$ on the basis of the fusion ring:
\[U= \sum_{\lambda\in\CA_{K,N}} c_{U,\lambda}[\lambda].\]
Let $U_\lambda:= [\lambda]^{\oplus c_{U,\lambda}}$ denote the \emph{$\lambda$-isotypic component} of $U$.

\begin{cor}\label{cor_lambda_mu_nu}
For $\lambda, \mu\in \CA_{K,N}$, if one writes the decomposition
\[[\lambda]\otimes [\mu] = \sum_{\nu\in \CA_{K,N}} [\nu]^{\oplus c_{\lambda, \mu}^\nu },\]
then $\check{R}_{[\mu],[\lambda]}\check{R}_{[\lambda],[\mu]}$ acts on the isotypic component $([\lambda]\otimes [\mu])_\nu=[\nu]^{\oplus c_{\lambda, \mu}^\nu }$ by the constant
\[ v^{\left\langle \nu,\nu+2\rho \right\rangle - \left\langle \lambda,\lambda+2\rho \right\rangle-\left\langle \mu,\mu+2\rho \right\rangle }. \]
\end{cor}

\begin{proof}
From relation~\eqref{eq_ThetaUW}, 
\[\check{R}_{[\mu],[\lambda]}\check{R}_{[\lambda],[\mu]} = \theta_{[\lambda]\otimes [\mu]}\left( \theta_{[\lambda]}\otimes \theta_{[\mu]}\right)^{-1}.\]
The result is then obtained using Lemma~\ref{lem_theta_lambda}.
\end{proof}

For $U,W\in \CF_v$, we write
\[\Hom_{\CC_\CF}(U,W)\]
for the set of morphisms from $U$ to $W$ in the fusion category $\CC_\CF$.

In particular, for $m\in \BZ_{\geq 0}$ and $U,W\in\CF_v$, we define the \emph{intertwining space} between $U$ and $W$ to be:
\[ M_{U,W}^{(m)}:=\Hom_{\CC_\CF}(U, V^{\otimes m}\otimes W).
\]
These morphisms can be represented using tangles. From top to bottom, the strand going in is labeled by $U$, and going out one has $m$ strands for the $m$ copies of $V$, numbered from right to left, and one strand labeled by $W$.

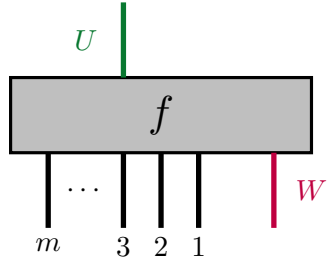
\begin{center}
\begin{tikzpicture}
\draw[fill=lightgray, very thick] (0,0) rectangle (4,1);
\draw[line width = 2pt, forest] (1.5,1) --  (1.5,2);
\draw[line width = 2pt, purple] (3.5,0) --  (3.5,-1);
\draw[line width = 2pt] (2.5,0) --  (2.5,-1);
\draw[line width = 2pt] (2,0) --  (2,-1);
\draw[line width = 2pt] (1.5,0) --  (1.5,-1);
\draw[line width = 2pt] (0.5,0) --  (0.5,-1);
\node at (1,-0.5) {$\cdots$};
\node[forest] at (1,1.5) {$U$};
\node[purple] at (4,-0.5) {$W$};
\node[below] at (2.5, -1) {$1$};
\node[below] at (2, -1) {$2$};
\node[below] at (1.5, -1) {$3$};
\node[below] at (0.5, -1) {$m$};
\node[scale=1.5] at (2,0.5) {$f$};
\end{tikzpicture}
\end{center}

As in \cite{Orellana2007Affine}, one can define an action of the Affine Hecke Algebra on the intertwining space $M_{U,W}^{(m)}$ by postcomposition. From Remark~\ref{rem_AHA}, the AHA $\dot{H}_q(m)$ is the subalgebra of $\ddot{H}_{q,t}(m)$ generated by $T_1,T_2,\ldots, T_{m-1}$ and $X_1^{\pm 1}$.

Let us define, for all $1\leq i \leq m-1$, 
\begin{align*}
R_i & := \Id_{V^{\otimes (m-i-1)}}\otimes \check{R}_{V,V}\otimes \Id_{V^{\otimes (i-1)}} \otimes \Id_W, \\
R_0^2 & := \Id_{V^{\otimes (m-1)}}\otimes \left(\check{R}_{W,V}\check{R}_{V,W}\right).
\end{align*}

\begin{prop}\label{prop_AHA_ribbon}
The map
\begin{align*}
T_i & \mapsto v\cdot R_i, \quad \textnormal{for}\quad 1\leq i\leq m-1,\\
X_1 & \mapsto R_0^2,
\end{align*}
defines an action of the Affine Hecke Algebra  $\dot{H}_q(m)$ on $M_{U,W}^{(m)}$.
\end{prop}

The action of $\dot{H}_q(m)$ can be represented in terms of tangles. The $T_i$s add a braiding of strands $i$ and $i+1$ at the bottom, and $X_1$ is a double braiding of strands $1$ and $W$. 

\begin{center}
\begin{tikzpicture}[scale=0.6]
\draw[fill=lightgray, very thick] (0,0) rectangle (8,2);
\node[scale=1.5] at (4,01) {$f$};
\draw[line width =2pt, purple] (7,0) --  (7,-2);
\draw[line width =2pt] (5.5,0) --  (5.5,-2);
\draw[line width =2pt] (5,0) --  (5,-2);
\draw[line width =2pt] (1,0) --  (1,-2);
\draw[line width =2pt] (2.5,0) --  (2.5,-1);
\draw[line width =2pt] (3.5,0) --  (3.5,-1);
\begin{knot}[clip width =4]
\strand[line width =2pt] (2.5,-1) to[out=down, in=up] (3.5,-2);
\strand[line width =2pt] (3.5,-1) to[out=down, in=up] (2.5,-2);
\end{knot}
\node at (1.75,-1) {$\cdots$};
\node at (4.25,-1) {$\cdots$};
\node[purple] at (7.6,-1) {$W$};
\node[below] at (5.5, -2) {$1$};
\node[below] at (5, -2) {$2$};
\node[below] at (3.5, -2) {$i$};
\node[below] at (2.5, -2) {$i+1$};
\node[below] at (1, -2) {$m$};
\node[scale=1.2] at (-2,0) {$T_i\cdot f=$};
\node[below] at (8.4,0) {,};
\end{tikzpicture}
\begin{tikzpicture}[scale=0.6]
\draw[fill=lightgray, very thick] (0,0) rectangle (8,2);
\node[scale=1.5] at (4,1) {$f$};
\draw[line width =2pt, purple] (7,0) --  (7,-0.4);
\draw[line width =2pt] (5,0) --  (5,-0.4);
\draw[line width =2pt] (4.5,0) --  (4.5,-2);
\draw[line width =2pt] (1,0) --  (1,-2);
\draw[line width =2pt] (4,0) --  (4,-2);
\draw[line width =2pt] (1.5,0) --  (1.5,-2);
\begin{knot}[clip width =4, flip crossing=2, ignore endpoint intersections=false]
\strand[line width =2pt] (5,-0.4) to[out=down, in=up] (7.5,-1.2) to[out= down, in =up] (5,-2);
\strand[line width =2pt, purple] (7,-0.4) to (7,-2);
\end{knot}
\node at (2.75,-1) {$\cdots$};
\node[purple] at (8,-0.6) {$W$};
\node[below] at (5, -2) {$1$};
\node[below] at (4.5, -2) {$2$};
\node[below] at (4, -2) {$3$};
\node[below] at (1, -2) {$m$};
\node[scale=1.2] at (-2,0) {$X_1\cdot f=$};
\end{tikzpicture}
\end{center}

\begin{proof}
One has check that the relations \eqref{rel_Hecke}, \eqref{rel_TTT},\eqref{rel_TiTj}, \eqref{rel_TiXj} and \eqref{rel_XiXj} from defintion~\ref{def_DAHA} are satisfied, for the generators of $\dot{H}_q(m)$. We check these relations using the tangle diagrammatics. For example, it is easy to see that relation \eqref{rel_TiTj}.

The $R$-matrix satisfied the Yang-Baxter equation, which translates into the braid relation:
\begin{equation}\label{YB}
\begin{tikzpicture}[scale=2]
\begin{knot}[clip width =4]
\strand[line width =2pt] (0,0) to[out=down, in=up] (0.5,-0.5) to[out=down, in=up] (1,-1) to[out=down, in=up] (1,-1.5);
\strand[line width =2pt] (0.5,0) to[out=down, in=up] (0,-0.5) to[out=down, in=up] (0,-1) to[out=down, in=up] (0.5,-1.5);
\strand[line width =2pt] (1,0) to[out=down, in=up] (1,-0.5) to[out=down, in=up] (0.5,-1) to[out=down, in=up] (0,-1.5);
\end{knot}
\node[scale=2] at (1.5,-0.75) {=};
\begin{scope}[shift={(2,0)}]
\begin{knot}[clip width =4]
\strand[line width =2pt] (0,0) to[out=down, in=up] (0,-0.5) to[out=down, in=up] (0.5,-1) to[out=down, in=up] (1,-1.5);
\strand[line width =2pt] (0.5,0) to[out=down, in=up] (1,-0.5) to[out=down, in=up] (1,-1) to[out=down, in=up] (0.5,-1.5);
\strand[line width =2pt] (1,0) to[out=down, in=up] (0.5,-0.5) to[out=down, in=up] (0,-1) to[out=down, in=up] (0,-1.5);
\end{knot}
\end{scope}
\end{tikzpicture}
\end{equation}

It implies the braid relation~\eqref{rel_TTT} of the DAHA presentation.

For relation~\eqref{rel_TiXj}, one has to check that, 
\[R_i R_0^2  = R_0^2 R_i, \qquad (2\leq i \leq m-1).\]
Which is clear using the tangle representation:

\begin{center}
\begin{tikzpicture}[scale=1.3]
\draw[line width =2pt, purple] (3.5,0) --  (3.5,-1);
\draw[line width =2pt] (2.75,0) --  (2.75,-1);
\draw[line width =2pt] (2.5,0) --  (2.5,-2);
\draw[line width =2pt] (0.5,0) --  (0.5,-2);
\draw[line width =2pt] (1.25,-1) --  (1.25,-2);
\draw[line width =2pt] (1.75,-1) --  (1.75,-2);
\begin{knot}[clip width =4]
\strand[line width =2pt] (1.25,0) to[out=down, in=up] (1.75,-1);
\strand[line width =2pt] (1.75,0) to[out=down, in=up] (1.25,-1);
\end{knot}
\begin{knot}[clip width =4, flip crossing=2, ignore endpoint intersections=false]
\strand[line width =2pt] (2.75,-1) to[out=down, in=up] (3.75,-1.5) to[out= down, in =up] (2.75,-2);
\strand[line width =2pt, purple] (3.5,-1) to (3.5,-2);
\end{knot}
\node at (0.825,-1) {$\cdots$};
\node at (2.075,-1) {$\cdots$};
\node[scale=1.5] at (4.5,-1) {=};

\begin{scope}[shift={(5,0)}]
\draw[line width =2pt, purple] (3.5,-1) --  (3.5,-2);
\draw[line width =2pt] (2.75,-1) --  (2.75,-2);
\draw[line width =2pt] (2.5,0) --  (2.5,-2);
\draw[line width =2pt] (0.5,0) --  (0.5,-2);
\draw[line width =2pt] (1.25,0) --  (1.25,-1);
\draw[line width =2pt] (1.75,0) --  (1.75,-1);
\begin{knot}[clip width =4]
\strand[line width =2pt] (1.25,-1) to[out=down, in=up] (1.75,-2);
\strand[line width =2pt] (1.75,-1) to[out=down, in=up] (1.25,-2);
\end{knot}
\begin{knot}[clip width =4, flip crossing=2, ignore endpoint intersections=false]
\strand[line width =2pt] (2.75,0) to[out=down, in=up] (3.75,-0.5) to[out= down, in =up] (2.75,-1);
\strand[line width =2pt, purple] (3.5,0) to (3.5,-1);
\node at (0.825,-1) {$\cdots$};
\node at (2.075,-1) {$\cdots$};
\end{knot}
\end{scope}
\end{tikzpicture}
\end{center}

For relation~\eqref{rel_XiXj}, we must check the following
\[R_1R_0^2R_1R_0^2 = R_0^2R_1R_0^2R_1 .\]
This relation can be observed by applying several times the braid relation:

\begin{center}
\begin{tikzpicture}
\begin{knot}[clip width =4,ignore endpoint intersections=false]
\strand[line width=2pt] (0,0) to[out=down, in=up] (0.5,-1) to[out=down, in=up] (1.7,-1.5) to[out=down, in=up] (0.5,-2) to[out=down, in=up] (0,-3) to[out=down, in=up] (0,-4.2);
\strand[line width=2pt] (0.5,0) to[out=down, in=up] (0,-1) to[out=down, in=up] (0,-2) to[out=down, in=up] (0.5,-3) to[out=down, in=up] (1.7,-3.5) to[out=down, in=up] (0.5,-4) to[out=down, in=up] (0.5,-4.2);
\strand[line width=2pt, purple] (1.2,0) to (1.2,-4.2);
\flipcrossings{2,4,6}
\end{knot}
\node[above] at (0,0) {$2$};
\node[above] at (0.5,0) {$1$};
\node[above, purple] at (1.2,0) {$W$};
\node at (2.1,-2.1) {=};
\draw[red, dashed] (-0.2,-1.5)rectangle (1.9,-3.5); 
\begin{scope}[shift={(2.5,0)}]
\begin{knot}[clip width =4,ignore endpoint intersections=false]
\strand[line width=2pt] (0,0) to[out=down, in=up] (0.5,-1) to[out=down, in=up] (2,-1.5) to[out=down, in=up] (2,-2) to[out=down, in=up] (1.5,-3) to[out=down, in=up] (0,-4) to[out=down, in=up] (0,-4.2);
\strand[line width=2pt] (0.5,0) to[out=down, in=up] (0,-1) to[out=down, in=up] (0,-1.5) to[out=down, in=up] (1.5,-2) to[out=down, in=up] (2,-3) to[out=down, in=up] (2,-3.5) to[out=down, in=up] (0.5,-4) to[out=down, in=up] (0.5,-4.2);
\strand[line width=2pt, purple] (1,0) to (1,-4.2);
\flipcrossings{2,4,6}
\end{knot}
\node[above] at (0,0) {$2$};
\node[above] at (0.5,0) {$1$};
\node[above, purple] at (1,0) {$W$};
\node at (2.45,-2.1) {=};
\draw[red, dashed] (-0.2,0)rectangle (2.2,-1.9); 
\draw[red, dashed] (-0.2,-2.1)rectangle (2.2,-4); 
\end{scope}
\begin{scope}
[shift={(5.4,0)}]
\begin{knot}[clip width =4,ignore endpoint intersections=false]
\strand[line width=2pt] (0,0) to[out=down, in=up] (0,-0.7) to[out=down, in=up] (1.5,-1.2) to[out=down, in=up] (2,-2.2) to[out=down, in=up] (2,-2.7) to[out=down, in=up] (0.5,-3.2) to[out=down, in=up] (0,-4.2);
\strand[line width=2pt] (0.5,0) to[out=down, in=up] (0.5,-0.2) to[out=down, in=up] (2,-0.7) to[out=down, in=up] (2,-1.2) to[out=down, in=up] (1.5,-2.2) to[out=down, in=up] (0,-3.2) to[out=down, in=up] (0.5,-4.2);
\strand[line width=2pt, purple] (1,0) to (1,-4.2);
\flipcrossings{2,4,6}
\end{knot}
\node[above] at (0,0) {$2$};
\node[above] at (0.5,0) {$1$};
\node[above, purple] at (1,0) {$W$};
\node at (2.5,-2.1) {=};
\draw[red, dashed] (-0.2,-0.7)rectangle (2.2,-2.8); 
\end{scope}
\begin{scope}
[shift={(8.3,0)}]
\begin{knot}[clip width =4,ignore endpoint intersections=false]
\strand[line width=2pt] (0,0) to[out=down, in=up] (0,-1.2) to[out=down, in=up] (0.5,-2.2) to[out=down, in=up] (1.7,-2.7) to[out=down, in=up] (0.5,-3.2) to[out=down, in=up] (0,-4.2);
\strand[line width=2pt] (0.5,0) to[out=down, in=up] (0.5,-0.2) to[out=down, in=up] (1.7,-0.7) to[out=down, in=up] (0.5,-1.2) to[out=down, in=up] (0,-2.2) to[out=down, in=up] (0,-3.2) to[out=down, in=up] (0.5,-4.2);
\strand[line width=2pt, purple] (1.2,0) to (1.2,-4.2);
\flipcrossings{2,4,6}
\end{knot}
\node[above] at (0,0) {$2$};
\node[above] at (0.5,0) {$1$};
\node[above, purple] at (1.2,0) {$W$};
\end{scope}
\end{tikzpicture}
\end{center}

Finally, one has to check the Hecke relation~\eqref{rel_Hecke} of Definiton~\ref{def_DAHA}. As $v^2=q$, these relations are, for all $1\leq i\leq m-1$, 
\begin{equation}\label{eq_Hecke_TR}
(T_i-v^2)(T_i+1)=0 \quad \Leftrightarrow \quad  (R_i-v)(R_i+v^{-1})=0.
\end{equation}

However, the action of the $R$-matrix of the quantum group $U_v(\mathfrak{gl}_N)$ on the tensor product $V\otimes V$ of two standard representations $V\simeq \BC^N$ is well-known. If $\{E_{ij}\}_{1\leq i,j\leq N}$ denotes the standard basis of $\End(V)$, then one has:
\[\check{R}_{V,V} = v \sum_{i=1}^N E_{ii}\otimes E_{ii} + \sum_{i\neq j} E_{ij}\otimes E_{ji} +(v-v^{-1})\sum_{i>j} E_{ii}\otimes E_{jj}.\]
One can see that $\check{R}_{V,V}$ satisfies $(\check{R}_{V,V} - v\Id_{V\otimes V})(\check{R}_{V,V} + v^{-1}\Id_{V\otimes V})=0$. Thus this relation is still valid in the fusion category, and the $R_i$s satisfy the Hecke relation~\eqref{eq_Hecke_TR}.
\end{proof}

\begin{remark}\label{rem_lambda_m_mu}
If $U,W$ are irreducible representations $W=[\lambda]$ and $U=[\mu]$, with $\lambda,\mu\in\CA_{K,N}$, we deduce from Proposition~\ref{propVLD} that for the intertwining space $M_{[\mu],[\lambda]}^{(m)}$ to be non-trivial, there must exist $(i_1,i_2,\ldots,i_m)\in\llbracket 1,N\rrbracket^m$ such that
\[ \mu = \lambda + \sum_{\ell=1}^m\varepsilon_{i_\ell}.\]
\end{remark}

\subsection{Lattice paths and standard skew tableaux}

Recall from Section~\ref{sect_DPT_part} that the weights in the fundamental alcove (prevously called partitions) are in bijection with the $(K,N)$-periodic lattice paths (Proposition~\ref{prop_latt_part}).

For $\lambda\in\CA_{K,N}$, let $\CL_\lambda$ denote the associated $(K,N)$-periodic lattice path. It is defined as follows, for all $0\leq i \leq N-1$, $r\in\BZ$,
\[\CL_\lambda(i + rN) = \lambda_{i+1} -r K.\]

\begin{example}\label{ex_lambda_L}
In the example of Figure~\ref{figweight}, for $(K,N)=(7,4)$  and the weight $\lambda = (4,1,1,-1)$, the corresponding lattice path is the following.
\begin{center}
\begin{tikzpicture}[scale=0.8]
\draw[step=0.5,gray, thin] (-1.3,-2.2) grid (6.2,2.2);
\draw[very thick, gray] (-1.3,0) -- (6.2,0);
\draw[very thick, gray] (0,-2.2) -- (0,2.2);
\draw[red, thick] (0,0) circle (0.8mm);
\draw[very thick] (2,0) -- (2,-0.5);
\draw[very thick] (0.5,-0.5) -- (2,-0.5);
\draw[very thick] (0.5,-0.5) -- (0.5,-1.5);
\draw[very thick] (-0.5,-1.5) -- (0.5,-1.5);
\draw[very thick] (-0.5,-1.5) -- (-0.5,-2);
\draw[very thick] (-1,-2) -- (-0.5,-2);
\draw[very thick] (-1,-2) -- (-1,-2.2);
\draw[thick, red, dashed] (-0.5,0) -- (3,0);
\draw[thick,red , dashed] (-0.5,-2) -- (-0.5,0);
\draw[thick, red, dashed] (3,-2) -- (3,0);
\draw[thick,red , dashed] (-0.5,-2) -- (3,-2);
\draw[very thick] (2,0) -- (3,0);
\draw[very thick] (3,0.5) -- (3,0);
\draw[very thick] (3,0.5) -- (4,0.5);
\draw[very thick] (4,1.5) -- (4,0.5);
\draw[very thick] (4,1.5) -- (5.5,1.5);
\draw[very thick] (5.5,2) -- (5.5,1.5);
\draw[very thick] (5.5,2) -- (6.2,2);
\node at (4.5,0.5) {$\CL_\lambda$};
\end{tikzpicture}
\end{center}
The red lines delimit the $(K,N)$-rectangle of the pattern which is repeated.
\end{example}

From now on, we consider the $(K,N)$-periodic lattice paths $(\CL_\lambda)_{\lambda_\in\CA_{K,N}}$ as the basis for the fusion ring $\CF_v$. This allows us to consider a ring structure on the set of $(K,N)$-periodic lattice paths, given by the fusion ring structure.
In particular, one can translate the Pieri rule in terms of lattice paths.

\begin{lem}
For all $\lambda \in \CA_{K,N}$,
\begin{equation}
\CL_V \cdot \CL_\lambda  = \sum_{\substack{1\leq j\leq N\\ \CL_\lambda(j-1)<\CL_\lambda(j)}} \CL_{\lambda+ \varepsilon_j}.
\end{equation}
\end{lem}

\begin{proof}
The lattice path $\CL_{\lambda+\varepsilon_j}$ is obtained from the lattice path $\CL_\lambda$ by replacing each $\CL_\lambda(j-1 + rN)$, for $r\in\BZ$, by $\CL_\lambda(j-1 + rN) + 1$ (adding one box in lines $j-1 + rN$). Thus, it is still a $(K,N)$-periodic lattice paths if and only if $\CL_\lambda(j-1)<\CL_\lambda(j)$.
\end{proof}






\begin{remark}

By multiplying by $V$, one obtains a sum of periodic lattice paths where one box has been added to a corner.
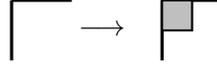
\begin{center}
\begin{tikzpicture}[scale=0.8]
\draw[very thick] (0,0) -- (1,0);
\draw[very thick] (0,0) -- (0,-1);
\node at (1.5,-0.5) {$\longrightarrow$};
\draw[very thick] (2.5,0) -- (3.5,0);
\draw[very thick] (2.5,0) -- (2.5,-1);
\draw[thick, fill=lightgray] (2.5,0) rectangle (3,-0.5);
\end{tikzpicture}
\end{center}
\end{remark}


We will use this representation to construct an explicit basis of the intertwining space $M_{[\mu],[\lambda]}^{(m)}$ using \emph{standard $(K,N)$-periodic skew tableaux}. In specific cases, we will recover doubly periodic tableaux.

\begin{definition}
For $\CL$, $\CM$ two $(K,N)$-periodic lattice paths such that $\CL\leq \CM$, a \emph{standard $(K,N)$-periodic skew tableau $\sigma$ of shape $\CM\setminus \CL$} is a standard filling of
\[\CM\setminus \CL :=\left\lbrace (x,y)\in\BZ^2 \mid \CL(y) \leq x < \CM(y) \right\rbrace,\]
such that:
\begin{enumerate}
\item $\sigma$ is $(K,N)$-periodic: for all $(x,y)\in\CM\setminus \CL$, then by definition $(x+K,y-N)\in  \CM\setminus \CL $, and
\[ \sigma(x+K,y-N) = \sigma(x,y). \]
\item if $m$ is the number of boxes of $\CM\setminus \CL$ in each $N$ consecutive lines (or $K$ consecutive columns), then all numbers $1,2 \ldots, m$ appear in $\sigma$.
\end{enumerate}
\end{definition}

Let $ST_{K,N}(\CM\setminus\CL)$ denote the set of $(K,N)$-periodic skew tableaux of shape $\CM\setminus\CL$.

\begin{example}\label{ex_CL_CM}
If $(K,N)=(3,2)$, $\CL=\CL_{(0,0)}$ and $\CM=\CL_{(4,1)}$, then $m=5$.
Here is an example of a $(3,2)$-periodic skew tableau of shape $\CM\setminus \CL$.

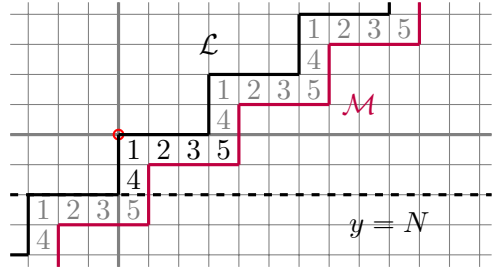
\begin{center}
\begin{tikzpicture}[scale=0.8]
\draw[step=0.5,gray, thin] (-1.8,-2.2) grid (6.2,2.2);
\draw[very thick, gray] (-1.8,0) -- (6.2,0);
\draw[very thick, gray] (0,-2.2) -- (0,2.2);
\draw[red, thick] (0,0) circle (0.8mm);
\draw[very thick] (0,0) -- (0,-1);
\draw[very thick] (-1.5,-1) -- (0,-1);
\draw[very thick] (-1.5,-1) -- (-1.5,-2);
\draw[very thick] (-1.8,-2) -- (-1.5,-2);
\draw[very thick] (0,0) -- (1.5,0);
\draw[very thick] (1.5,1) -- (1.5,0);
\draw[very thick] (1.5,1) -- (3,1);
\draw[very thick] (3,2) -- (3,1);
\draw[very thick] (3,2) -- (4.5,2);
\draw[very thick] (4.5,2) -- (4.5,2.2);
\node at (1.5,1.5) {$\CL$};
\draw[very thick, purple] (0.5,-0.5) -- (0.5,-1.5);
\draw[very thick, purple] (-1,-1.5) -- (0.5,-1.5);
\draw[very thick, purple] (-1,-1.5) -- (-1,-2.2);
\draw[very thick, purple] (0.5,-0.5) -- (2,-0.5);
\draw[very thick, purple] (2,-0.5) -- (2,0.5);
\draw[very thick, purple] (2,0.5) -- (3.5,0.5);
\draw[very thick, purple] (3.5,1.5) -- (3.5,0.5);
\draw[very thick, purple] (3.5,1.5) -- (5,1.5);
\draw[very thick, purple] (5,1.5) -- (5,2.2);
\node[purple] at (4,0.5) {$\CM$};
\draw[very thick, dashed] (-1.8,-1) -- (6.2,-1);
\node at (4.5,-1.5) {$y=N$};
\node[gray] at (-1.25,-1.25) {1};
\node[gray] at (-0.75,-1.25) {2};
\node[gray] at (-0.25,-1.25) {3};
\node[gray] at (0.25,-1.25) {5};
\node[gray] at (-1.25,-1.75) {4};
\node at (0.25,-0.25) {1};
\node at (0.75,-0.25) {2};
\node at (1.25,-0.25) {3};
\node at (1.75,-0.25) {5};
\node at (0.25,-0.75) {4};
\node[gray] at (1.75,0.75) {1};
\node[gray] at (2.25,0.75) {2};
\node[gray] at (2.75,0.75) {3};
\node[gray] at (3.25,0.75) {5};
\node[gray] at (1.75,0.25) {4};
\node[gray] at (3.25,1.75) {1};
\node[gray] at (3.75,1.75) {2};
\node[gray] at (4.25,1.75) {3};
\node[gray] at (4.75,1.75) {5};
\node[gray] at (3.25,1.25) {4};
\end{tikzpicture}
\end{center}
Notice here that we are actually in the case $\CM=\CL[4,1]$, similar to Figure~\ref{labeling_convention} appearing in Section~\ref{sectDPTlattice}. As in that section, we see that we must consider infinite skew tableaux, as not all standard filling of a fundamental domain (the boxes apearing in $N$ consecutive lines for example) extend to standard filling of the entire skew diagram.

For example, the following filling of the skew shape $(4,1)\setminus (0,0) =(4,1)$ does not extend : $\ytableaushort{1234,5}$.
\end{example}

Let $\lambda,\mu\in\CA_{K,N}$, if $\lambda_i\leq \mu_i$, for all $1\leq i\leq N$, then fix $m$
\begin{equation}\label{eq_m}
m= \sum_{i=1}^N\left( \mu_i -\lambda_i\right),
\end{equation}
the number of boxes $\lambda$ has more than $\mu$. 

For any $(K,N)$-periodic skew tableau $\sigma$ of shape $\CL_\mu\setminus \CL_\lambda$, one can associate a chain of $m+1$ dominant weights $\mathbf{u} = (u_0,u_1,\ldots,u_m)$ as follows. Let $u_0=\lambda$, then for each $0\leq \ell\leq m-1$, let $u_{\ell+1} = u_{\ell} + \varepsilon_{i_\ell}$, where $i_\ell$ is the column between 1 and $N$ where the value $\ell$ appears in $\sigma$. Then, from Remark~\ref{rem_lambda_m_mu},
\[u_m = \mu = \lambda + \sum_{\ell=1}^N \varepsilon_{i_\ell}.\]

\begin{example}
If we continue Example~\ref{ex_CL_CM}, with $(K,N)=(3,2)$, $\mu=(4,1)$ and $\lambda=(0,0)$, then the chain $\textbf{u}$ corresponding to this choice of tableau $\sigma$ is the following:
\begin{align*}
u_0=\lambda=(0,0), \quad u_1=(1,0), \quad u_2=(2,0), \\
u_3=(3,0), \quad u_4=(3,1), \quad u_5=(4,1)=\mu.
\end{align*}

\end{example}


\subsection{Quantum groups intertwiners}

We now have the tools to relate the doubly periodic tableaux of Section~\ref{sect_DPT} to intertwining spaces for the quantum group $U_v(\mathfrak{gl}_N)$ at a root of unity. We show here that 
\[ \dim M_{[\mu],[\lambda]}^{(m)} = \# ST_{K,N}(\CL_\mu\setminus\CL_\lambda).\]

In this section, we apply a method similar to that of \cite{Jordan2019Rectangular}, which we fully describe in our context. We show that to each standard $(K,N)$-periodic skew tableau $\sigma$ of shape $\CL_\mu\setminus\CL_\lambda$, one can associate a line $L_\CT$ in $M_{[\mu],[\lambda]}^{(m)}$ as follows.

Recall from the Pieri rule that
\[V\otimes [\lambda] = \sum_{\lambda + \varepsilon_j\in\CA_{K,N}}[\lambda + \varepsilon_j].\]

Hence, for all $j$ such that $\lambda + \varepsilon_j\in\CA_{K,N}$, the $(\lambda + \varepsilon_j)$-isotypic component of $V\otimes [\lambda]$ is of dimension 1 and $\dim \Hom([\lambda + \varepsilon_j], V\otimes [\lambda]) =1$.

Similarly, for $i,j$ such that $\lambda+ \varepsilon_i \in\CA_{K,N}$ and $\lambda + \varepsilon_i + \varepsilon_j \in\CA_{K,N}$, 
\[ \Hom\left([\lambda + \varepsilon_i + \varepsilon_j], \left(V\otimes(V\otimes[\lambda])_{\lambda + \varepsilon_i}\right)\cap\left((V^{\otimes 2}\otimes [\lambda])_{\lambda + \varepsilon_i + \varepsilon_j}\right)\right),\]
is a one-dimensional subspace of $\Hom([\lambda + \varepsilon_i + \varepsilon_j], V^{\otimes 2}\otimes [\lambda])$. It records that we have added first $\varepsilon_i$ and then $\varepsilon_j$. 
Conversely, $\Hom([\lambda + \varepsilon_i + \varepsilon_j], V^{\otimes 2}\otimes [\lambda])$ is of dimension 1 or 2, depending on whether $i=j$, and on whether it is possible to also add first $\varepsilon_j$ then $\varepsilon_i$.

Now, let $\sigma\in ST_{K,N}(\CL_\mu\setminus\CL_\lambda)$, and let $\mathbf{u}_\sigma =(u_0,\ldots, u_m)$ be the associated chain of dominant weights. Then
\begin{equation}\label{def_L_sigma}
L_\sigma:=\Hom\left([\mu], \bigcap_{i=0}^mV^{\otimes m-i}\otimes\left(V^{\otimes i}\otimes [\lambda] \right)_{u_i}\right)
\end{equation}
is a one dimensional subspace of $M_{[\mu],[\lambda]}^{(m)}$.

We obtain the following.

\begin{thm}\label{thm_M_ST}
There is an isomorphism of vectors spaces
\[M_{[\mu],[\lambda]}^{(m)}\simeq \bigoplus_{\sigma\in ST_{K,N}(\CL_\mu\setminus\CL_\lambda)} L_\sigma.\]
In particular, the dimension of $M_{[\mu],[\lambda]}^{(m)}$ is equal to the number of standard $(K,N)$-periodic fillings of the infinite skew tableau $\CL_\mu\setminus\CL_\lambda$.
\end{thm}

We recall the content function of Definiton~\ref{def_content}, and extend it to standard $(K,N)$-periodic skew tableaux. The \emph{content function} $C_\sigma$ of a standard $(K,N)$-periodic skew tableaux $\sigma$ is the function $\BZ \to \BZ/(N+K)\BZ$ such that
\[C_\sigma(i) = x-y, \quad \text{for } (x,y) \in\BZ^2 \text{ such that } \sigma(x,y)=i.\]

\begin{prop}\label{prop_Xi_L}
For all $\sigma \in ST_{K,N}(\CL_\mu\setminus\CL_\lambda)$, the generators $X_i$ of the affine Hecke algebra $\dot{H}_q(m)$ acts on $L_\sigma$ as
\begin{equation}
X_i\cdot L_\sigma = v^{2C_\sigma(i)}, \quad 1\leq i\leq m-1.
\end{equation}
\end{prop}

\begin{proof}
For all $1\leq i\leq m-1$, we write
\[X_i = q^{-i+1}T_{i-1}\cdots T_2T_1 X_1T_1T_2\cdots T_{i-1}. \]
Thus each $X_i$ acts on $M_{[\mu],[\lambda]}^{(m)}$ by $R_{i-1}\cdots R_2R_1 R_0^2R_1R_2\cdots R_{i-1}$. Using the properties of the $\mathcal{R}$-matrix, and \eqref{eq_ThetaUW}, one has
\begin{align*}
X_i \, \mapsto \,& \Id_{V^{\otimes m-i}}\otimes \left( \check{R}_{V^{\otimes i-1}\otimes[\lambda],V}\check{R}_{V,V^{\otimes i-1}\otimes[\lambda]} \right),\\
= \, & \Id_{V^{\otimes m-i}}\otimes \left( \theta_{V^{\otimes i}\otimes [\lambda]}\left(\theta_V\otimes \theta_{V^{\otimes i-1}\otimes [\lambda]}\right)^{-1}\right).
\end{align*}
Hence, using the definition \eqref{def_L_sigma} of $L_\sigma$, as well as Lemma~\ref{lem_theta_lambda}
\[X_i\cdot L_\sigma = v^{\left\langle u_i, u_i + 2\rho \right\rangle -\left\langle u_{i-1}, u_{i-1} + 2\rho \right\rangle -\left\langle \varepsilon_1, \varepsilon_1 + 2\rho \right\rangle},\]
where $u_i$ is the $i$th element of $\mathbf{u}_\sigma = (u_0,u_1,\ldots,u_m)$, the chain of dominant weights associated to $\sigma$.
Writting $u_i= u_{i-1} + \varepsilon_{j_i}$ (the box numbered $i$ is added to the $j_i$th line of $u_{i-1}$), we compute
\begin{multline*}
\left\langle u_i, u_i + 2\rho \right\rangle -\left\langle u_{i-1}, u_{i-1} + 2\rho \right\rangle -\left\langle \varepsilon_1, \varepsilon_1 + 2\rho \right\rangle \\= 2\left( \left\langle  u_{i-1}, \varepsilon_{j_i} \right\rangle  + \left\langle  \rho, \varepsilon_{j_i} \right\rangle - \left\langle  \rho, \varepsilon_{1} \right\rangle\right)
= 2\left((u_{i-1})_{j_i}+1 - j_i)\right).
\end{multline*}
As $(u_{i-1})_{j_i}$ is the size of the $j_i$th row of $u_{i-1}$, $(u_{i-1})_{j_i}+1$ is the $x$ coordinate of the added box $i$ in $u_i$. 

Thus the coordinates of the box numbered $i$ in $\sigma$ are $((u_{i-1})_{j_i}+1,j_i)$ and we have the result.
\end{proof}

As in Section \ref{sect_DPT}, let us fix $(a,b)\in \BZ^2$ such that $m=aN-bK >0$.

For all dominant weights $\lambda\in \CA_{K,N}$, consider 
\[ [\mu] = D^{a}\otimes L^{-b}\otimes [\lambda].\]
Using Lemma~\ref{lem_latt_DL}, we know that $\CL_\mu =\CL_\lambda[a,b]$.
We recover here the doubly periodic tableaux of Section~\ref{sect_DPT}. More precisely, using Corollary~\ref{cor_DPT_CL}, we have the isomorphism
\[ST_{K,N}(\CL_\lambda[a,b]\setminus \CL_\lambda) \simeq \DPT(\CL_\lambda).\]

Moreover, from Lemma~\ref{lem_Delta_m_cells}, $m$ is equal to the number of boxes is the skew diagram $\lambda[a,b]\setminus\lambda$, and it satisfies \eqref{eq_m}. 
Thus Theorem~\ref{thm_M_ST} gives in this case the following.

\begin{cor}\label{cor_Hom_DPT}
For all dominant weights $\lambda\in \CA_{K,N}$, and $(a,b)\in \BZ^2$ such that $m=aN-bK >0$, 
\begin{equation}
\dim \left(\Hom_{\CC_\CF}\left(D^{a}\otimes L^{-b}\otimes [\lambda], V^{\otimes m}\otimes [\lambda]  \right) \right) = \# \DPT(\CL_\lambda).
\end{equation}
\end{cor}






\begin{example}
Let us continue Example~\ref{ex_lambda_L}, for which $(K,N)=(7,4)$, and the considered weight is $\lambda=(4,1,1,-1)$. Fix $(a,b)= (3,1)$, such that $m=aN-bK = 5>0$.
We have drawn below $\CL_\lambda$, $\CL_\lambda[3,1]$ and $\Delta' = \CL_\lambda[3,1]\setminus \CL_\lambda$.

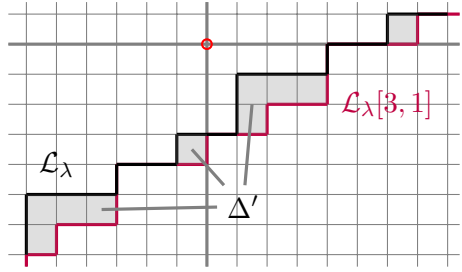
\begin{center}
\begin{tikzpicture}[scale=0.8]
\draw[step=0.5,gray, thin] (-1.3,-2.2) grid (6.2,2.2);
\draw[very thick, gray] (-1.3,1.5) -- (6.2,1.5);
\draw[very thick, gray] (2,-2.2) -- (2,2.2);
\draw[red, thick] (2,1.5) circle (0.8mm);
\draw[very thick, purple] (2,0) -- (2,-0.5);
\draw[very thick, purple] (0.5,-0.5) -- (2,-0.5);
\draw[very thick, purple] (0.5,-0.5) -- (0.5,-1.5);
\draw[very thick, purple] (-0.5,-1.5) -- (0.5,-1.5);
\draw[very thick, purple] (-0.5,-1.5) -- (-0.5,-2);
\draw[very thick, purple] (-1,-2) -- (-0.5,-2);
\draw[very thick, purple] (-1,-2) -- (-1,-2.2);
\draw[very thick, purple] (2,0) -- (3,0);
\draw[very thick, purple] (3,0.5) -- (3,0);
\draw[very thick, purple] (3,0.5) -- (4,0.5);
\draw[very thick, purple] (4,1.5) -- (4,0.5);
\draw[very thick, purple] (4,1.5) -- (5.5,1.5);
\draw[very thick, purple] (5.5,2) -- (5.5,1.5);
\draw[very thick, purple] (5.5,2) -- (6.2,2);
\node[purple] at (5,0.5) {$\CL_\lambda[3,1]$};
\draw[very thick] (-1,-1) -- (-1, -2);
\draw[very thick] (-1,-1) -- (0.5, -1);
\draw[very thick] (0.5,-0.5) -- (0.5, -1);
\draw[very thick] (0.5,-0.5) -- (1.5, -0.5);
\draw[very thick] (1.5,0) -- (1.5, -0.5);
\draw[very thick] (1.5,0) -- (2.5, 0);
\draw[very thick] (2.5,1) -- (2.5, 0);
\draw[very thick] (2.5,1) -- (4, 1);
\draw[very thick] (4,1.5) -- (4, 1);
\draw[very thick] (5,1.5) -- (4, 1.5);
\draw[very thick] (5,1.5) -- (5, 2);
\draw[very thick] (6,2) -- (5, 2);
\node at (-0.5,-0.5) {$\CL_\lambda$};
\fill[gray,nearly transparent] (-1,-2) -- (-1,-1) -- (0.5,-1) -- (0.5,-1.5) -- (-0.5, -1.5) -- (-0.5, -2) -- (-1,-2);
\fill[gray,nearly transparent] (2.5,0) -- (2.5,1) -- (4,1) -- (4,0.5) -- (3, 0.5) -- (3,0) -- (2.5,0);
\fill[gray,nearly transparent] (1.5,0) rectangle (2,-0.5);
\fill[gray,nearly transparent] (5,2) rectangle (5.5,1.5);
\node (D) at (2.6,-1.2) {$\Delta'$};
\draw[very thick,gray] (1.75,-0.25) -- (D);
\draw[very thick,gray] (2.75,0.5) -- (D);
\draw[very thick,gray] (0.25,-1.25) -- (D);
\end{tikzpicture}
\end{center}
The number of standard $(7,4)$-periodic fillings of the infinite skew diagram $\Delta'$ is 15, the number of standard skew tableaux of shape $\ytableausetup{boxsize=0.9em}\ydiagram{2+3,2+1,1}$ :
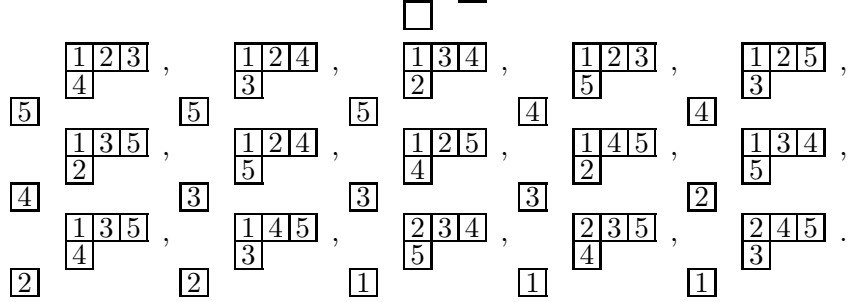
\begin{center}
\ytableausetup{boxsize=0.9em}
    \ytableaushort{
    \none\none123,\none\none4,5}
    \, ,
    \ytableaushort{
    \none\none124,\none\none3,5}
    \, ,
    \ytableaushort{
    \none\none134,\none\none2,5}
    \, ,
    \ytableaushort{
    \none\none123,\none\none5,4}
    \, ,
    \ytableaushort{
    \none\none125,\none\none3,4} \, ,
    \\
    \ytableaushort{
    \none\none135,\none\none2,4}
    \, ,
    \ytableaushort{
    \none\none124,\none\none5,3}
    \, ,
    \ytableaushort{
    \none\none125,\none\none4,3}
    \, ,
    \ytableaushort{
    \none\none145,\none\none2,3}
    \, ,
    \ytableaushort{
    \none\none134,\none\none5,2} \, ,
    \\
    \ytableaushort{
    \none\none135,\none\none4,2}
    \, ,
    \ytableaushort{
    \none\none145,\none\none3,2} 
    \, ,
    \ytableaushort{
    \none\none234,\none\none5,1}
    \, ,
    \ytableaushort{
    \none\none235,\none\none4,1}
    \, ,
    \ytableaushort{
    \none\none245,\none\none3,1} \, .
\end{center}
And 15 is also the dimension of the space
\[\Hom_{\CC_\CF}\left(D^{3}\otimes L^{-1}\otimes [\lambda], V^{\otimes 5}\otimes  [\lambda]  \right).\]
\end{example}


For all dominant weights $\lambda\in \CA_{K,N}$, we introduce the notation
\[M_\lambda:=\Hom_{\CC_\CF}\left(D^{a}\otimes L^{-b}\otimes [\lambda], V^{\otimes m}\otimes [\lambda]  \right).\]
Then, Corollary~\ref{cor_Hom_DPT} gives an isomorphism of vector spaces
\begin{equation}\label{iso_M_W}
M_\lambda \simeq W_{\CL_\lambda}.
\end{equation}

\begin{prop}
The isomorphism \eqref{iso_M_W} is an isomorphism of $\dot{H}_q(m)$-modules.
\end{prop}

\begin{proof}
The $\dot{H}_q(m)$-module structure of $W_{\CL_\lambda}$ is given in Proposition~\ref{prop_AHA_act}.
Using Proposition~\ref{prop_Xi_L}, we know that the $\dot{H}_q(m)$-modules $M_\lambda $ and $W_{\CL_\lambda}$ are both $X$-semisimple and have exactly the same weights and weight spaces.  Thus they are isomorphic as $\dot{H}_q(m)$-modules. 
\end{proof}

  \subsection{DAHA action}

From the results of Section~\ref{sect_DAHA_DPT}, we know that the direct sum
\begin{equation}\label{W_equal_sum_M}
W_{(K,N,a,b)} = \bigoplus_{\lambda\in\CA_{K,N}} W_{\CL_\lambda} \simeq \bigoplus_{\lambda\in\CA_{K,N}} M_\lambda,
\end{equation}
is a module for the Double Affine Hecke Algebra $\ddot{H}_{q,t}(m)$. 

The action of the generators of $\ddot{H}_{q,t}(m)$ can be given explicitly on the intertwiner spaces, using the ribbon category structure of the fusion category. However, in order to state this, one needs some additional framework in ribbon calculus, detailed in the next section.

\section{Ribbon calculus}\label{sect_ribbon}

\subsection{Preliminaries}

We recall here standard facts about ribbon calculus, see for example \cite{Turaev2016Quantum} for detailed exposition.

Let $M$ be an oriented compact $3$-manifold with boundary. Assume we are given a collection of framed oriented marked points\footnote{A framing of a point is the choice of a tangent vector. An orientation of a point is the choice whether the point is \emph{input} or \emph{output}.} $p_1, p_2,\ldots$ on $\partial M$.

\begin{definition}
Let $\CC$ be a ribbon category and. A \emph{ribbon graph} in $M$ is a collection of bands and coupons embedded in $M$. Each band is a framed oriented arc or a loop. Each coupon is a small bigon with one side the input side and the other side the output side. Bands are not allowed to meet. Bands can meet the boundaries of the coupons and $\partial M$ only at their boundaries. Incoming resp. outgoing bands are allowed to meet a coupon only at the input resp. output side, and the orientations and the framings are required to match. Bands are only allowed to meet $\partial M$ at the marked points, the orientations and the framings are required to match, and the intersection has to be transversal.
\end{definition}

\begin{definition}
A \emph{coloring of the boundary} $\partial M$ is the choice of an object of $\CC$ for each marked point.
\end{definition}

\begin{definition}
A \emph{coloring} of a ribbon graph is a labeling of each band by an object of $\CC$, and each coupon by a morphism of $\CC$ which goes from the (ordered) tensor product of the objects associated with the incoming bands to the tensor product of the objects associated with the outgoing bands. Given a coloring of the boundary, a coloring of a ribbon graph is \emph{compatible} if the object associated to a band that meets the boundary is the same object as associated to the corresponding marked point.
\end{definition}

Ribbon calculus associates a morphism in $\CC$ to any ribbon graph in the $3$-ball and it is a standard fact that this association is well-defined.

\begin{definition}
Let us fix a coloring of the boundary. The associated \emph{ribbon skein module} is the vector space of formal linear combinations of compatible colored ribbon graphs in $M$ modulo the equivalence relation generated by the following moves:
\begin{enumerate}
\item A ribbon graph can be replaced by an isotopic ribbon graph, where the isotopy is required to be relative to $\partial M$.
\item For each ribbon graph $\Gamma$ with a coupon $c$ if the morphism in $c$ is a linear combination of morphisms $\alpha_1 f_1 + \alpha_2 f_2 +\cdots$, then we are allowed to replace $\Gamma$ by the corresponding linear combination of ribbon graphs $\alpha_1 \Gamma_1+\alpha_2\Gamma_2+\cdots$, where $\Gamma_i$ is the same as $\Gamma$, but has $f_i$ in $c$.
\item If the morphism at a coupon $c$ of a ribbon graph $\Gamma$ can be represented by a colored ribbon graph $\Gamma_0$ in the $3$-ball using standard ribbon calculus, then we can replace a small neighborhood of $c$ by $\Gamma_0$.
\end{enumerate}
\end{definition}

The following must be well-known to experts, but seems to be missing from the literature. See however Brochier \cite{Brochier2018Integrating}, where a similar idea has been worked out from the factorization homology point of view.

\begin{lem}\label{lem_torus_skein_module}
Let $M$ be the solid torus 
\[
M = S^1\times I\times I = I\times I \times I/(0,x,y)\sim (1,x,y) \;(x,y\in I)
\]
with a single marked point $p=(0,0,0)$. Choose a coloring of $\partial M$ by associating to $p$ an object $U$ of $\CC$. The corresponding ribbon skein module can be explicitly presented as follows:
\[
\bigoplus_{V\in\CC} \Hom(V, U\otimes V)\;/\; \{(\Id_U\otimes \varphi) \circ \psi - \psi\circ \varphi \;|\; \varphi\in\Hom(V,V'),\;\psi\in\Hom(V', U\otimes V)\}.
\]
\end{lem}
\begin{proof}
Denote the skein module by $S$ and the quotient above by $S'$.

To construct the inverse map we introduce the operation of slicing a colored ribbon graph in $M$, represented in Figure~\ref{fig_slice}. For any $s_0\in (0,1)$ we call a colored ribbon graph $\Gamma$ generic $D_{s_0}:=\{s_0\}\times I \times I$ does not meet the coupons of $\Gamma$ and meets the bands transversally. Fix any $s_0\in(0,1)$. We can always isotope $\Gamma$ in a small neighborhood of $D_{s_0}$ to make it generic. We say a band of $\Gamma$ is correctly oriented if crossing $D_{s_0}$ it goes in the positive $s$-direction. Using (3) we modify $\Gamma$ in a small neighborhood of $D_{s_0}$ by splitting each incorrectly oriented band into three to make sure that all bands intersecting $D_{s_0}$ are correctly oriented.


Consider the map $\pi_{s_0}:I\times I \times I\to M$ given by
\[
\pi_{s_0}(s,x,y) = (s+s_0 \pmod \BZ,\quad x,\quad y).
\]
Pulling back $\Gamma$ via $\pi$ produces a colored ribbon graph in the cube $\pi^*\Gamma\in I\times I\times I$. We choose a homeomorphism of the boundary of the cube to the boundary of the standard cube to make the face $\{0\}\times I^2$ the input face, and the union $I\times \{0\}\times I \cup \{1\}\times I^2$ the output face. Further applying a homeomorphism on the input and output faces we can make sure that the objects appear in order $V_1, V_2,\cdots$ on the input face, and in order $U, V_1, V_2,\cdots$ on the output face. Let $V=V_1\otimes V_2\otimes \cdots$. By ribbon calculus, $\pi^* \Gamma$ produces a morphism
\[
V \to U\otimes V.
\]
We denote the result by $\slice_{s_0} \Gamma\in S'$ and call it the slice of $\Gamma$ at $s_0$.

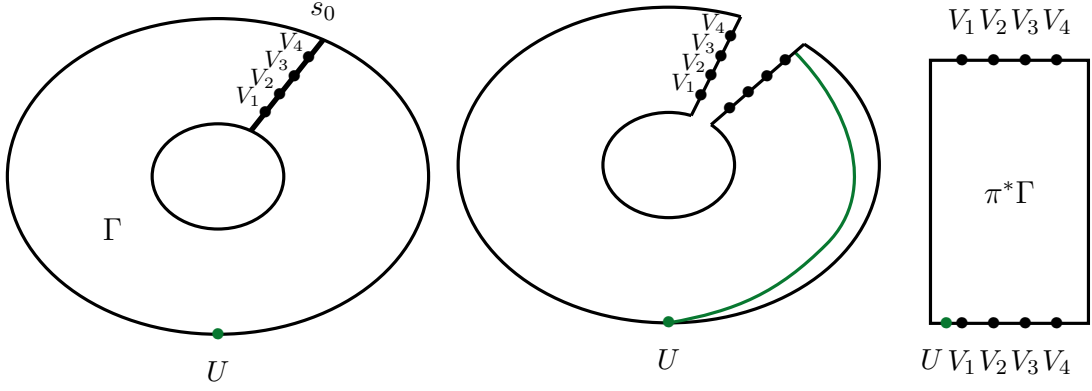
\begin{figure}\caption{Slicing of $\Gamma$ at $s_0$.}
\begin{tikzpicture}[scale=0.8]
\node at (0,0) {\begin{tikzpicture}[scale=0.7]
\draw[very thick] (0,0) ellipse (4 and 3);
\draw[very thick] (0,0) ellipse (1.25 and 1);
\node at (-2,-1) {\large$\Gamma$};
\node[label=below:{$U$},forest] (U) at (270:4 and 3) {$\bullet$};
\node[label={$s_0$}] (s) at (60:4 and 3) {};
\coordinate (s') at (60:1.25 and 1);
\node [label={[xshift=-0.6em, yshift=-0.7em]\footnotesize$V_1$}] at ($ (s')!1.0/5!(s) $) {$\bullet$};
\node [label={[xshift=-0.6em, yshift=-0.7em]\footnotesize$V_2$}] at ($ (s')!2.0/5!(s) $) {$\bullet$};
\node [label={[xshift=-0.6em, yshift=-0.7em]\footnotesize$V_3$}] at ($ (s')!3.0/5!(s) $) {$\bullet$};
\node [label={[xshift=-0.6em, yshift=-0.7em]\footnotesize$V_4$}] at ($ (s')!4.0/5!(s) $) {$\bullet$};
\coordinate (U') at ($ (s')!1.0/10!(s) $);
\draw[line width=2pt] (60:4 and 3) to (60:1.25 and 1);
\end{tikzpicture}};

\node at (7.5,0) {\begin{tikzpicture}[scale=0.7]
\draw[very thick] (0,0) ellipse (4 and 3);
\draw[very thick] (0,0) ellipse (1.25 and 1);
\draw[line width = 3pt,white] (50:4cm and 3cm) arc   (50:70:4cm and 3cm);
\draw[line width = 3pt,white] (50:1.25cm and 1cm) arc   (50:70:1.25cm and 1cm);
\node[label=below:{$U$},forest] (U) at (270:4 and 3) {$\bullet$};
\coordinate (s1) at (50:4 and 3);
\coordinate (s1') at (50:1.25 and 1);
\coordinate (s2) at (70:4 and 3);
\coordinate (s2') at (70:1.25 and 1);
\node [label={[xshift=-0.6em, yshift=-0.8em]\footnotesize$V_1$}] at ($ (s2')!1.0/5!(s2) $) {$\bullet$};
\node [label={[xshift=-0.6em, yshift=-0.8em]\footnotesize$V_2$}] at ($ (s2')!2.0/5!(s2) $) {$\bullet$};
\node [label={[xshift=-0.6em, yshift=-0.8em]\footnotesize$V_3$}] at ($ (s2')!3.0/5!(s2) $) {$\bullet$};
\node [label={[xshift=-0.6em, yshift=-0.8em]\footnotesize$V_4$}] at ($ (s2')!4.0/5!(s2) $) {$\bullet$};
\node  at ($ (s1)!1.0/5!(s1') $) {$\bullet$};
\node  at ($ (s1)!2.0/5!(s1') $) {$\bullet$};
\node  at ($ (s1)!3.0/5!(s1') $) {$\bullet$};
\node at ($ (s1)!4.0/5!(s1') $) {$\bullet$};
\coordinate (U') at ($ (s1)!1.0/10!(s1') $);
\draw[very thick, forest] (270:4 and 3) to[out = 10, in =225] (3,-1.5)  to[out = 45, in =315] (U');
\draw[very thick] (50:4cm and 3cm) to (50:1.25cm and 1cm);
\draw[very thick] (70:4cm and 3cm) to (70:1.25cm and 1cm);
\end{tikzpicture}};

\node at (13,0) {\begin{tikzpicture}[scale=0.7]
\coordinate (A) at (0,0);
\coordinate (B) at (3,0);
\coordinate (C) at (0,-5);
\coordinate (D) at (3,-5);
\node at (1.5,-2.5) {\large$\pi^* \Gamma$};
\draw[very thick] (A) rectangle (D);
\node [label={$V_1$}] at ($ (A)!1.0/5!(B) $) {$\bullet$};
\node [label={$V_2$}] at ($ (A)!2.0/5!(B) $) {$\bullet$};
\node [label={$V_3$}] at ($ (A)!3.0/5!(B) $) {$\bullet$};
\node [label={$V_4$}] at ($ (A)!4.0/5!(B) $) {$\bullet$};
\node [label=below:{$V_1$}] at ($ (C)!1.0/5!(D) $) {$\bullet$};
\node [label=below:{$V_2$}] at ($ (C)!2.0/5!(D) $) {$\bullet$};
\node [label=below:{$V_3$}] at ($ (C)!3.0/5!(D) $) {$\bullet$};
\node [label=below:{$V_4$}] at ($ (C)!4.0/5!(D) $) {$\bullet$};
\node [label={[xshift=-0.5em, yshift=-2.5em]$U$},forest] at ($ (C)!1.0/10!(D) $) {$\bullet$};
\end{tikzpicture}};
\end{tikzpicture}
\end{figure}\label{fig_slice}

Given two points $0<s_1<s_2<1$ and a colored ribbon graph $\Gamma$ we can perform a slice at $s_1$ and a slice at $s_2$. The slicing operation consists of two steps: first we do a local modification, and the take the pullback. Denote the result of local modification of $\Gamma$ during slicing at $s_i$ by $\Gamma_i$, and the result of applying both modifications by $\Gamma'$. We have
\[
\slice_{s_1} \Gamma = \slice_{s_1} \Gamma_1 = \slice_{s_1} \Gamma'
\]
because local modifications in a neighborhood of $D_{s_2}$ can be repeated for $\pi_{s_1}^{\Gamma_1}$ provided the neighborhood of $D_{s_2}$ does not intersect the neighborhood of $D_{s_1}$. Now the disc $\pi_{s_1}^* D_{s_2}$ splits the cube into two, and correspondingly the morphism $\slice_{s_1} \Gamma$ becomes decomposed into a product of the form
\[
\slice_{s_1} \Gamma' = \psi\circ \varphi.
\]
On the other hand, for $\slice_{s_2} \Gamma = \slice_{s_2} \Gamma'$ we obtain
\[
\slice_{s_2} \Gamma' = (\Id_U\otimes \varphi) \circ \psi.
\]
Thus we have $\slice_{s_1} \Gamma = \slice_{s_2} \Gamma$ (see Figure~\ref{fig_comp_slices}), so slicing does not depend on the choice of $s_0\in (0,1)$.

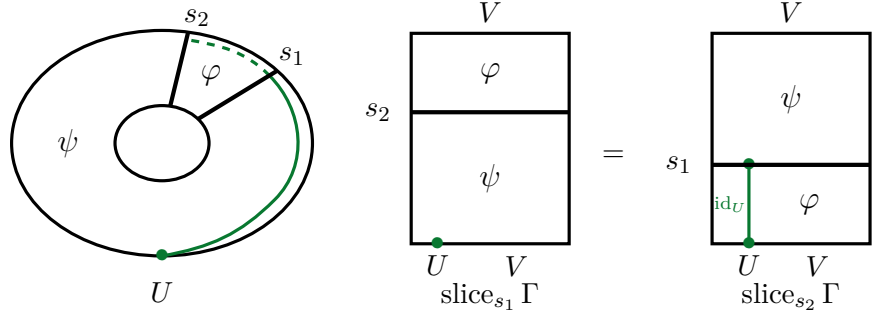
\begin{figure}
\begin{tikzpicture}
\node at (0,0) {
\begin{tikzpicture}[scale=0.5]
\draw[very thick] (0,0) ellipse (4 and 3);
\draw[very thick] (0,0) ellipse (1.25 and 1);
\node[label={[xshift=0.5em,yshift=-0.5em]$s_1$}] (s1) at (40:4 and 3) {};
\node[label={[xshift=0.3em,yshift=-0.5em]$s_2$}] (s2) at (80:4 and 3) {};
\node[label=below:{$U$},forest] (U) at (270:4 and 3) {$\bullet$};
\coordinate (s1') at (40:1.25 and 1);
\coordinate (s2') at (80:1.25 and 1);
\coordinate (U1) at ($ (s1)!1.0/10!(s1') $);
\coordinate (U2) at ($ (s2)!1.0/10!(s2') $);
\draw[very thick, forest] (270:4 and 3) to[out = 10, in =225] (3,-1.5)  to[out = 45, in =315] (U1);
\draw[very thick, forest, dashed] (U1) to[out = 135, in=350] (U2);
\draw[ultra thick] (40:4cm and 3cm) to (40:1.25cm and 1cm);
\draw[ultra thick] (80:4cm and 3cm) to (80:1.25cm and 1cm);
\node at (1.3,1.8) {\large$\varphi$};
\node at (-2.5,0) {\large$\psi$};
\end{tikzpicture}};

\node at (4,0) {\begin{tikzpicture}[scale=0.7]
\coordinate (A) at (0,0);
\coordinate (B) at (3,0);
\coordinate (C) at (0,-4);
\coordinate (D) at (3,-4);
\draw[very thick] (A) rectangle (D);
\node[label=left:{$s_2$}] at (0,-1.5) (s2) {};
\draw[ultra thick] (0,-1.5) -- (3,-1.5);
\node at (1.5,-0.75) {\large$\varphi$};
\node at (1.5,-2.75) {\large$\psi$};
\node[label={[yshift=-2em]$U$}, forest] at (0.5,-4) {$\bullet$};
\node[label={[yshift=-1.8em]$V$}] at (2,-4) {};
\node[label={[yshift=-0.3em]$V$}] at (1.5,0) {};
\node at (1.5,-5) {$\slice_{s_1}\Gamma$}; 
\end{tikzpicture}};

\node at (8,0) {\begin{tikzpicture}[scale=0.7]
\coordinate (A) at (0,0);
\coordinate (B) at (3,0);
\coordinate (C) at (0,-4);
\coordinate (D) at (3,-4);
\draw[very thick] (A) rectangle (D);
\node[label=left:{$s_1$}] at (0,-2.5) (s2) {};
\node at (1.85,-3.25) {\large$\varphi$};
\node at (1.5,-1.25) {\large$\psi$};
\node[label={[yshift=-2em]$U$}, forest] at (0.7,-4) {$\bullet$};
\node[ forest] at (0.7,-2.5) {$\bullet$};
\draw [very thick, forest] (0.7,-4) to (0.7,-2.5);
\node at (0.35,-3.25) {\color{forest}\tiny$\id_U$};
\node[label={[yshift=-1.8em]$V$}] at (2,-4) {};
\node[label={[yshift=-0.3em]$V$}] at (1.5,0) {};
\node at (1.5,-5) {$\slice_{s_2}\Gamma$}; 
\draw[ultra thick] (0,-2.5) -- (3,-2.5);
\end{tikzpicture}};
\node at (6,0) {=};
\end{tikzpicture}
\caption{Comparing two slices.}
\end{figure}\label{fig_comp_slices}

In order to show that slicing induces a well-defined map $S\to S'$ we check that the relations (1)--(3) are respected. Decomposing any isotopy is a composition of small isotopies we can make sure that all transformations (1)--(3) affect only a small ball $B\subset M$. We can always choose $s_0$ in such a way that the neighborhood of $D_{s_0}$ where we do the local modification does not intersect $B$. The same transformations can be repeated in  $\pi_{s_0}^{-1}(B)$, and we obtain that the result of $\slice_{s_0}$ as an element of $\Hom(V, U\otimes V)$ remains invariant.

So we have a well-defined map $S\to S'$. Any morphism $\psi:V\to U\otimes V$ can be represented by a coupon labeled by $\psi$, and two endpoints can then be connected inside $M$. It can be checked that this induces a well-defined map $S' \to S$, which automatically satisfies that $S'\to S \to S'$ is the identity, hence $S'\to S$ is injective. To show that $S\to S'\to S$ is the identity, we start with a colored ribbon graph $\Gamma$. Applying the local modifications does not change its class in $S$. Applying transformation (3) inside $M\setminus D_{s_0}$ we can replace $\Gamma$ by a single coupon, and the parallel bands crossing $D_{s_0}$ can be replaced by a single band labeled by the tensor product. So we arrive at a ribbon graph equivalent to $\Gamma$ which is in the image of $S'\to S$. Thus $S'\to S$ is also surjective.
\end{proof}

\begin{lem}\label{lem_skein_sum}
If the category $\CC$ is additionally semi-simple, then for all objects $U$, the ribbon skein module $S_U$ is isomorphic to
\[\bigoplus_{I\in\Irr(\CC)}\Hom(I,U\otimes I).\]
\end{lem}

\begin{proof}
Fix $W$ an object in the category $\CC$. Since $\CC$ is semi-simple, $W$ can be written as a finite sum of irreducible objects $W=\bigoplus_{i=1}^r W_i$. For $1\leq i\leq r$, consider the projections and inclusion maps:
\begin{align*}
\pi_i : W &\twoheadrightarrow W_i,\\
\iota_i : W_i & \hookrightarrow W.
\end{align*}
Let $f\in\Hom(W,U\otimes W)$, considered in the ribbon skein module $S_U$. Then using $\Id_W = \sum_{i=1}^r\iota_i \circ \pi_i$, one can write,  $f = \sum_{i=1}^r f \circ \iota_i  \circ \pi_i$. Moreover, for all $1\leq i\leq r$, 
\[f\circ \iota_i \circ \pi_i \sim (\Id_U\otimes \pi_i)\circ(f \circ \iota_i)=:f_i \quad \in \Hom(W_i,U\otimes W_i).\]
Hence the morphism
\[
S_U  \to \bigoplus_{I\in\Irr(\CC)}\Hom(I,U\otimes I),\quad
f  \mapsto \sum_{i=1}^r f_i,
\]
is well defined, and clearly injective. It is also surjective, as for all $I\in\Irr(\CC)$, and $f_I \in \Hom(I,U\otimes I)$, one can consider $f_I$ as a morphism in $S_U$.
\end{proof}

\subsection{DAHA action on quantum group intertwininers}

Consider again the solid torus $M$, with one marked point $p$ on the boundary $\partial M$. The fusion category $\CC_\CF$ is a semi-simple ribbon category, thus we can use the results of the previous section.

Let us color the point $p$ by the representation $U = D^{-a}\otimes L^b\otimes V^{\otimes m}$. Then using Lemma~\ref{lem_skein_sum}, the ribbon skein module $S_U$ is isomorphic to the direct sum
\begin{equation}\label{eq:S_U}
S_U \simeq \bigoplus_{\lambda\in\CA_{K,N}} \Hom([\lambda],U\otimes [\lambda]) \simeq \bigoplus_{\lambda\in\CA_{K,N}} M_\lambda \simeq W_{(K,N,a,b)},
\end{equation}
by \eqref{W_equal_sum_M}.

From the proof of Lemma~\ref{lem_torus_skein_module}, for each module $W$, each morphism $f :W \to U\otimes W$, seen as a ribbon graph as in Section~\ref{sect_ribbon_Uq}, can be embedded in $M$, as in Figure~\ref{fig_hom_torus}.

\begin{figure}\label{fig_hom_torus}
\begin{tikzpicture}
\draw[thick] (0,0) ellipse (5 and 3);
\draw[thick] (0,0) ellipse (1.25 and 1);
\draw[fill=lightgray, very thick] (-4,0.5) rectangle (-1.5,-0.5);
\node[scale=1.5] at (-2.75,0) {$f$};
\draw[line width = 3pt, purple] (-2,0.5) to[out=90, in = 180]  (0,2) to[out=0,in = 90] (2.5,0) to[out=270,in=0] (0,-2) to[out=180,in=270] (-2,-0.5); 
\node[purple, scale=1.5] at (3, 0) {$W$};
\draw[line width=2pt, blue] (214:5 and 3) to[out=34,in=270] (-3.75,-0.5)  ;
\draw[line width=2pt] (217:5 and 3) to[out=37,in=270] (-3.5,-0.5)  ;
\draw[line width=2pt] (220:5 and 3) to[out=40,in=270] (-3.25,-0.5)  ;
\draw[line width=2pt] (226:5 and 3) to[out=46,in=270] (-2.75,-0.5)  ;
\draw[line width=2pt] (229:5 and 3) to[out=49,in=270] (-2.5,-0.5)  ;
\node[xshift=-0.15cm, yshift=-0.2cm] at (229:5 and 3) {$1$};
\node[xshift=-0.15cm, yshift=-0.2cm] at (226:5 and 3) {$2$};
\node[xshift=-0.15cm, yshift=-0.2cm] at (217:5 and 3) {$m$};
\node at (-3.05,-1) {...};
\node[blue] at (-5,-1.75) {$D^{-a}L^b$};
\draw[line width = 3pt,forest] (212:5cm and 3cm) arc   (212:231:5cm and 3cm);
\node[forest, scale=1.5] at (-3,-2.75) {$U$};
\end{tikzpicture}
\caption{Embedding of $\Hom(W,U\otimes W)$ into the solid torus.}
\end{figure}
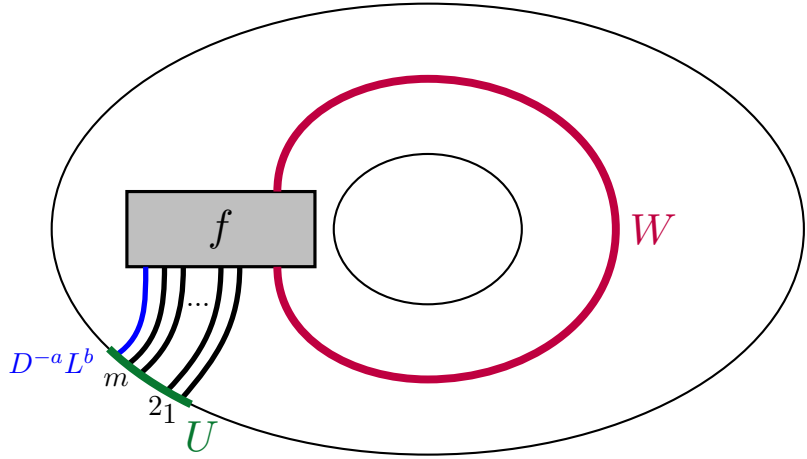

Then one can define an action of the generators $\pi^{\pm 1}$ of the DAHA $\ddot{H}_{q,t}(m)$ on the intertwining spaces, using ribbon calculus. For all $W\in \CF_v$, and $f\in\Hom(W, V^{\otimes m}\otimes D^{-a}\otimes L^b \otimes W)$, then $\pi f$ and $\pi^{-1}f$ are given as follows.

\begin{minipage}{0.45\textwidth}
\begin{equation}\label{act_pi}
\begin{tikzpicture}[scale=1.2]
\draw[fill=lightgray, very thick] (0,0) rectangle (3,1);
\node[scale=1.5] at (1.5,0.5) {$f$};
\draw[line width=2pt, purple] (2.5,0) --  (2.5,-2);
\draw[line width=2pt, purple] (2.5,1) --  (2.5,2.25);
\draw[line width=2pt] (1.5,0) --  (1.5,-0.75)to[out=270,in=90]  (1.25,-1.5) to (1.25,-2);
\draw[line width=2pt] (1.25,0) --  (1.25,-0.75) to[out=270,in=90]  (1,-1.5) to (1,-2);
\draw[line width=2pt] (0.75,0) --  (0.75,-0.75) to[out=270,in=90]  (0.5,-1.5) to (0.5,-2);
\draw[line width=2pt] (0.5,0) --  (0.5,-0.5);
\begin{knot}[clip width =4,ignore endpoint intersections=false]
\strand[line width=2pt] (0.5,-0.5) to[out=270,in=0]  (0,-1);
\strand[line width=2pt, blue] (0.25,0) to (0.25,-2);
\end{knot}
\draw[line width=2pt] (0,-1) to[out=180,in=270]  (-0.5,-0.5) to (-0.5,0.75) to[out=90,in=180]  (1,1.5)  to[out=0,in=270]  (2.25,1.75) to (2.25,2.25);
\draw[line width=2pt] (1.5,-2) to (1.5,-1.5) to[out=90,in=180] (2.05,-1) to[out=0,in=90]  (2.25,-1.25) to (2.25,-2);
\draw [purple,thick,decorate,decoration={brace,amplitude=4pt, mirror}] (2.15,-2.05) -- (2.6,-2.05);
\node[purple] at (2.4,-2.4) {$V^{\ast}\otimes W$};

\node at (1,-0.5) {$\cdots$};
\node at (0.75,-1.5) {$\cdots$};

\node[below, blue] at (-0.1,-2) {$D^{-a}L^b$};

\node[below] at (1.5,-2) {$1$};
\node[below] at (1.25,-2) {$2$};
\node[below] at (1,-2) {$3$};
\node[below] at (0.6,-2) {$m$};

\node[scale=1.5] at (-1.5,0) {$\pi \cdot f=$};
\draw [purple,thick,decorate,decoration={brace,amplitude=4pt}] (2.15,2.3) -- (2.6,2.3);
\node[purple] at (2.4,2.7) {$V^{\ast}\otimes W$};
\end{tikzpicture}
\end{equation}
\end{minipage}
\begin{minipage}{0.45\textwidth}
\begin{equation}\label{act_pi-1}
\begin{tikzpicture}[scale=1.2]
\draw[fill=lightgray, very thick] (0,0) rectangle (3,1);
\node[scale=1.5] at (1.5,0.5) {$f$};
\draw[line width=2pt, purple] (2.5,0) --  (2.5,-2);
\draw[line width=2pt, purple] (2.5,1) --  (2.5,2.25);
\draw[line width=2pt] (1.5,0) --  (1.5,-0.75) to[out=270,in=180] (2,-1.25) to[out=0,in=90]  (2.25,-1.5) to (2.25,-2);
\draw[line width=2pt] (1.25,0) --  (1.25,-0.75) to[out=270,in=90]  (1.5,-1.5) to (1.5,-2);
\draw[line width=2pt] (1,0) --  (1,-0.75) to[out=270,in=90]  (1.25,-1.5) to (1.25,-2);
\draw[line width=2pt] (0.5,0) --  (0.5,-0.75) to[out=270,in=90]  (0.75,-1.5) to (0.75,-2);
\begin{knot}[clip width =4,ignore endpoint intersections=false]
\strand[line width=2pt] (0,-1) to[out=0,in=90]  (0.5,-1.5) to (0.5,-2);
\strand[line width=2pt, blue] (0.25,0) to (0.25,-2);
\end{knot}
\draw[line width=2pt] (0,-1) to[out=180,in=270]  (-0.5,-0.5) to (-0.5,0.75) to[out=90,in=180]  (1,1.5)  to[out=0,in=270]  (2.25,1.75) to (2.25,2.25);
\draw [purple,thick,decorate,decoration={brace,amplitude=4pt, mirror}] (2.15,-2.05) -- (2.6,-2.05);
\node[purple] at (2.45,-2.4) {$V\otimes W$};
\draw[very thick] (1.5,-1.5)  -- (1.5,-2);

\node at (1,-1.5) {$\cdots$};
\node at (0.75,-0.5) {$\cdots$};

\node[blue, below] at (-0.2,-2) {$ D^{-a}L^b$};

\node[below] at (1.5,-2) {$1$};
\node[below] at (1.25,-2) {$2$};
\node[below] at (0.5,-2) {$m$};

\node[scale=1.5] at (-1.7,0) {$\pi^{-1} \cdot f=$};
\draw [purple,thick,decorate,decoration={brace,amplitude=4pt}] (2.15,2.3) -- (2.6,2.3);
\node[purple] at (2.45,2.7) {$V\otimes W$};
\end{tikzpicture}
\end{equation}
\end{minipage}

\begin{remark}
Note here that the top and bottom double strands $V^\ast\otimes W$ and $V\otimes W$ actually go around the core of the torus and meet each other.
\end{remark}

\begin{prop}
The action of $\pi^{\pm 1}$ given in \eqref{act_pi} and \eqref{act_pi-1} extends the $\dot{H}_{q}(m)$-module structure on $S_U$ defined in Proposition~\ref{prop_AHA_ribbon} to a $\ddot{H}_{q,t}(m)$-module structure. 
\end{prop}

\begin{proof}
The action of $T_0$ is given by the relation~\eqref{rel_piTi} $T_0 = \pi T_{m-1}\pi^{-1}$ of Definition~\ref{def_DAHA}. Then, in order to check that all the defining relations of the Double Affine Hecke Algebra $\ddot{H}_{q,t}(m)$ are satisfied, one has to check the following:
\begin{enumerate}
\item\label{pi_T_i} $\pi T_i = T_{i+1}\pi$, for all $i\geq 1$,
\item\label{pi_X_1} $\pi X_1 = X_2 \pi$,
\item\label{T_1_pi2} $T_1 \pi^2 = \pi^2 T_{m-1}$,
\item\label{pi_X_m} $\pi X_m = t^{-1}X_1 \pi$.
\end{enumerate}

Relation~\eqref{pi_T_i} is clear from the diagrams. For all $W\in \CF_v$ and $f\in\Hom(W,U\otimes W)$,

\begin{center}
\begin{tikzpicture}[scale=0.8]
\node[scale=1.2] at (0,3.5) {$\pi T_i \cdot f$};
\node[scale=1.2] at (8,3.5) {$T_{i+1}\pi \cdot f$};
\node at (0,0){
    \begin{tikzpicture}[scale=0.8]
        \coordinate (w) at (1.5,5);
        \coordinate (z) at (1.5,8);
        \coordinate (y) at (6.66,9.5);
        \foreach \x in {0,...,11} \coordinate (\x) at (2+\x*.5,6);
        \foreach \x in {0,...,11} \coordinate (\x\x) at (2+\x*.5,5);
        \foreach \x in {0,...,11} \coordinate (\x\x\x) at (2+\x*.5,4);
        \foreach \x in {0,...,11} \coordinate (\x\x\x\x) at (2+\x*.5,3);
        \draw[line width=5pt, white](1) to [out=-90,in=90] (1111);
        \draw[blue, line width=2pt](1) to [out=-90,in=90] (1111);
        \draw[line width=5pt, white](2) to [out=-90,in=90] (22) to [out=-90, in=-90] (w) to [out=90,in=-90] (z) to [out=90, in=-90] (y);
        \draw[line width=2pt](2) to [out=-90,in=90] (22) to [out=-90, in=-90] (w) to [out=90,in=-90] (z) to [out=90, in=-90] (y) to [out=90,in=-90] (6.66,10);
        \draw[line width=5pt, white](3) to [out=-90,in=90] (33);
        \draw[line width=2pt](3) to [out=-90,in=90] (33);
        \draw[line width=5pt, white](33) to [out=-90,in=90] (222) to [out=-90,in=90](2222);
        \draw[line width=2pt](33) to [out=-90,in=90] (222) to [out=-90,in=90](2222);
        \draw[line width=5pt, white](5) to [out=-90,in=90] (66);
        \draw[line width=2pt](5) to [out=-90,in=90] (66);
        \draw[line width=5pt, white](66) to [out=-90,in=90] (555) to [out=-90,in=90](5555);
        \draw[line width=2pt](66) to [out=-90,in=90] (555) to [out=-90,in=90](5555);
        \draw[line width=5pt, white](6) to [out=-90,in=90] (55);
        \draw[line width=2pt](6) to [out=-90,in=90] (55);
        \draw[line width=5pt, white](55) to [out=-90,in=90] (444) to [out=-90,in=90](4444);
        \draw[line width=2pt](55) to [out=-90,in=90] (444) to [out=-90,in=90](4444);
        \draw[line width=5pt, white](8) to [out=-90,in=90] (88);
        \draw[line width=2pt](8) to [out=-90,in=90] (88);
        \draw[line width=5pt, white](88) to [out=-90,in=90] (777) to [out=-90,in=90](7777);
        \draw[line width=2pt](88) to [out=-90,in=90] (777) to [out=-90,in=90](7777);
        \draw[line width=5pt, white](10) to [out=-90,in=90] (10101010);
        \draw[line width=2pt,purple](10) to [out=-90,in=90] (10101010);
        \draw[line width=2pt, purple](7,8) to [out=90,in=-90] (7,10);
        \draw[line width=2pt] (8888) to [out=90, in=180] (6.33,4) to [out=0,in=90] (6.66,3);
        \draw[ultra thick, fill=lightgray] (2,6) rectangle (7.5,8);
        \node[scale=2] at (4.5,7){$f$};
        \node[below] at (8888){$1$};
        \node[below] at (7777){$2$};
        \node[below] at (5555){$i+1$};
        \node[below] at (2222){$m$};
        \node[above] at (8){$1$};
        \node[above] at (6){$i$};
        \node[above] at (2){$m$};
        \node at (4,5.5){$\dots$};
        \node at (5.5,5.5){$\dots$};
        \node at (3.5,3.5){$\dots$};
        \node at (5,3.5){$\dots$};
        \node[blue, below] at (2,3) {$ D^{-a}L^b$};
        \node[purple] at (7.5,9) {$ W$};
        \node[purple] at (7.5,5) {$ W$};
    \end{tikzpicture}
};
\node[scale=2] at (4,0){$=$};
\node at (8,0){
    \begin{tikzpicture}[scale=0.8]
        \coordinate (w) at (1.5,5);
        \coordinate (z) at (1.5,8);
        \coordinate (y) at (6.66,9.5);
        \foreach \x in {0,...,11} \coordinate (\x) at (2+\x*.5,6);
        \foreach \x in {0,...,11} \coordinate (\x\x) at (2+\x*.5,5);
        \foreach \x in {0,...,11} \coordinate (\x\x\x) at (2+\x*.5,4);
        \foreach \x in {0,...,11} \coordinate (\x\x\x\x) at (2+\x*.5,3);
        \draw[line width=5pt, white](1) to [out=-90,in=90] (1111);
        \draw[blue, line width=2pt](1) to [out=-90,in=90] (1111);
        \draw[line width=5pt, white](2) to [out=-90,in=90] (22) to [out=-90, in=-90] (w) to [out=90,in=-90] (z) to [out=90, in=-90] (y);
        \draw[line width=2pt](2) to [out=-90,in=90] (22) to [out=-90, in=-90] (w) to [out=90,in=-90] (z) to [out=90, in=-90] (y) to [out=90,in=-90] (6.66,10);
        \draw[line width=5pt, white](3) to [out=-90,in=90] (33);
        \draw[line width=2pt](3) to [out=-90,in=90] (33);
        \draw[line width=5pt, white](33) to [out=-90,in=90] (222) to [out=-90,in=90](2222);
        \draw[line width=2pt](33) to [out=-90,in=90] (222) to [out=-90,in=90](2222);
        \draw[line width=5pt, white](5) to [out=-90,in=90] (55);
        \draw[line width=2pt](5) to [out=-90,in=90] (55);
        \draw[line width=5pt, white](66) to [out=-90,in=90] (555) to [out=-90,in=90](4444);
        \draw[line width=2pt](66) to [out=-90,in=90] (555) to [out=-90,in=90](4444);
        \draw[line width=5pt, white](6) to [out=-90,in=90] (66);
        \draw[line width=2pt](6) to [out=-90,in=90] (66);
        \draw[line width=5pt, white](55) to [out=-90,in=90] (444) to [out=-90,in=90](5555);
        \draw[line width=2pt](55) to [out=-90,in=90] (444) to [out=-90,in=90](5555);
        \draw[line width=5pt, white](8) to [out=-90,in=90] (88);
        \draw[line width=2pt](8) to [out=-90,in=90] (88);
        \draw[line width=5pt, white](88) to [out=-90,in=90] (777) to [out=-90,in=90](7777);
        \draw[line width=2pt](88) to [out=-90,in=90] (777) to [out=-90,in=90](7777);
        \draw[line width=5pt, white](10) to [out=-90,in=90] (10101010);
        \draw[line width=2pt,purple](10) to [out=-90,in=90] (10101010);
        \draw[line width=2pt, purple](7,8) to [out=90,in=-90] (7,10);
        \draw[line width=2pt] (8888) to [out=90, in=-90] (888) to [out=90, in=180] (6.33,4.5) to [out=0,in=90] (6.66,4) to [out=-90,in=90] (6.66,3);
        \draw[ultra thick, fill=lightgray] (2,6) rectangle (7.5,8);
        \node[scale=2] at (4.5,7){$f$};
        \node[below] at (8888){$1$};
        \node[below] at (7777){$2$};
        \node[below] at (5555){$i+1$};
        \node[below] at (2222){$m$};
        \node[above] at (8){$1$};
        \node[above] at (6){$i$};
        \node[above] at (2){$m$};
        \node at (4,5.5){$\dots$};
        \node at (5.5,5.5){$\dots$};
        \node at (3.5,3.5){$\dots$};
        \node at (5,3.5){$\dots$};
        \node[blue, below] at (2,3) {$ D^{-a}L^b$};
        \node[purple] at (7.5,9) {$ W$};
        \node[purple] at (7.5,5) {$ W$};
    \end{tikzpicture}
};
\end{tikzpicture}
\end{center}

Relation~\eqref{pi_X_1} is also clearly satisfied.

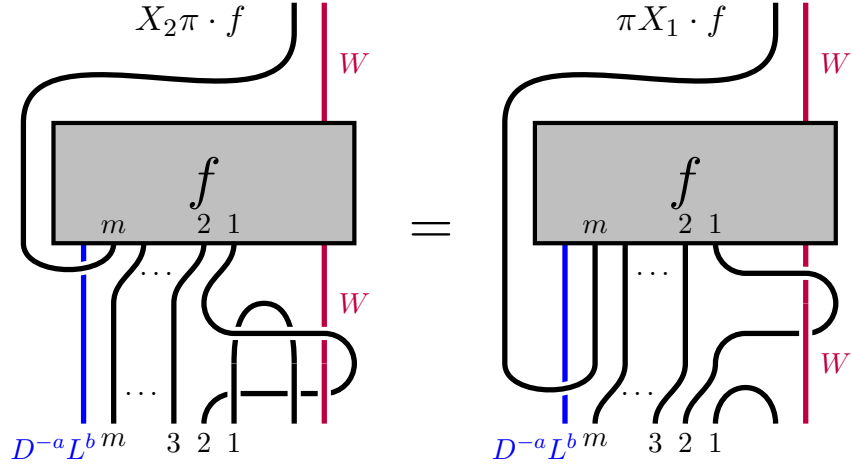
\begin{center}
\begin{tikzpicture}[scale=0.8]
\node[scale=1.2] at (0,3.5) {$X_2\pi \cdot f$};
\node[scale=1.2] at (8,3.5) {$\pi X_1 \cdot f$};
\node at (0,0){
    \begin{tikzpicture}[scale=0.8]
        \coordinate (w) at (1,6);
        \coordinate (z) at (1,8);
        \coordinate (y) at (5.5,9.5);
        \foreach \x in {0,...,11} \coordinate (\x) at (2+\x*.5,6);
        \foreach \x in {0,...,11} \coordinate (\x\x) at (2+\x*.5,5);
        \foreach \x in {0,...,11} \coordinate (\x\x\x) at (2+\x*.5,4);
        \foreach \x in {0,...,11} \coordinate (\x\x\x\x) at (2+\x*.5,3);
        \draw[line width=5pt, white](0) to [out=-90,in=90] (0000);
        \draw[blue, line width=2pt](0) to [out=-90,in=90] (0000);
        \draw[line width=5pt, white](1) to [out=-90, in=-90] (w) to [out=90,in=-90] (z) to [out=90, in=-90] (y) to [out=90,in=-90] (5.5,10);
        \draw[line width=2pt](1) to [out=-90, in=-90] (w) to [out=90,in=-90] (z) to [out=90, in=-90] (y) to [out=90,in=-90] (5.5,10);
        \draw[line width=2pt](2) to [out=-90,in=90] (11)--(1111);
        \draw[line width=2pt](4) to [out=-90,in=90] (33)--(3333);
        \draw[line width=2pt](555) to [out=90,in=180] (66) to [out=0,in=90] (777);
        \draw[line width=2pt,purple](8) to [out=-90,in=90] (888);
        \draw[line width=5pt, white](5) to [out=-90,in=90] (44) to [out=-90,in=180] (4.5,4.5) to [out=0,in=180] (6,4.5) to [out=0,in=90] (6.5,4) to [out=-90,in=0] (6,3.5);
        \draw[line width=2pt](5) to [out=-90,in=90] (44) to [out=-90,in=180] (4.5,4.5) to [out=0,in=180] (6,4.5) to [out=0,in=90] (6.5,4) to [out=-90,in=0] (6,3.5) -- (4.5,3.5) to[out=180,in=90] (4444);
        \draw[line width=2pt, purple](6,8) to [out=90,in=-90] (6,10);
        \draw[line width=5pt,white](888) -- (8888);
        \draw[line width=2pt,purple](888) -- (8888);
        \draw[line width=5pt,white](777) -- (7777);
        \draw[line width=2pt](777) -- (7777);
        \draw[line width=5pt,white](555) -- (5555);
        \draw[line width=2pt](555) -- (5555);
        \draw[ultra thick, fill=lightgray] (1.5,6) rectangle (6.5,8);
        \node[scale=2] at (4,7){$f$};
        \node[below] at (5555){$1$};
        \node[below] at (4444){$2$};
        \node[below] at (3333){$3$};
        \node[below] at (1111){$m$};
        \node[above] at (5){$1$};
        \node[above] at (4){$2$};
        \node[above] at (1){$m$};
        \node at (3.25,5.5){$\dots$};
        \node at (3,3.5){$\dots$};
        \node[blue, below] at (1.5,3) {$ D^{-a}L^b$};
        \node[purple] at (6.5,9) {$ W$};
        \node[purple] at (6.5,5) {$ W$};
    \end{tikzpicture}
};
\node[scale=2] at (4,0){$=$};
\node at (8,0){
        
    \begin{tikzpicture}[scale=0.8]
        \coordinate (w) at (1,4);
        \coordinate (z) at (1,8);
        \coordinate (y) at (5.5,9.5);
        \foreach \x in {0,...,11} \coordinate (\x) at (2+\x*.5,6);
        \foreach \x in {0,...,11} \coordinate (\x\x) at (2+\x*.5,5);
        \foreach \x in {0,...,11} \coordinate (\x\x\x) at (2+\x*.5,4);
        \foreach \x in {0,...,11} \coordinate (\x\x\x\x) at (2+\x*.5,3);
        \draw[line width=5pt, white](0) to [out=-90,in=90] (0000);
        \draw[blue, line width=2pt](0) to [out=-90,in=90] (0000);
        \draw[line width=5pt, white](1) --(111) to [out=-90, in=-90] (w) to [out=90,in=-90] (z) to [out=90, in=-90] (y) to [out=90,in=-90] (5.5,10);
        \draw[line width=2pt](1) --(111) to [out=-90, in=-90] (w) to [out=90,in=-90] (z) to [out=90, in=-90] (y) to [out=90,in=-90] (5.5,10);
        \draw[line width=2pt](2) -- (222) to [out=-90,in=90] (1111);
        \draw[line width=2pt](4) -- (444) to [out=-90,in=90] (3333);
        \draw[line width=2pt](5555) to [out=90,in=180] (5,3.6) to [out=0,in=90] (7777);
        \draw[line width=2pt,purple](8) to [out=-90,in=90] (88);
        \draw[line width=5pt, white](5) to [out=-90,in=180] (5,5.5) to [out=0,in=180] (6,5.5) to [out=0,in=90] (6.5,5);
        \draw[line width=2pt](5) to [out=-90,in=180] (5,5.5) to [out=0,in=180] (6,5.5) to [out=0,in=90] (6.5,5) to [out=-90,in=0] (6,4.5) to [out=-90,in=0] (6,4.5) -- (5,4.5) to [out=180,in=90] (555) to [out=-90,in=90] (4444);
        \draw[line width=2pt, purple](6,8) to [out=90,in=-90] (6,10);
        \draw[line width=5pt,white](88) -- (8888);
        \draw[line width=2pt,purple](88) -- (8888);
        \draw[ultra thick, fill=lightgray] (1.5,6) rectangle (6.5,8);
        \node[scale=2] at (4,7){$f$};
        \node[below] at (5555){$1$};
        \node[below] at (4444){$2$};
        \node[below] at (3333){$3$};
        \node[below] at (1111){$m$};
        \node[above] at (5){$1$};
        \node[above] at (4){$2$};
        \node[above] at (1){$m$};
        \node at (3.5,5.5){$\dots$};
        \node at (3.25,3.5){$\dots$};
        \node[blue, below] at (1.5,3) {$ D^{-a}L^b$};
        \node[purple] at (6.5,9) {$ W$};
        \node[purple] at (6.5,4) {$ W$};
    \end{tikzpicture}
};
\end{tikzpicture}
\end{center}

The picture for Relation~\eqref{T_1_pi2} is the following.

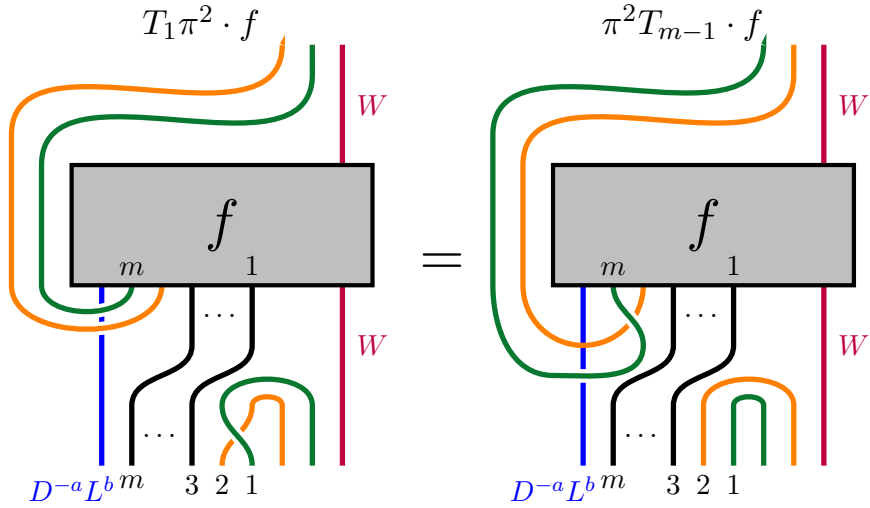
\begin{center}
  \begin{tikzpicture}[scale=0.8]
\node[scale=1.2] at (0,4) {$T_1\pi^2 \cdot f$};
\node[scale=1.2] at (8,4) {$\pi^2 T_{m-1} \cdot f$};
\node at (0,0){
    \begin{tikzpicture}[scale=0.8]
        \coordinate (w) at (1.5,6);
        \coordinate (z) at (1.5,8);
        \coordinate (y) at (6,9.5);
        \coordinate (ww) at (1,6);
        \coordinate (zz) at (1,8.5);
        \coordinate (yy) at (5.5,10);
        \foreach \x in {0,...,11} \coordinate (\x) at (2+\x*.5,6);
        \foreach \x in {0,...,11} \coordinate (\x\x) at (2+\x*.5,5);
        \foreach \x in {0,...,11} \coordinate (\x\x\x) at (2+\x*.5,4);
        \foreach \x in {0,...,11} \coordinate (\x\x\x\x) at (2+\x*.5,3);
        \draw[blue, line width=2pt](1) to [out=-90,in=90] (1111);
        \draw[line width=5pt, white](2) to [out=-90, in=-90] (w) to [out=90,in=-90] (z) to [out=90, in=-90] (y) to [out=90,in=-90] (6.166,10);
        \draw[line width=2pt,forest](2) to [out=-90, in=-90] (w) to [out=90,in=-90] (z) to [out=90, in=-90] (y) to [out=90,in=-90] (6,10);
        \draw[line width=5pt, white](3) to [out=-90, in=-90] (ww) to [out=90,in=-90] (zz) to [out=90, in=-90] (yy) to [out=90,in=-90] (5.5,10);
        \draw[line width=2pt,orange](3) to [out=-90, in=-90] (ww) to [out=90,in=-90] (zz) to [out=90, in=-90] (yy) to [out=90,in=-90] (5.5,10);
        \draw[line width=2pt](4) --(44) to [out=-90,in=90] (222)--(2222);
        \draw[line width=2pt](6) --(66) to [out=-90,in=90] (444)--(4444);
        \draw[line width=2pt,orange] (5555) to [out=90,in=-90] (666) to [out=90,in=90] (777)--(7777);
        \draw[line width=5pt,white] (6666) to [out=90,in=-90] (555) to [out=90,in=90] (888)--(8888);
        \draw[line width=2pt,forest] (6666) to [out=90,in=-90] (555) to [out=90,in=90] (888)--(8888);
        \draw[line width=2pt,purple] (9)--(9999);
        \draw[line width=2pt,purple] (6.5,8)--(6.5,10);
        \draw[ultra thick, fill=lightgray] (2,6) rectangle (7,8);
        \node[scale=2] at (4.5,7){$f$};
        \node[below] at (6666){$1$};
        \node[below] at (5555){$2$};
        \node[below] at (4444){$3$};
        \node[below] at (2222){$m$};
        \node[above] at (6){$1$};
        \node[above] at (2){$m$};
        \node at (4.5,5.5){$\dots$};
        \node at (3.5,3.5){$\dots$};
        \node[blue, below] at (2,3) {$ D^{-a}L^b$};
        \node[purple] at (7,9) {$ W$};
        \node[purple] at (7,5) {$ W$};
    \end{tikzpicture}
};
\node[scale=2] at (4,0){$=$};
\node at (8,0){
        
    \begin{tikzpicture}[scale=0.8]
        \coordinate (w) at (1.5,6);
        \coordinate (z) at (1.5,8);
        \coordinate (y) at (6,9.5);
        \coordinate (ww) at (1,6);
        \coordinate (zz) at (1,8.5);
        \coordinate (yy) at (5.5,10);
        \foreach \x in {0,...,11} \coordinate (\x) at (2+\x*.5,6);
        \foreach \x in {0,...,11} \coordinate (\x\x) at (2+\x*.5,5);
        \foreach \x in {0,...,11} \coordinate (\x\x\x) at (2+\x*.5,4);
        \foreach \x in {0,...,11} \coordinate (\x\x\x\x) at (2+\x*.5,3);
        \draw[blue, line width=2pt](1) to [out=-90,in=90] (1111);
        \draw[line width=5pt, white](2) to [out=-90,in=90] (33) to [out=-90,in=0] (2,4.5) to [out=180, in=-90] (ww) to [out=90,in=-90] (zz) to [out=90, in=-90] (yy) to [out=90,in=-90] (5.5,10);
        \draw[line width=2pt,orange](3) to [out=-90,in=0] (11) to [out=180, in=-90] (w) to [out=90,in=-90] (z) to [out=90, in=-90] (y) to [out=90,in=-90] (6,10);
        \draw[line width=5pt,white](2) to [out=-90,in=90] (33) to [out=-90,in=0] (2,4.5) to [out=180, in=-90] (ww) to [out=90,in=-90] (zz) to [out=90, in=-90] (yy) to [out=90,in=-90] (5.5,10);
        \draw[line width=2pt,forest](2) to [out=-90,in=90] (33) to [out=-90,in=0] (2,4.5) to [out=180, in=-90] (ww) to [out=90,in=-90] (zz) to [out=90, in=-90] (yy) to [out=90,in=-90] (5.5,10);

        \draw[line width=2pt](4) --(44) to [out=-90,in=90] (222)--(2222);
        \draw[line width=2pt](6) --(66) to [out=-90,in=90] (444)--(4444);
        \draw[line width=2pt,orange] (5555) to [out=90,in=-90] (555) to [out=90,in=90] (888)--(8888);
        \draw[line width=2pt,forest] (6666) to [out=90,in=-90] (666) to [out=90,in=90] (777)--(7777);
        \draw[line width=2pt,purple] (9)--(9999);
        \draw[line width=2pt,purple] (6.5,8)--(6.5,10);
        \draw[ultra thick, fill=lightgray] (2,6) rectangle (7,8);
        \node[scale=2] at (4.5,7){$f$};
        \node[below] at (6666){$1$};
        \node[below] at (5555){$2$};
        \node[below] at (4444){$3$};
        \node[below] at (2222){$m$};
        \node[above] at (6){$1$};
        \node[above] at (2){$m$};
        \node at (4.5,5.5){$\dots$};
        \node at (3.5,3.5){$\dots$};
        \node[blue, below] at (2,3) {$ D^{-a}L^b$};
        \node[purple] at (7,9) {$ W$};
        \node[purple] at (7,5) {$ W$};
    \end{tikzpicture}
};
\end{tikzpicture}
\end{center}

We can slide the crossing around the torus, or equivalently use the equivalence
\[(\Id_U\otimes \check{R}_{V,V}\otimes \Id_W)\circ(\Id_{V\otimes V}\otimes f) \sim (\Id_{V\otimes V}\otimes f) \circ (\check{R}_{V,V}\otimes \Id_W),\]
from Lemma~\ref{lem_torus_skein_module}.

For relation~\eqref{pi_X_m}, we draw the embeddings in the solid torus, as in Figure~\ref{fig_hom_torus}. In order to simplify the picture, the torus is only represented by its core, noted $*$, and the band $U$ on its boundary.
Then one has to remember that the lines we draw are actually bands, and take into account the twists which may appear in those bands.
\begin{equation}\label{draw_pi_Xm_1}
    \begin{tikzpicture}[scale=0.8]
        \coordinate (p) at (5.5,5.5);
        \coordinate (w) at (1.5,6);
        \coordinate (z) at (1.5,8);
        \coordinate (y) at (5.1339,5.5);
        \coordinate (ww) at (1,6);
        \coordinate (zz) at (1,8.5);
        \coordinate (yy) at (6,5);
        \coordinate (www) at (1,6);
        \coordinate (zzz) at (1,8);
        \coordinate (yyy) at (6.7,9.5);
        \foreach \x in {0,...,11} \coordinate (\x) at (2+\x*.5,6);
        \foreach \x in {0,...,11} \coordinate (\x\x) at (2+\x*.5,5);
        \foreach \x in {0,...,11} \coordinate (\x\x\x) at (2+\x*.5,4);
        \foreach \x in {0,...,11} \coordinate (\x\x\x\x) at (2+\x*.5,3);
        \draw[blue, line width=2pt](1) to [out=-90,in=90] (2.5,5.25);
            \draw[purple,line width=2pt] (5,8) to [out=90,in=-135] ({6-1.414*.5},{8+1.414*.5});
                \draw[line width=5pt,white](2) to [out=-90,in=180] (y) to [out=0,in=180] (p) to [out=0,in=90] (yy);
                \draw[line width=2pt](2) to [out=-90,in=180] (4,5.5)-- (y) to [out=0,in=180] (p) to [out=0,in=90] (yy) to [out=-90,in=0] (5.5,4.5)-- (3,4.5);
            \draw[purple,line width=2pt] ({6-1.414*.5},{8+1.414*.5}) to[out=45,in=180] (6,9) to[out=0,in=90](7,8) to [out=-90,in=90] (7,6) to [out=-90,in=0] (6,5) to [out=180,in=-90] (5,6);
        \draw[white, line width=5pt](2.5,5.25) --(2.5,4.5);
        \draw[blue, line width=2pt](2.5,5.25) --(2.5,3);
                \draw[white, line width=5pt](3,4.5) to [out=180,in=-90] (www);
                \draw[line width=2pt](3,4.5) to [out=180,in=-90]  (www)--(zzz) to [out=90,in=135] (yyy) to [out=-45,in=90] (7.5,7) to [out=-90,in=0](6,4) to [out=180,in=90] (6666);
        \draw[line width=2pt] (3)--(3333);
        \draw[line width=2pt] (5)--(5555);
                \draw[line width=5pt,white](2) to [out=-90,in=180] (4,5.5)-- (y) to [out=0,in=180] (p) to [out=0,in=90] (yy);
                \draw[line width=2pt](2) to [out=-90,in=180] (4,5.5)-- (y) to [out=0,in=180] (p) to [out=0,in=90] (yy);
            \draw[white,line width=5pt] ({6-1.414*.5},{8+1.414*.5}) to[out=45,in=180] (6,9) to[out=0,in=90](7,8) to [out=-90,in=90] (7,6) to [out=-90,in=0] (6,5) to [out=180,in=-45] (5.292,5.292);
            \draw[purple,line width=2pt] ({6-1.414*.5},{8+1.414*.5}) to[out=45,in=180] (6,9) to[out=0,in=90](7,8) to [out=-90,in=90] (7,6) to [out=-90,in=0] (6,5)to [out=180,in=-45] (5.292,5.292);
        \draw[line width=5pt, white] (33)--(3333);
        \draw[line width=5pt, white] (55)--(5555);
        \draw[line width=2pt] (33)--(3333);
        \draw[line width=2pt] (55)--(5555);
        \draw[ultra thick, fill=lightgray] (2,6) rectangle (5.5,8);
        \node[scale=2] at (4,7){$f$};
        \node[below] at (5,3) {$1$};
        \node[below] at (4.5,3) {$2$};
        \node[below] at (3.5,3){$m$};
        \node at (4,5){$\dots$};
        \node at (4,3.5){$\dots$};
        \node[scale=1.5] at (6,7){$*$};
        \draw[line width=2pt,forest] (2,3)--(5.5,3);
        \node[forest] at (6,3) {$U$};
        \node[purple] at (6.5,8) {$W$};
        \node[below,blue] at (2.5,3){$D^{-a}L^b$};
        \node[scale=1.5] at (4,11) {$\pi X_m\cdot f$};
        \node[scale=1.5] at (9,6) {$=$};
\begin{scope}[xshift=10cm]
        \coordinate (p) at (5.5,8.5);
        \coordinate (w) at (1.5,6);
        \coordinate (z) at (1.5,8);
        \coordinate (y) at (5.1339,8.5);
        \coordinate (ww) at (1,6);
        \coordinate (zz) at (1,8.5);
        \coordinate (yy) at (5.5,8.866);
        \coordinate (www) at (.5,6);
        \coordinate (zzz) at (.5,9);
        \coordinate (yyy) at (6.7,9.5);
        \foreach \x in {0,...,11} \coordinate (\x) at (2+\x*.5,6);
        \foreach \x in {0,...,11} \coordinate (\x\x) at (2+\x*.5,5);
        \foreach \x in {0,...,11} \coordinate (\x\x\x) at (2+\x*.5,4);
        \foreach \x in {0,...,11} \coordinate (\x\x\x\x) at (2+\x*.5,3);
        \draw[blue, line width=2pt](1) to [out=-90,in=90] (2.5,5.25);
            \draw[purple,line width=2pt] (5,8) to [out=90,in=-135] ({6-1.414*.5},{8+1.414*.5});
            \draw[line width=5pt,white](2) to [out=-90,in=-90] (w) --(z) to [out=90,in=135] (y) to [out=-45,in=-135] (p);
                \draw[line width=2pt](2) to [out=-90,in=-90] (w) --(z) to [out=90,in=135] (y) to [out=-45,in=-135] (p);
                \draw[white, line width=5pt] (p) to [out=45,in=-45] (yy) to [out=135,in=90] (zz);
                \draw[line width=2pt] (p) to [out=45,in=-45] (yy) to [out=135,in=90] (zz);
                \draw[white,line width=5pt] ({6-1.414*.5},{8+1.414*.5}) to[out=45,in=180] (6,9) to[out=0,in=90](7,8) to [out=-90,in=90] (7,6) to [out=-90,in=0] (6,5) to [out=180,in=-90] (5,6);
            \draw[purple,line width=2pt] ({6-1.414*.5},{8+1.414*.5}) to[out=45,in=180] (6,9) to[out=0,in=90](7,8) to [out=-90,in=90] (7,6) to [out=-90,in=0] (6,5) to [out=180,in=-90] (5,6);
                \draw[line width=2pt](zz) --(ww) to [out = -90,in=180] (2.5,5) to [out=0,in=90] (3,4.5) ;
        \draw[white, line width=5pt](2.5,5.25) --(2.5,4.5);
        \draw[blue, line width=2pt](2.5,5.25) --(2.5,3);
                \draw[white, line width=5pt](3,4.5) to [out=-90,in=0] (2.5,4) to [out=180,in=-90] (www);
                \draw[line width=2pt](3,4.5) to [out=-90,in=0] (2.5,4) to [out=180,in=-90] (www)--(zzz) to [out=90,in=135] (yyy) to [out=-45,in=90] (7.5,7) to [out=-90,in=0](6,4) to [out=180,in=90] (6666);
        \draw[line width=2pt] (3)--(3333);
        \draw[line width=2pt] (5)--(5555);
        \draw[ultra thick, fill=lightgray] (2,6) rectangle (5.5,8);
        \node[scale=2] at (4,7){$f$};
        \node[below] at (5,3) {$1$};
        \node[below] at (4.5,3) {$2$};
        \node[below] at (4.5,3) {$2$};
        \node[below] at (3.5,3){$m$};
        \node at (4,5.5){$\dots$};
        \node at (4,3.5){$\dots$};
        \draw[ultra thick,fill=lightgray] (-0.25,7) rectangle (.75,8);
        \node at (.25,7.5){$\theta_V^{-1}$};
        \node[scale=1.5] at (6,7){$*$};
        \draw[line width=2pt,forest] (2,3)--(5.5,3);
        \node[forest] at (6,3) {$U$};
        \node[purple] at (6.5,8) {$W$};
        \node[below,blue] at (2.5,3){$D^{-a}L^b$};
        \end{scope}
    \end{tikzpicture}
\end{equation}

During this operation, a (negative) twist has appeared in the leftmost $V$ strand, which correspond to multiplying the strand by the inverse of the ribbon element $\theta_V^{-1}$.

We assume without loss of generality that $0 \leq b < N$, then using \eqref{D_lambda} and \eqref{L_lambda}, $D^{-a}\otimes L^b = L(\lambda)$ and $V\otimes D^{-a}\otimes L^b = L(\mu)$, where 
\begin{align*}
\lambda & = (\underbrace{K-a, \ldots, K-a}_{b},-a,-a,\ldots, -a),\\
\mu & = (\underbrace{K-a, \ldots, K-a}_{b},1-a,-a,\ldots, -a).
\end{align*}
Using Corollary~\ref{cor_lambda_mu_nu}, $\check{R}_{D^{-a}\otimes L^b, V}^{-1}\check{R}_{V,D^{-a}\otimes L^b}^{-1}= (\check{R}_{V,D^{-a}\otimes L^b}\check{R}_{D^{-a}\otimes L^b, V})^{-1}$ acts on $V\otimes D^{-a}\otimes L^b$ by multiplication by the constant
\[ v^{\left\langle \varepsilon_1, \varepsilon_1 + 2\rho\right\rangle + \left\langle \lambda, \lambda + 2\rho\right\rangle - \left\langle \mu, \mu + 2\rho\right\rangle} = q^{a+b}=t^{-1}.\]
Thus 

\begin{equation}\label{eq:t relation skein}
\begin{tikzpicture}[baseline=-1cm]
\begin{knot}[clip width =4]
\strand[line width =2pt] (0,0) to[out= 270, in=90] (1,-1) to[out= 270, in=90] (0,-2);
\strand[line width =2pt,blue] (0.5,0) to (0.5,-2);
\flipcrossings{1}
\end{knot}
\node[scale=1.5] at (2,-1) {$= t^{-1}\cdot$};
\draw[line width =2pt] (3,0) to (3,-2);
\draw[line width =2pt,blue] (3.5,0) to (3.5,-2);
\node[above] at (0,0) {$V$};
\node[above] at (3,0) {$V$};
\node[xshift = 0.5cm, yshift= 0.3cm,blue] at (0.5,0) {$D^{-a} L^b$};
\node[xshift = 0.5cm, yshift= 0.3cm,blue] at (3.5,0) {$D^{-a} L^b$};
\end{tikzpicture}
\end{equation}

Hence, we can replace the left picture of \eqref{draw_pi_Xm_1} by

\begin{equation}\label{draw_pi_Xm_2}
    \begin{tikzpicture}[scale=0.8]
        \coordinate (p) at (5.5,5.5);
        \coordinate (w) at (1.5,6);
        \coordinate (z) at (1.5,8);
        \coordinate (y) at (5.1339,5.5);
        \coordinate (ww) at (1,6);
        \coordinate (zz) at (1,8.5);
        \coordinate (yy) at (6,5);
        \coordinate (www) at (.5,6);
        \coordinate (zzz) at (.5,9);
        \coordinate (yyy) at (6.7,9.5);
        \foreach \x in {0,...,11} \coordinate (\x) at (2+\x*.5,6);
        \foreach \x in {0,...,11} \coordinate (\x\x) at (2+\x*.5,5);
        \foreach \x in {0,...,11} \coordinate (\x\x\x) at (2+\x*.5,4);
        \foreach \x in {0,...,11} \coordinate (\x\x\x\x) at (2+\x*.5,3);
                    \draw[blue, line width=2pt](1) to [out=-90,in=90] (2.5,5.25);
            \draw[purple,line width=2pt] (5,8) to [out=90,in=-135] ({6-1.414*.5},{8+1.414*.5});
            \draw[purple,line width=2pt] ({6-1.414*.5},{8+1.414*.5}) to[out=45,in=180] (6,9) to[out=0,in=90](7,8) to [out=-90,in=90] (7,6) to [out=-90,in=0] (6,5) to [out=180,in=-90] (5,6);
                \draw[white,line width=5pt] (2) to [out=-90,in=-90] (ww) -- (zz) to[out=90,in=135] (yyy) to [out=-45,in=45] (6.5,8.5) to [out=-135,in=180] (7,7.5) to [out=0,in=90] (7.5,6) to [out=-90,in=0] (6,4.5) to [out=180,in=90] (5,4) -- (5,3);
                \draw[line width=2pt] (2) to [out=-90,in=-90] (ww) -- (zz) to[out=90,in=135] (yyy) to [out=-45,in=45] (6.5,8.5) to [out=-135,in=180] (7,7.5) to [out=0,in=90] (7.5,6) to [out=-90,in=0] (6,4.5) to [out=180,in=90] (5,4) -- (5,3);
                    \draw[white, line width=5pt](2.5,5.25) --(2.5,4.5);
                    \draw[blue, line width=2pt](2.5,5.25) --(2.5,3);
                    \draw[line width=2pt] (3)--(3333);
                    \draw[line width=2pt] (5)--(5555);
            \draw[white,line width=5pt](7,8) to [out=-90,in=90] (7,6) to [out=-90,in=0] (6,5) to [out=180,in=-45] (5.292,5.292);
            \draw[purple,line width=2pt](7,8) to [out=-90,in=90] (7,6) to [out=-90,in=0] (6,5)to [out=180,in=-45] (5.292,5.292);
                    \draw[line width=5pt, white] (33)--(3333);
                    \draw[line width=5pt, white] (55)--(5555);
                    \draw[line width=2pt] (33)--(3333);
                    \draw[line width=2pt] (55)--(5555);
        \draw[ultra thick, fill=lightgray] (2,6) rectangle (5.5,8);
        \node[scale=2] at (4,7){$f$};
        \node[below] at (5,3) {$1$};
        \node[below] at (4.5,3) {$2$};
        \node[below] at (3.5,3){$m$};
        \node at (4,5.5){$\dots$};
        \node at (4,3.5){$\dots$};
        \draw[ultra thick,fill=lightgray] (0.5,7) rectangle (1.5,8);
        \node at (1,7.5){$\theta_V^{-1}$};
        \node[scale=1.5] at (6,7){$*$};
        \draw[line width=2pt,forest] (2,3)--(5.5,3);
        \node[scale=1.5] at (0,6.5){$t^{-1}\cdot$};
        \node[forest] at (6,3) {$U$};
        \node[purple] at (6,8.5) {$W$};
        \node[below,blue] at (2.5,3){$D^{-a}L^b$};
        \node[scale=1.5] at (8.3,6) {$=$};
\begin{scope}[xshift=9.5cm]
        \coordinate (p) at (5.5,5.5);
        \coordinate (w) at (1.5,6);
        \coordinate (z) at (1.5,8);
        \coordinate (y) at (5.1339,5.5);
        \coordinate (ww) at (1,6);
        \coordinate (zz) at (1,8.5);
        \coordinate (yy) at (5.5,5);
        \coordinate (www) at (.5,6);
        \coordinate (zzz) at (.5,9);
        \coordinate (yyy) at (6.7,9.5);
        \foreach \x in {0,...,11} \coordinate (\x) at (2+\x*.5,6);
        \foreach \x in {0,...,11} \coordinate (\x\x) at (2+\x*.5,5);
        \foreach \x in {0,...,11} \coordinate (\x\x\x) at (2+\x*.5,4);
        \foreach \x in {0,...,11} \coordinate (\x\x\x\x) at (2+\x*.5,3);
                    \draw[blue, line width=2pt](1) to [out=-90,in=90] (2.5,5.25);
            \draw[purple,line width=2pt] (5,8) to [out=90,in=-135] ({6-1.414*.5},{8+1.414*.5});
            \draw[purple,line width=2pt] ({6-1.414*.5},{8+1.414*.5}) to[out=45,in=180] (6,9) to[out=0,in=90](7,8) to [out=-90,in=90] (7,6) to [out=-90,in=0] (6,5) to [out=180,in=-90] (5,6);
                \draw[line width=5pt,white] (2) to [out=-90,in=-90] (ww) -- (zz) to[out=90,in=135] (yyy) to [out=-45,in=90] (7.5,6) to [out=-90,in=0] (6,4.5) to [out=180,in=-45] (5,5) to [out=135,in=90] (4.5,4.5) to [out=-90,in=-90] (5.5,4.5) to [out=90,in=-90] (yy) to [out=90,in=180] (6,5.5) to [out=0,in=90] (6.5,5) to [out=-90,in=90] (6.5,4) to [out=-90,in=90] (5,3.5)--(5,3);
                \draw[line width=2pt] (2) to [out=-90,in=-90] (ww) -- (zz) to[out=90,in=135] (yyy) to [out=-45,in=90] (7.5,6) to [out=-90,in=0] (6,4.5) to [out=180,in=-45] (5,5) to [out=135,in=90] (4.5,4.5) to [out=-90,in=-90] (5.5,4.5) to [out=90,in=-90] (yy) to [out=90,in=180] (6,5.5) to [out=0,in=90] (6.5,5) to [out=-90,in=90] (6.5,4) to [out=-90,in=90] (5,3.5)--(5,3);
                    \draw[white, line width=5pt](2.5,5.25) --(2.5,4.5);
                    \draw[blue, line width=2pt](2.5,5.25) --(2.5,3);
                    \draw[line width=2pt] (3) to [out=-90,in=90] (22)--(2222);
                    \draw[line width=2pt] (5)to [out=-90,in=90] (44)--(4444);
            \draw[white,line width=5pt](7,8) to [out=-90,in=90] (7,6) to [out=-90,in=0] (6,5) to [out=180,in=-45] (5.292,5.292);
            \draw[purple,line width=2pt](7,8) to [out=-90,in=90] (7,6) to [out=-90,in=0] (6,5)to [out=180,in=-45] (5.292,5.292);
                \draw[line width=5pt,white] (7.5,6) to [out=-90,in=0] (6,4.5);
                \draw[line width=2pt] (7.5,6) to [out=-90,in=0] (6,4.5);
                
                \draw[line width=5pt,white] (4.5,4.5) to [out=-90,in=-90] (5.5,4.5) to [out=90,in=-90] (yy) to [out=90,in=180] (6,5.5);
                \draw[line width=2pt] (4.5,4.5) to [out=-90,in=-90] (5.5,4.5) to [out=90,in=-90] (yy) to [out=90,in=180] (6,5.5);
        \draw[ultra thick, fill=lightgray] (2,6) rectangle (5.5,8);
        \node[scale=2] at (4,7){$f$};
        \node[below] at (5,3) {$1$};
        \node[below] at (4,3) {$2$};
        \node at (3.7,5.5){$\dots$};
        \node at (3.5,3.5){$\dots$};
        \node[scale=1.5] at (6,7){$*$};
        \draw[line width=2pt,forest] (2,3)--(5.5,3);
        \node[scale=1.5] at (0,6.5){$t^{-1}\cdot$};
        \node[forest] at (6,3) {$U$};
        \node[purple] at (6.5,8) {$W$};
        \node[below,blue] at (2.5,3){$D^{-a}L^b$};
\end{scope}
    \end{tikzpicture}
\end{equation}
In the picture on the right of \eqref{draw_pi_Xm_2}, the loop on the bottom introduces a new twist in the strand and it simplifies with the twist which appeared in \eqref{draw_pi_Xm_1}. In the picture on the right, we recognize $X_1 \pi\cdot f$, which concludes the proof.

\end{proof}

\subsection{Application: proof of the Morton-Samuelson conjecture}

In \cite{morton2021dahas} the algebra $\ddot{H}_{q,t}(m)$ was identified with the group algebra of braids on the torus $T^2\setminus\{*\}$ modulo certain relations. This algebra is then mapped to another algebra denoted by $\Sk_n(T^2,*)$ (Theorem 4.1 in op.cit.). By definition, $\Sk_n(T^2,*)$ is an algebra over $\BC[s^{\pm 1}, c^{\pm 1}, v^{\pm 1}]$ consists of formal linear combinations of ribbon graphs in $T^2\times I$ avoiding the so-called \emph{base string} $*\times I$ modulo isotopy, skein relations and a relation that allows a band going on one side of the base string to be replaced by a band going on the other side multiplied by $c^2$ (see Section 4 in \cite{morton2021dahas}). The bands have $m$ inputs on $T^2\times \{0\}$ and $m$ outputs on $T^2\times \{1\}$. Then the following is shown:
\begin{thm}[Theorem 4.2, \cite{morton2021dahas}]\label{thm:ms} The homomorphism 
\[
\ddot{H}_{q,t} \otimes \BC[v^{\pm 1}] \otimes_{\BC[q^{\pm 1},t^{\pm 1}]} \BC(s,c) \to \Sk_n(T^2,*)\otimes_{\BC[s^{\pm 1},c^{\pm 1}]} \BC(s,c)
\]
is surjective\footnote{In \cite{morton2021dahas} the theorem is formulated without the localization $\otimes_{\BC[s^{\pm 1},c^{\pm 1}]} \BC(s,c)$, but localization is clearly necessary, for instance an unknot is equivalent to $\frac{v^{-1}-v}{s-s^{-1}}$, which has denominator. Another point is that tensoring with $\BC[v^\pm]$ is clearly assumed but not explicitly stated in the original formulation.}.
\end{thm}

It is conjectured in op.cit. that the homomorphism above is an isomorphism. Denote 
\[
\ddot{H}_{s,c,v}=\ddot{H}_{q,t} \otimes_{\BC[q^{\pm 1},t^{\pm 1}]} \BC[s^{\pm 1},c^{\pm 1}] \otimes\BC[v^{\pm 1}]
\]
where the variables $q,t$ are related to $s,c$ as follows:
\[
q=s^2,\qquad t=c^2.
\]
Let $\varphi$ be the homomorphism before taking the localization:
\[
\varphi: \ddot{H}_{s,c,v} \to \Sk_n(T^2,*).
\]
\begin{thm}
The homomorphism $\varphi$ is injective.
\end{thm}
\begin{proof}
Let us call a representation $W$ of $\ddot{H}_{s,c,v}$ \emph{ribbon} if the homomorphism $\ddot{H}_{s,c,v}\to \End(W)$ factors through $\varphi$. Suppose we have a ribbon representation $W$ which is faithful. Then the composition
\[
\ddot{H}_{q,t} \otimes \BC[v^\pm] \xrightarrow{\varphi} \Sk_n(T^2,*) \to \End(W)
\]
is injective, which implies injectivity of $\varphi$, so we are done.

For each $K,N,a,b$ the representation $W_{(K,N,a,b)}$ after the base change to $\BC[s^{\pm1},c^{\pm1}]$ is isomorphic to the ribbon module $S_U$ (see \eqref{eq:S_U}), which is naturally a module for the skein algebra $\Sk_n(T^2,*)$: the action is obtained by placing $T\times I$ on the boundary of the solid torus and connecting bands in the solid torus which represent elements of $S_U$ with bands in $T\times I$ representing elements of $\Sk_n(T^2,*)$. The base string is labeled by the representation $D^{-a} L^b$ and therefore satisfies the local relation, see \eqref{eq:t relation skein}. Thus we see that $W_{(K,N,a,b)}$ is ribbon and therefore the direct sum $\CW=\bigoplus_{aN-bK=m} W_{(K,N,a,b)}$ is ribbon.

By Proposition \ref{prop:faithful}, $\CW$ is a faithful representation of $\ddot{H}_{q,t}$. We need to show that it remains a faithful representation when passing to the larger algebra $\ddot{H}_{s,c,v}$. Note that the action of $v$ is given by $s^{-N}$. Similarly to the proof of Proposition \ref{prop:faithful}, it is enough to show that the set of triples $(s,c,v)$ is Zariski dense in $\BC^3$ where $c=s^{m(u+v)}$, $v=s^{-N}$, $s^2=q$, $q$ is a primitive $m(N+\alpha)$-th root of unity and $u,N,v,\alpha$ range over the set of solutions of $uN-v\alpha=1$, $N$ and $\alpha$ sufficiently large. It is sufficient to prove that the image of this set under the map 
\[
(s,c,u)\to (s^{2m},c^2,u^{-2m})=(q^m,q^{m(u+v)},q^{mN})
\]
is dense, so replacing $q^m$ by $q$ we reduce the claim to the following statement: The set of triples $(q,q^{u+v},q^N)$ is Zariski dense in $\BC^3$ where $q$ runs over the primitive $(N+\alpha)$-th roots of unity and $u,v,N,\alpha$ are as above.

We have $(u+v)N-v(\alpha+N)=1$. In particular, $u+v$ and $\alpha+N$ are relatively prime and we can replace $q$ by $q^{u+v}$. We then have
\[
(q^{u+v}, q^{(u+v)^2},q^{N(u+v)})=(q^{u+v}, q^{(u+v)^2},q).
\]
As in the proof of Proposition \ref{prop:faithful} such triples for fixed $u,v$ are Zariski dense on the curve $\{(z^{u+v}, z^{(u+v)^2},z)\,|\,z\in\BC\}$. So it remains to show that the union of curves
\[
C_r = \{(z,z^r,z^{r^2})\,|\,z\in\BC\}
\]
where $r$ ranges over positive integers are Zariski dense in $\BC^3$. Suppose this is not the case, i.e. there exists a polynomial
\[
P(x,y,z)=\sum_{i,j,k<d} c_{ijk} x^i y^j z^k\qquad(c_{ijk}\in \BC)
\]
vanishing on all these curves. Take $r$ larger than $d$. Then $i+rj+r^2 k$ are all distinct for distinct triples $(i,j,k)$ satisfying $i,j,k<d$ and so $P(z,z^r,z^{r^2})$ cannot vanish. Contradiction.
\end{proof}

\begin{cor}
Morton-Samuelson's conjecture is true, i.e. the homomorphism in Theorem \ref{thm:ms} is an isomorphism.
\end{cor}
\begin{proof}
Since $\varphi$ is injective and localization is flat, the homomorphism in Theorem \ref{thm:ms} is injective.
\end{proof}

\bibliographystyle{plain}
\bibliography{references}

\end{document}